\documentclass[12pt,a4paper]{amsart}
\usepackage{mathtools,amsthm,amssymb}
\usepackage{color}
\usepackage[top=30truemm,bottom=30truemm,left=25truemm,right=25truemm]{geometry}
\usepackage{tikz}
\usetikzlibrary{positioning}
\usepackage{mathdots}
\usepackage{ytableau}
\usepackage{hyperref}

\newtheorem{thm}{Theorem}[subsection]
\newtheorem{lem}[thm]{Lemma}
\newtheorem{prop}[thm]{Proposition}
\newtheorem{cor}[thm]{Corollary}
\newtheorem{assum}[thm]{Assumption}

\theoremstyle{definition}
\newtheorem{defi}[thm]{Definition}
\newtheorem{rem}[thm]{Remark}

\newtheorem{ex}[thm]{Example}

\author{Hideya Watanabe}
\address{(H. Watanabe) College of Science, Rikkyo University, 3-34-1, Nishi-Ikebukuro, Toshima-ku, Tokyo, 171-8501, Japan}
\email{watanabehideya@gmail.com}
\date{\today}

\title[KN tableaux, GT patterns, and quantum symmetric pairs]{Kashiwara--Nakashima tableaux, Gelfand--Tsetlin patterns, and quantum symmetric pairs}

\begin{document}
\maketitle

\begin{abstract}
  Kashiwara--Nakashima tableaux and Gelfand--Tsetlin patterns of orthogonal type are famous as combinatorial models of the finite-dimensional irreducible representations of the special orthogonal Lie algebras \(\mathfrak{so}_N\).
  The former is constructed based on the representation theory of quantum groups, especially the theory of crystals, while the latter based on the branching rule for \((\mathfrak{so}_N,\mathfrak{so}_{N-1})\).
  In the present paper, we construct a natural bijection between them by means of the representation theory of quantum symmetric pairs corresponding to \((\mathfrak{so}_N,\mathfrak{so}_{N-1})\).
\end{abstract}

\section{Introduction}
\subsection{Gelfand--Tsetlin bases and patterns}
Let \(\mathfrak{gl}_N\) (\(N \in \mathbb{Z}_{\geq 1}\)) denote the \emph{general linear Lie algebra} over the field \(\mathbb{C}\) of complex numbers.
Every finite-dimensional representation of \(\mathfrak{gl}_N\) is completely reducible, and the finite-dimensional irreducible representations are classified by their \emph{highest weights}.
The set \(X^+_{\mathfrak{gl}_N}\) of highest weights (or dominant weights) consists of \(N\)-tuples \((\lambda_1,\dots,\lambda_N) \in \mathbb{C}^N\) such that
\[
  \lambda_i-\lambda_{i+1} \in \mathbb{Z}_{\geq 0} \quad \text{ for all } i \in \{ 1,\dots,N-1 \}.
\]
Let \(V^{\mathfrak{gl}_N}(\lambda)\) denote the simple \(\mathfrak{gl}_N\)-module corresponding to \(\lambda \in X^+_{\mathfrak{gl}_N}\).

As a set, \(\mathfrak{gl}_N\) is identical to the set of \(N \times N\) matrices.
Hence, there exists a natural embedding
\[
  \mathfrak{gl}_N \mapsto \mathfrak{gl}_{N+1};\ A \mapsto
  \begin{pmatrix}
    A & \mathbf{o} \\
    {}^{\mathsf{t}} \mathbf{o} & 0
  \end{pmatrix}
\]
of Lie algebras, where \({}^{\mathsf{t}} \cdot\) denotes the transpose.
Every \(\mathfrak{gl}_{N+1}\)-module can be regarded as a \(\mathfrak{gl}_N\)-module via this embedding.
Especially, the \(\mathfrak{gl}_N\)-module \(V^{\mathfrak{gl}_{N+1}}(\lambda)\) is known to be semisimple with multiplicity-free, and the simple component \(V^{\mathfrak{gl}_N}(\mu)\) (\(\mu \in X^+_{\mathfrak{gl}_N}\)) occurs if and only if the highest weights \(\lambda,\mu\) satisfy the following \emph{interlacing condition}:
\begin{itemize}
  \item \(\lambda_1 \geq \mu_1 \geq \lambda_2 \geq \mu_2 \geq \cdots \geq \lambda_N \geq \mu_N \geq \lambda_{N+1}\),
  \item \(\lambda_i-\mu_i, \mu_i-\lambda_{i+1} \in \mathbb{Z}_{\geq 0}\) for all \(i \in \{ 1,\dots,N \}\).
\end{itemize}

Based on this branching rule, Gelfand and Tsetlin \cite{GeTs50} constructed a basis for each \(V^{\mathfrak{gl}_N}(\lambda)\), called the \emph{Gelfand--Tsetlin basis}.
Each Gelfand--Tsetlin basis is parametrized by a sequence of highest weights whose successive pairs satisfy the interlacing condition.
Such a sequence is called a \emph{Gelfand--Tsetlin pattern} of shape \(\lambda\).
Let \(\mathsf{GTP}^{\mathfrak{gl}_N}(\lambda)\) denote the set of Gelfand--Tsetlin patterns of shape \(\lambda\).

The story above has an \(\mathfrak{so}_N\)-analogue.
Let \(\mathfrak{so}_N\) (\(N \geq 2\)) denote the \emph{special orthogonal Lie algebra}.
It is often defined to be the Lie subalgebra of \(\mathfrak{gl}_N\) consisting of matrices \(A\) such that \({}^{\mathsf{at}} A = -A\), where
\[
  {}^{\mathsf{at}} (a_{i,j})_{1 \leq i,j \leq N} := (a_{N-j+1,N-i+1})_{1 \leq i,j \leq N}
\]
denotes the \emph{anti-transpose}, i.e., transpose with respect to the anti-diagonal.

Set \(n := \lfloor \frac{N}{2} \rfloor\).
The set \(X^+_{\mathfrak{so}_N}\) of highest weights are described as
\[
  \begin{cases}
    \{ (\lambda_1,\dots,\lambda_n) \in \mathbb{Z}^n \sqcup (\frac{1}{2} + \mathbb{Z})^n \mid \lambda_1 \geq \cdots \geq \lambda_{n-1} \geq |\lambda_n| \} & \text{ if } N = 2n, \\
    \{ (\lambda_1,\dots,\lambda_n) \in \mathbb{Z}^n \sqcup (\frac{1}{2} + \mathbb{Z})^n \mid \lambda_1 \geq \cdots \geq \lambda_n \geq 0 \} & \text{ if } N = 2n+1.
  \end{cases}
\]
Let \(V^{\mathfrak{so}_N}(\lambda)\) denote the simple \(\mathfrak{so}_N\)-module corresponding to \(\lambda \in X^+_{\mathfrak{so}_N}\).

There exists an embedding
\[
  \phi_N \colon \mathfrak{so}_N \to \mathfrak{so}_{N+1}
\]
of Lie algebras defined by
\begin{align*}
  &\phi_{2n}
  \begin{pmatrix}
    A & B \\
    C & -{}^{\mathsf{at}}A
  \end{pmatrix}
  =
  \begin{pmatrix}
    A & \mathbf{o} & B \\
    {}^{\mathsf{at}}\mathbf{o} & 0 & {}^{\mathsf{at}}\mathbf{o} \\
    C & \mathbf{o} & -{}^{\mathsf{at}}A
  \end{pmatrix}
  , \\
  &\phi_{2n+1}
  \begin{pmatrix}
    A & \mathbf{u} & B \\
    {}^{\mathsf{at}}\mathbf{v} & 0 & -{}^{\mathsf{at}}\mathbf{u} \\
    C & -\mathbf{v} & -{}^{\mathsf{at}}A
  \end{pmatrix}
  =
  \begin{pmatrix}
    A & \frac{\mathbf{u}}{\sqrt{2}} & \frac{\mathbf{u}}{\sqrt{2}} & B \\
    \frac{{}^{\mathsf{at}}\mathbf{v}}{\sqrt{2}} & 0 & 0 & -\frac{{}^{\mathsf{at}}\mathbf{u}}{\sqrt{2}} \\
    \frac{{}^{\mathsf{at}}\mathbf{v}}{\sqrt{2}} & 0 & 0 & -\frac{{}^{\mathsf{at}}\mathbf{u}}{\sqrt{2}} \\
    C & -\frac{\mathbf{v}}{\sqrt{2}} & -\frac{\mathbf{v}}{\sqrt{2}} & -{}^{\mathsf{at}}A
  \end{pmatrix}.
\end{align*}

The \(\mathfrak{so}_N\)-module \(V^{\mathfrak{so}_{N+1}}(\lambda)\) defined by the homomorphism \(\phi_N\) is known to be semisimple with multiplicity-free, and the simple factor \(V^{\mathfrak{so}_N}(\mu)\) (\(\mu \in X^+_{\mathfrak{so}_N}\)) occurs if and only if the highest weights \(\lambda,\mu\) satisfy the following \emph{interlacing condition}:
\begin{itemize}
  \item \(\mathsf{s}(\lambda) = \mathsf{s}(\mu)\), where
  \[
    \mathsf{s}(\lambda) :=
    \begin{cases}
      0 & \text{ if } \lambda \in \mathbb{Z}^n, \\
      \frac{1}{2} & \text{ if } \lambda \in (\frac{1}{2} + \mathbb{Z})^n,
    \end{cases}
  \]
  \item \(
  \begin{cases}
    \lambda_1 \geq \mu_1 \geq \lambda_2 \geq \mu_2 \geq \cdots \geq \lambda_n \geq |\mu_n| & \text{ if } N = 2n, \\
    \lambda_1 \geq \mu_1 \geq \lambda_2 \geq \mu_2 \geq \cdots \geq \lambda_n \geq \mu_n \geq |\lambda_{n+1}| & \text{ if } N = 2n+1.
  \end{cases}
  \)
\end{itemize}

Based on this branching rule, Gelfand and Tsetlin \cite{GeTs50b} constructed a basis for each \(V^{\mathfrak{so}_N}(\lambda)\).
As the \(\mathfrak{gl}\) case, the bases are parametrized by \emph{Gelfand--Tsetlin patterns}, sequences of highest weights whose successive pairs satisfy the interlacing condition for \(\mathfrak{so}\).
Let \(\mathsf{GTP}^{\mathfrak{so}_N}(\lambda)\) denote the set of Gelfand--Tsetlin patterns for \(V^{\mathfrak{so}_N}(\lambda)\).

For our purpose, it is convenient to rewrite the interlacing condition for \(\mathfrak{so}\) in terms of \emph{partitions} (of natural numbers).
Given a highest weight \(\lambda \in X^+_{\mathfrak{so}_N}\), set
\[
  \mathsf{par}(\lambda) := (\lceil \lambda_1 \rceil,\dots,\lceil \lambda_{n-1} \rceil, \lceil |\lambda_n| \rceil).
\]
It is a partition.
Then, two highest weights \(\lambda,\mu\) satisfy the interlacing condition if and only if \(\mathsf{s}(\lambda) = \mathsf{s}(\mu)\) and \(\mathsf{par}(\mu) \overset{\textsf{hor}}{\subseteq} \mathsf{par}(\lambda)\), meaning that the skew partition \(\mathsf{par}(\lambda)/\mathsf{par}(\mu)\) is a \emph{horizontal strip} (\S\ref{ssect_part}).
Hence, we have
\[
  V^{\mathfrak{so}_{N+1}}(\lambda)|_{\mathfrak{so}_N} \simeq \bigoplus_{\substack{\mu \in X^+_{\mathfrak{so}_N, \mathsf{s}(\lambda)} \\ \mathsf{par}(\mu) \overset{\textsf{hor}}{\subseteq} \mathsf{par}(\lambda)}} V^{\mathfrak{so}_N}(\mu),
\]
where
\[
  X^+_{\mathfrak{so}_N, s} := \{ \mu \in X^+_{\mathfrak{so}_N} \mid \mathsf{s}(\mu) = s \} \quad \text{ for each } s \in \left\{ 0,\frac{1}{2} \right\}.
\]

\subsection{Semistandard tableaux}
Let \(\lambda = (\lambda_1,\dots,\lambda_N) \in X^+_{\mathfrak{gl}_N}\) be such that \(\lambda_N \in \mathbb{Z}_{\geq 0}\).
Then, it is a partition.
The irreducible representation \(V^{\mathfrak{gl}_N}(\lambda)\) has another basis that is parametrized by the set \(\mathsf{SST}_N(\lambda)\) of \emph{semistandard tableaux} of shape \(\lambda\) with entries in \(\{ 1,\dots,N \}\); see \cite[Notation]{Ful97} for the precise definition (semistandard tableaux here are simply called tableaux there).

Since both \(\mathsf{SST}_N(\lambda)\) and \(\mathsf{GTP}^{\mathfrak{gl}_N}(\lambda)\) parametrize a basis of the same representation, they must be in a one-to-one correspondence.
Such a bijection is well-known and quite simple.
The Gelfand--Tsetlin pattern corresponding to a tableau \(T\) is
\[
  (\mathsf{sh}(T), \mathsf{sh}(T|_{N-1}), \dots, \mathsf{sh}(T|_1)),
\]
where \(\mathsf{sh}(S)\) denotes the shape of a tableau \(S\), and \(T|_{k}\) denotes the tableau obtained from \(T\) by removing the boxes whose entry is greater than \(k\).
For example, the tableau
\[
  \begin{ytableau}
    1 & 1 & 2 & 3 & 3 \\
    2 & 3 & 5 \\
    4 & 5
  \end{ytableau}
\]
corresponds to the pattern
\[
  \begin{tikzpicture}
    \node at (0,0) {\(5\) \(3\) \(2\) \(0\) \(0\)};
    \node at (0.00625,-0.5) {\(5\) \(2\) \(1\) \(0\)};
    \node at (0.00625*2,-0.5*2) {\(5\) \(2\) \(0\)};
    \node at (0.00625*3,-0.5*3) {\(3\) \(1\)};
    \node at (0.00625*4,-0.5*4) {\(2\)};
  \end{tikzpicture}.
\]
Here, we aligned the highest weights from top to bottom.

The operations \(\cdot|_k\) are quite natural in the theory of quantum groups (a.k.a. quantized enveloping algebras), especially crystals.
They amount to regard \(V^{\mathfrak{gl}_N}(\lambda)\) the \(\mathfrak{gl}_N\)-module as a \(\mathfrak{gl}_k\)-module via the chain
\[
  \mathfrak{gl}_k \hookrightarrow \mathfrak{gl}_{k+1} \hookrightarrow \cdots \hookrightarrow \mathfrak{gl}_N
\]
of Lie algebra homomorphisms.

\subsection{Kashiwara--Nakashima tableaux}
An \(\mathfrak{so}\) analogue of semistandard tableaux is \emph{Kashiwara--Nakashima tableaux}, introduced in \cite{KaNa94}, where a symplectic Lie algebra analogue is also defined.

Let \(\mathbf{U} = \mathbf{U}(\mathfrak{so}_N)\) denote the quantum group (a.k.a. quantized enveloping algebra) of \(\mathfrak{so}_N\).
For each \(\lambda \in X^+_{\mathfrak{so}_N}\), let \(V_q(\lambda) = V^{\mathfrak{so}_N}_q(\lambda)\) denote the simple \(\mathbf{U}\)-module of highest weight \(\lambda\).
Its crystal basis can be realized as the set \(\mathsf{KNT}^{\mathfrak{so}_N}(\lambda)\) of Kashiwara--Nakashima tableaux of shape \(\lambda\).

Since \(V_q(\lambda)\) tends to \(V^{\mathfrak{so}_N}(\lambda)\) as \(q\) tends to \(1\), the two sets \(\mathsf{KNT}^{\mathfrak{so}_N}(\lambda)\) and \(\mathsf{GTP}^{\mathfrak{so}_N}(\lambda)\) must be in a one-to-one correspondence.
However, unlike the \(\mathfrak{gl}\) case, no natural bijection is known.
The main obstacle is that there is no natural embedding \(\mathbf{U}(\mathfrak{so}_N) \to \mathbf{U}(\mathfrak{so}_{N+1})\) corresponding to \(\phi_N \colon \mathfrak{so}_N \to \mathfrak{so}_{N+1}\).

\subsection{Quantum symmetric pairs}
Kolb and Stephens \cite{KoSt24} pointed out that (variants of) the embeddings \(\phi_N \colon \mathfrak{so}_N \to \mathfrak{so}_{N+1}\) can be quantized to embeddings involving quantum symmetric pairs of type \(B\mathrm{II}\) or \(D\mathrm{II}\).
A \emph{quantum symmetric pair} is a pair consisting of a quantum group \(\mathbf{U}\) and its certain coideal subalgebra \(\mathbf{U}^\imath\), called an \(\imath\)quantum group; see \cite{Kol14} for a general theory, and \cite{Wan22} for a survey.
For type \(B\mathrm{II}_n\), the quantum group is \(\mathbf{U}(\mathfrak{so}_{2n+1})\), and the \(\imath\)quantum group is a quantum deformation \(\mathbf{U}^\imath_{B\mathrm{II}_n}\) of \(\mathfrak{so}_{2n}\).
Also, \(\mathbf{U}^\imath_{B\mathrm{II}_n}\) has a subalgebra isomorphic to \(\mathbf{U}(\mathfrak{so}_{2n-1})\).
For type \(D\mathrm{II}_n\), the quantum group is \(\mathbf{U}(\mathfrak{so}_{2n})\), and the \(\imath\)quantum group is a quantum deformation \(\mathbf{U}^\imath_{D\mathrm{II}_n}\) of \(\mathfrak{so}_{2n-1}\).
Also, \(\mathbf{U}^\imath_{D\mathrm{II}_n}\) has a subalgebra isomorphic to \(\mathbf{U}(\mathfrak{so}_{2n-2})\).
Therefore, there are two chains of algebras:
\begin{align}
  &\mathbf{U}^\imath_{B\mathrm{II}_1} \hookrightarrow \mathbf{U}(\mathfrak{so}_3) \hookrightarrow \cdots \mathbf{U}^\imath_{B\mathrm{II}_n} \hookrightarrow \mathbf{U}(\mathfrak{so}_{2n+1}) \hookrightarrow \cdots, \label{eq_chain_BII}\\
  &\mathbf{U}(\mathfrak{so}_2) \hookrightarrow \mathbf{U}^\imath_{D\mathrm{II}_2} \hookrightarrow \cdots \mathbf{U}(\mathfrak{so}_{2n-2}) \hookrightarrow \mathbf{U}^\imath_{D\mathrm{II}_n} \hookrightarrow \mathbf{U}(\mathfrak{so}_{2n}) \hookrightarrow \cdots.\label{eq_chain_DII}
\end{align}

In \cite{KoSt24}, Kolb and Stephens studied finite-dimensional representation theory of \(\mathbf{U}^\imath_{D\mathrm{II}_n}\), and proved that it is almost the same as \(\mathfrak{so}_{2n-1}\).
Namely, for each \(X^+_{\mathfrak{so}_{2n-1}}\), there exists a unique simple \(\mathbf{U}^\imath\)-module \(V^\imath(\nu)\) such that \(V^\imath(\nu)\) tends to \(V^{\mathfrak{so}_{2n-1}}(\nu)\) as \(q\) tends to \(1\), and
\[
  V_q(\lambda)|_{\mathbf{U}^\imath_{D\mathrm{II}_n}} \simeq \bigoplus_{\substack{\nu \in X^+_{\mathfrak{so}_{2n-1}, \mathsf{s}(\lambda)} \\ \mathsf{par}(\nu) \overset{\textsf{hor}}{\subseteq} \mathsf{par}(\lambda)}} V^\imath(\nu), \quad V^\imath(\nu)|_{\mathbf{U}(\mathfrak{so}_{2n-2})} \simeq \bigoplus_{\substack{\mu \in X^+_{\mathfrak{so}_{2n-2}, \mathsf{s}(\nu)} \\ \mathsf{par}(\mu) \overset{\textsf{hor}}{\subseteq} \mathsf{par}(\nu)}} V_q(\mu).
\]

Similar results for type \(B\mathrm{II}_n\) have not been published yet.
However, at the conference ``Quantum Symmetric Pairs, Hecke Algebras, and Representations: Exploring Spherical Functions (Q-SPHERE 2026)'', held at the Radboud University, Nijmegen, Netherlands in June 2026, the author attended a talk by Xinyang Liu, who announced her results on the finite-dimensional representation theory of \(\mathbf{U}^\imath_{B\mathrm{II}_n}\).

\subsection{Results}\label{ssect_results}
In the present paper, we construct a bijection between \(\mathsf{KNT}^{D_n}(\lambda) := \mathsf{KNT}^{\mathfrak{so}_{2n}}(\lambda)\) and \(\mathsf{GTP}^{\mathfrak{so}_{2n}}(\lambda)\) for each \(\lambda \in X^+_{2n,0}\) based on the chain \eqref{eq_chain_DII} of algebras.
It is a byproduct of a study of the representation theory of \(\mathbf{U}^{\imath}_{D\mathrm{II}_n}\).
Recall that a Kashiwara--Nakashima tableau \(T\) is a sequence \((\mathbf{a}^1,\dots,\mathbf{a}^m)\) of columns.
We first define the notion of columns of type \(D\mathrm{II}_n\) (Definition \ref{def_col_DII}).

\begin{defi}[Kashiwara--Nakashima tableaux of type \(D\mathrm{II}_n\) (Definition \ref{def_knt_DII})]
  A Kashiwara--Nakashima tableau \(T = (\mathbf{a}^1,\dots,\mathbf{a}^m)\) is said to be of \emph{type \(D\mathrm{II}_n\)} if its first column \(\mathbf{a}^1\) is of type \(D\mathrm{II}_n\).
\end{defi}

The Kashiwara--Nakashima tableaux of type \(D\mathrm{II}_n\) parametrize bases of simple \(\mathbf{U}^\imath\)-modules.

\begin{thm}[Basis of \(V^\imath(\nu)\) (Theorem \ref{thm_alm_orthn_basis_Vinu_DII})]
  Let \(\nu \in X^+_{2n-1,0}\).
  Then, the simple \(\mathbf{U}^\imath_{D\mathrm{II}_n}\)-module \(V^\imath(\nu)\) has a basis of the form
  \[
    \{ b^\imath_T \mid T \in \mathsf{KNT}^{D\mathrm{II}_n}(\nu) \}.
  \]
\end{thm}

One can transform each Kashiwara--Nakashima tableau of type \(D\mathrm{II}_n\) into of type \(D_{n-1}\) by removing the boxes whose entry is \(\overline{1}\) and then shifting the remaining entries by \(1\).
For example,
\[
  \begin{ytableau}
    2 & 2 \\
    *(lightgray) 3 \\
    *(lightgray) 1
  \end{ytableau}
  \mapsto
  \begin{ytableau}
    1 & 1 \\
    *(lightgray) 2
  \end{ytableau}.
\]
Here, we represent a box with entry \(\overline{a}\) by a shaded box with entry \(a\).

\begin{thm}[Reduction from \(D\mathrm{II}_n\) to \(D_{n-1}\) (Theorem \ref{thm_red_map_DII_D})]
  Let \(\nu \in X^+_{\mathfrak{so}_{2n-1},0}\).
  The assignment \(T \mapsto T \downarrow^{D\mathrm{II}_n}_{D_{n-1}}\) gives rise to a bijection
  \[
    \mathsf{KNT}^{D\mathrm{II}_n}(\nu) \to \bigsqcup_{\substack{\mu \in X^+_{2n-2,0} \\ \mathsf{par}(\mu) \overset{\textsf{hor}}{\subseteq} \mathsf{par}(\nu)}} \mathsf{KNT}^{D_{n-1}}(\mu).
  \]
\end{thm}

To transform a Kashiwara--Nakashima tableau of type \(D_n\) into type \(D\mathrm{II}_n\) is more involved (although it is quite natural from the perspective of representation theory, and similar transformation has appeared in the combinatorial representation theory of quantum symmetric pairs of type \(A\mathrm{II}\) \cite{Wat23f}).
Let \(T = (\mathbf{a}^1,\dots,\mathbf{a}^m)\) be a Kashiwara--Nakashima tableau of type \(D_n\).
If its first column \(\mathbf{a}^1\) is of type \(D\mathrm{II}_n\), then there is nothing to do; the tableau \(T\) is already of type \(D\mathrm{II}_n\).
Otherwise, transform \(\mathbf{a}^1\) into a column \(\mathbf{a}^1 \downarrow^{D_n}_{D\mathrm{II}_n}\) of type \(D\mathrm{II}_n\) (Definition \ref{def_red_col_D_DII}).
Then, construct a new tableau \(\mathbf{a}^1 \downarrow^{D_n}_{D\mathrm{II}_n} * \mathbf{a}^2 * \cdots * \mathbf{a}^m\) according to Lecouvey's Robinson--Schensted correspondence \cite{Lec03}.
Repeating this procedure sufficiently many times, we will eventually obtain a tableau \(T \downarrow^{D_n}_{D\mathrm{II}_n}\) of type \(D\mathrm{II}_n\).

\begin{thm}[Reduction from \(D_n\) to \(D\mathrm{II}_{n}\) (Theorem \ref{thm_red_KNT_D_DII})]
  Let \(\lambda \in X^+_{\mathfrak{so}_{2n},0}\).
  The assignment \(T \mapsto T \downarrow^{D_n}_{D\mathrm{II}_n}\) gives rise to a bijection
  \[
    \mathsf{KNT}^{D_n}(\lambda) \to \bigsqcup_{\substack{\nu \in X^+_{2n-1,0} \\ \mathsf{par}(\nu) \overset{\textsf{hor}}{\subseteq} \mathsf{par}(\lambda)}} \mathsf{KNT}^{D\mathrm{II}_n}(\nu).
  \]
\end{thm}

Combining the two bijections above, we finally obtain bijections between Kashiwara--Nakashima tableaux and Gelfand--Tsetlin patterns.
Let \(\lambda \in X^+_{2n,0}\) and \(T \in \mathsf{KNT}^{D_n}(\lambda)\).
Define tableaux \(T \downarrow^{D_n}_{D_{k-1}}\) for \(k \in \{ 2,\dots,n \}\), and  \(T \downarrow^{D_n}_{D\mathrm{II}_k}\) for \(k \in \{ 2,\dots,n-1 \}\) inductively by
\[
  T \downarrow^{D_n}_{D_{k-1}} := (T \downarrow^{D_n}_{D\mathrm{II}_{k}}) \downarrow^{D\mathrm{II}_{k}}_{D_{k-1}}, \quad T \downarrow^{D_n}_{D\mathrm{II}_k} := (T \downarrow^{D_n}_{D_{k}}) \downarrow^{D_{k}}_{D\mathrm{II}_{k}}.
\]
Similarly, for each \(\nu \in X^+_{2n-1,0}\) and \(T \in \mathsf{KNT}^{D\mathrm{II}_n}(\nu)\) define
\[
  T \downarrow^{D\mathrm{II}_n}_{D\mathrm{II}_{k}} := (T \downarrow^{D\mathrm{II}_n}_{D_{k}}) \downarrow^{D_{k}}_{D\mathrm{II}_{k}}, \quad T \downarrow^{D\mathrm{II}_n}_{D_{k-1}} := (T \downarrow^{D\mathrm{II}_n}_{D\mathrm{II}_{k}}) \downarrow^{D\mathrm{II}_{k}}_{D_{k-1}}.
\]

\begin{thm}[Bijections between tableaux and patterns (Corollary \ref{cor_bij_D})]
  \hfill
  \begin{enumerate}
    \item Let \(\lambda \in X^+_{2n,0}\).
    The assignment
    \[
      T \mapsto (\mathsf{sh}(T), \mathsf{sh}(T \downarrow^{D_n}_{D\mathrm{II}_n}), \mathsf{sh}(T \downarrow^{D_n}_{D_{n-1}}),\dots,\mathsf{sh}(T \downarrow^{D_n}_{D_1}))
    \]
    gives rise to a bijection
    \[
      \mathsf{KNT}^{D_n}(\lambda) \to \mathsf{GTP}^{\mathfrak{so}_{2n}}(\lambda).
    \]
    \item Let \(\nu \in X^+_{2n-1,0}\).
    The assignment
    \[
      T \mapsto (\mathsf{sh}(T), \mathsf{sh}(T \downarrow^{D\mathrm{II}_n}_{D_{n-1}}), \mathsf{sh}(T \downarrow^{D\mathrm{II}_n}_{D\mathrm{II}_{n-1}}),\dots,\mathsf{sh}(T \downarrow^{D\mathrm{II}_n}_{D_1}))
    \]
    gives rise to a bijection
    \[
      \mathsf{KNT}^{D\mathrm{II}_n}(\nu) \to \mathsf{GTP}^{\mathfrak{so}_{2n-1}}(\nu).
    \]
  \end{enumerate}
\end{thm}

For example, when \(n = 3\), we have
\[
  T :=
  \begin{ytableau}
    1 & 3 \\
    3 \\
    *(lightgray) 3
  \end{ytableau}
  \xrightarrow{\cdot \downarrow^{D_3}_{D\mathrm{II}_3}}
  \begin{ytableau}
    3 & 3 \\
    *(lightgray) 3
  \end{ytableau}
  \xrightarrow{\cdot \downarrow^{D\mathrm{II}_3}_{D_2}}
  \begin{ytableau}
    2 & 2 \\
    *(lightgray) 2
  \end{ytableau}
  \xrightarrow{\cdot \downarrow^{D_2}_{D\mathrm{II}_2}}
  \begin{ytableau}
    2
  \end{ytableau}
  \xrightarrow{\cdot \downarrow^{D\mathrm{II}_2}_{D_1}}
  \begin{ytableau}
    1
  \end{ytableau}.
\]
Hence, the Gelfand--Tsetlin pattern corresponding to \(T\) is
\[
  \begin{tikzpicture}
    \node[anchor=west] at (0,0) {\(2\) \(1\) \(-1\)};
    \node[anchor=west] at (0.2,-0.5) {\(2\) \(1\)};
    \node[anchor=west] at (0.2*2,-0.5*2) {\(2\) \(-1\)};
    \node[anchor=west] at (0.2*3,-0.5*3) {\(1\)};
    \node[anchor=west] at (0.2*4,-0.5*4) {\(1\)};
  \end{tikzpicture}.
\]

Our strategy is invalid for \(\lambda \in X^+_{\mathfrak{so}_{2n},\frac{1}{2}}\) by the following reason.
From the viewpoint of representation theory of quantum groups, Kashiwara--Nakashima tableaux live in tensor products of modules corresponding to columns.
The key to our proof is to reduce the leftmost column; \(\mathbf{a}^1 \mapsto \mathbf{a}^1 \downarrow^{D_n}_{D\mathrm{II}_n}\).
Since the \(\imath\)quantum group is a right coideal of the quantum group, we cannot reduce the other columns.
For \(\lambda \in X^+_{\mathfrak{so}_{2n},\frac{1}{2}}\), the \(\mathbf{U}\)-module corresponding to the first column is already simple as a \(\mathbf{U}^\imath_{D\mathrm{II}_n}\)-module.
Hence, we cannot reduce the first column.

If one modifies the Kashiwara--Nakashima tableaux in a way such that the spin columns lie in the rightmost, then our main theorem would be extended to the whole of \(X^+_{\mathfrak{so}_{2n}}\).

Finally, let us mention the type \(B\) case.
All the arguments for type \(D\) in the present paper have type \(B\) analogues.
The only problem is that we do not know if each simple \(\mathfrak{so}_{2n}\)-module has a unique quantum deformation as a classical weight \(\mathbf{U}^\imath_{B\mathrm{II}_n}\)-module (in the sense of \cite{Wat21b}).
In the present paper, we construct bijections under the assumption that this is true.

\subsection{Organization}
This paper is organized as follows.
\S\ref{sect_prelim} is a preliminary part.
We introduce basic notations concerning the special orthogonal Lie algebras, Gelfand--Tsetlin patterns of orthogonal type, and crystals.
In \S\ref{sect_KNT_D}, we recall the definition of Kashiwara--Nakashima tableaux of type \(D\) and Lecouvey's Robinson--Schensted correspondence.
\S\ref{sect_manipulation_D} is devoted to combinatorial arguments involving Kashiwara--Nakashima tableaux of type \(D\) and \(D\mathrm{II}\).
Our main results are presented and partially proved there.
We complete the proof with the aid of representation theory of quantum symmetric pairs of type \(D\mathrm{II}\) in \S\S\ref{sect_qg_D}--\ref{sect_qsp_DII}.
In \S\ref{sect_B}, we briefly describe the results for type \(B\).

\subsection*{Acknowledgements}
The author would like to dedicate this paper to his grandmother Shizue Fukumoto, who passed away while he was working on this project, with deep gratitude and respect for her character.
This work was supported by JSPS KAKENHI Grant Number JP24K16903.

\section{Preliminaries}\label{sect_prelim}
In this section, we introduce basic notations concerning the special orthogonal Lie algebras, Gelfand--Tsetlin patterns of orthogonal type, and crystals.
We refer the reader to \cite{Mol06b} for Gelfand--Tsetlin patterns and to \cite{BuSc17} for crystals.

\subsection{Special orthogonal Lie algebras}\label{ssect_so}
Let \(N \in \mathbb{Z}_{\geq 2}\), and \(\mathfrak{gl}_N = \mathsf{Mat}_{N}(\mathbb{C})\) denote the \emph{general linear Lie algebra} over the field of complex numbers.
For each \(A = (a_{i,j})_{1 \leq i,j \leq N} \in \mathfrak{gl}_N\), let \({}^{\mathsf{at}}A := (a_{N-j+1,N-i+1})_{1 \leq i,j \leq N}\) denote the anti-transpose of \(A\); the matrix obtained by flipping \(A\) with respect to the anti-diagonal.

Let \(\mathfrak{so}_N\) denote the special orthogonal Lie algebra:
\[
  \mathfrak{so}_N := \{ A \in \mathfrak{gl}_N \mid {}^{\mathsf{at}}A = -A \}.
\]
It is spanned by
\[
  \{ Z_{i,j} := E_{i,j} - {}^{\mathsf{at}}E_{i,j} \mid 1 \leq i,j \leq N \},
\]
where \(E_{i,j}\)'s denote the matrix units.

Set \(n := \lfloor \frac{N}{2} \rfloor\).
Then, there exists an injective Lie algebra homomorphism
\begin{align}\label{eq_so_embd}
  \phi_N \colon \mathfrak{so}_N \to \mathfrak{so}_{N+1}
\end{align}
such that
\[
  \phi_N(Z_{i,j}) =
  \begin{cases}
    Z_{i,j} & \text{ if } 1 \leq i,j \leq n, \\
    Z_{i,j+1} & \text{ if } 1 \leq i,N-j+1 \leq n, \\
    Z_{i+1,j} & \text{ if } 1 \leq N-i+1,j \leq n, \\
    \frac{1}{\sqrt{2}}(Z_{i,n+1} + Z_{i,n+2}) & \text{ if } N = 2n+1,\ 1 \leq i \leq n, \text{ and } j = n+1, \\
    \frac{1}{\sqrt{2}}(Z_{n+1,j} + Z_{n+2,j}) & \text{ if } N = 2n+1,\ 1 \leq j \leq n, \text{ and } i = n+1.
  \end{cases}
\]
In terms of block decomposition, we have
\begin{align*}
  &\phi_{2n}
  \begin{pmatrix}
    A & B \\
    C & -{}^{\mathsf{at}}A
  \end{pmatrix}
  =
  \begin{pmatrix}
    A & \mathbf{o} & B \\
    {}^{\mathsf{at}}\mathbf{o} & 0 & {}^{\mathsf{at}}\mathbf{o} \\
    C & \mathbf{o} & -{}^{\mathsf{at}}A
  \end{pmatrix}
  , \\
  &\phi_{2n+1}
  \begin{pmatrix}
    A & \mathbf{u} & B \\
    {}^{\mathsf{at}}\mathbf{v} & 0 & -{}^{\mathsf{at}}\mathbf{u} \\
    C & -\mathbf{v} & -{}^{\mathsf{at}}A
  \end{pmatrix}
  =
  \begin{pmatrix}
    A & \frac{\mathbf{u}}{\sqrt{2}} & \frac{\mathbf{u}}{\sqrt{2}} & B \\
    \frac{{}^{\mathsf{at}}\mathbf{v}}{\sqrt{2}} & 0 & 0 & -\frac{{}^{\mathsf{at}}\mathbf{u}}{\sqrt{2}} \\
    \frac{{}^{\mathsf{at}}\mathbf{v}}{\sqrt{2}} & 0 & 0 & -\frac{{}^{\mathsf{at}}\mathbf{u}}{\sqrt{2}} \\
    C & -\frac{\mathbf{v}}{\sqrt{2}} & -\frac{\mathbf{v}}{\sqrt{2}} & -{}^{\mathsf{at}}A
  \end{pmatrix}.
\end{align*}

For each \(i \in \{ 1,\dots,n \}\), set \(d_i := Z_{i,i}\), and
\[
  \mathfrak{h}_N := \mathsf{Span}_{\mathbb{C}} \{ d_1,\dots,d_n \}.
\]

When \(N = 2\), we have \(\mathfrak{so}_2 = \mathfrak{h}_2 = \mathbb{C} d_1\) and \(\phi_2(d_1) = d_1\).

When \(N \geq 3\), for each \(i \in \{ 1,\dots,n \}\), set
\begin{align*}
  &e_i :=
  \begin{cases}
    Z_{i,i+1} & \text{ if } i < n, \\
    Z_{n-1,n+1} & \text{ if } i = n \text{ and } N = 2n, \\
    \sqrt{2}Z_{n,n+1} & \text{ if } i = n \text{ and } N = 2n+1,
  \end{cases} \\
  &f_i :=
  \begin{cases}
    Z_{i+1,i} & \text{ if } i < n, \\
    Z_{n+1,n-1} & \text{ if } i = n \text{ and } N = 2n, \\
    \sqrt{2}Z_{n+1,n} & \text{ if } i = n \text{ and } N = 2n+1,
  \end{cases} \\
  &h_i :=
  \begin{cases}
    d_i-d_{i+1} & \text{ if } i < n, \\
    d_{n-1} + d_{n} & \text{ if } i = n \text{ and } N = 2n, \\
    2d_n & \text{ if } i = n \text{ and } N = 2n+1.
  \end{cases}
\end{align*}
Then, the triples \(e_i,f_i,h_i\) form a system of Chevalley generators.
For each \(i \in \{ 1,\dots,n \}\), we have
\begin{align*}
  &\phi_N(d_i) = d_i, \\
  &\phi_N(e_i)
  =
  \begin{cases}
    e_i & \text{ if } 1 \leq i < n, \\
    Z_{n-1,n+2} & \text{ if } i = n \text{ and } N = 2n, \\
    e_n + e_{n+1} & \text{ if } i = n \text{ and } N = 2n+1,
  \end{cases} \\
  &\phi_N(f_i)
  =
  \begin{cases}
    f_i & \text{ if } 1 \leq i < n, \\
    Z_{n+2,n-1} & \text{ if } i = n \text{ and } N = 2n, \\
    f_n + f_{n+1} & \text{ if } i = n \text{ and } N = 2n+1,
  \end{cases} \\
  &\phi_N(h_i)
  =
  \begin{cases}
    h_i & \text{ if } 1 \leq i < n, \\
    h_{n-1} + h_n & \text{ if } i = n \text{ and } N = 2n, \\
    h_n + h_{n+1} & \text{ if } i = n \text{ and } N = 2n+1.
  \end{cases}
\end{align*}

Let us go back to the general case; \(N \geq 2\).
Let \(\{ \epsilon_1,\dots,\epsilon_n \} \subset \mathfrak{h}_N^*\) denote the dual basis of \(\{ d_1,\dots,d_n \}\):
\[
  \langle d_i, \epsilon_j \rangle = \delta_{i,j} \quad \text{ for all } i,j \in \{ 1,\dots,n \}.
\]
We often identify an element \(\lambda \in \mathfrak{h}_N^*\) with the \(n\)-tuple \((\langle d_1, \lambda \rangle,\dots,\langle d_n, \lambda \rangle) \in \mathbb{C}^n\).
Then, the set \(X_N\) of integral weights is identified with \(\mathbb{Z}^n \sqcup (\frac{1}{2} + \mathbb{Z})^n\), and the set \(X^+_N\) of dominant weights is described as follows:
\[
  X^+_N =
  \begin{cases}
    \{ (\lambda_1,\dots,\lambda_n) \in X_N \mid \lambda_1 \geq \cdots \geq \lambda_{n-1} \geq |\lambda_n| \} & \text{ if } N = 2n, \\
    \{ (\lambda_1,\dots,\lambda_n) \in X_N \mid \lambda_1 \geq \cdots \geq \lambda_n \geq 0 \} & \text{ if } N = 2n+1.
  \end{cases}
\]

For each \(\lambda \in X^+_{N}\), we define its \emph{sign} \(\mathsf{e}(\lambda)\) and \emph{spin} \(\mathsf{s}(\lambda)\) by
\[
  \mathsf{e}(\lambda) :=
  \begin{cases}
    + & \text{ if } N = 2n \text{ and } \lambda_n > 0, \\
    - & \text{ if } N = 2n \text{ and } \lambda_n < 0, \\
    0 & \text{ otherwise},
  \end{cases}
  \quad
  \mathsf{s}(\lambda) :=
  \begin{cases}
    0 & \text{ if } \lambda \in \mathbb{Z}^n, \\
    \frac{1}{2} & \text{ if } \lambda \in \left(\frac{1}{2} + \mathbb{Z}\right)^n.
  \end{cases}
\]
Set
\[
  X^+_{N,\pm} := \{ \lambda \in X^+_N \mid \mathsf{e}(\lambda) = \pm \},
  \quad
  X^+_{N,s} := \{ \lambda \in X^+_N \mid \mathsf{s}(\lambda) = s \} \ \text{ for each } s \in \left\{ 0, \frac{1}{2} \right\}.
\]

For each \(\lambda \in X^+_N\), let \(V(\lambda) = V_N(\lambda)\) denote the finite-dimensional simple \(\mathfrak{so}_N\)-module of highest weight \(\lambda\).

\subsection{Gelfand--Tsetlin patterns}
Let \(N \in \mathbb{Z}_{\geq 2}\).
Set \(n := \lfloor \frac{N}{2} \rfloor\).

\begin{defi}[Interlacing condition]\label{def_interlace_cond}
  We say that \(\lambda \in X^+_{N+1}\) and \(\mu \in X^+_N\) satisfy the \emph{interlacing condition} if the following hold:
  \begin{itemize}
    \item \(\mathsf{s}(\lambda) = \mathsf{s}(\mu)\),
    \item \(
    \begin{cases}
      \lambda_1 \geq \mu_1 \geq \lambda_2 \geq \mu_2 \geq \cdots \geq \lambda_n \geq |\mu_n| & \text{ if } N = 2n, \\
      \lambda_1 \geq \mu_1 \geq \lambda_2 \geq \mu_2 \geq \cdots \geq \lambda_n \geq \mu_n \geq |\lambda_{n+1}| & \text{ if } N = 2n+1.
    \end{cases}
    \)
  \end{itemize}
\end{defi}

For each \(\lambda \in X^+_{N+1}\), the \(\mathfrak{so}_{N+1}\)-module \(V_{N+1}(\lambda)\) can be regarded as an \(\mathfrak{so}_N\)-module via the embedding \(\phi_N\) \eqref{eq_so_embd}.
It is known that this \(\mathfrak{so}_N\)-module is semisimple and multiplicity-free.
Moreover, the irreducible component of type \(V_N(\mu)\) occurs if and only if \(\lambda\) and \(\mu\) satisfy the interlacing condition:
\begin{align}\label{eq_branching_so}
  V_{N+1}(\lambda)|_{\mathfrak{so}_N} \simeq \bigoplus_{\substack{\mu \in X^+_N \\ \text{\(\lambda,\mu\) satisfy the interlacing condition}}} V_N(\mu).
\end{align}

Based on this branching rule for various \(N\), Gelfand and Tsetlin \cite{GeTs50b} constructed a basis of \(V_{N}(\lambda)\), called the \emph{Gelfand--Tsetlin basis}.
It is parametrized by arrays of integers or half integers, called the \emph{Gelfand--Tsetlin patterns}.

\begin{defi}[Gelfand--Tsetlin patterns {\cite{GeTs50b}, \cite[\S4.3]{Mol06b}}]\label{def_GTP}
  Let \(\lambda \in X^+_N\).
  A \emph{Gelfand--Tsetlin pattern of type \(\mathfrak{so}_N\) and shape \(\lambda\)} is an array \((\lambda^k_i)_{\substack{2 \leq k \leq N \\ 1 \leq i \leq \lfloor \frac{k}{2} \rfloor}}\) of integers or half integers satisfying the following:
  \begin{description}
    \item[Entry] \(\lambda^k := (\lambda^k_1,\dots,\lambda^k_{\lfloor \frac{k}{2} \rfloor}) \in X^+_k\) for all \(k \in \{ 2,\dots,N \}\),
    \item[Initial Condtion] \(\lambda^N = \lambda\)
    \item[Interlacing Condtion] \(\lambda^{k+1}\) and \(\lambda^k\) satisfy the interlacing condition for all \(k \in \{ 2,\dots,N-1 \}\) (Figure \ref{fig_interlacing_cond}).
  \end{description}
  Let
  \[
    \mathsf{GTP}_N(\lambda)
  \]
  denote the set of Gelfand--Tsetlin patterns of shape \(\lambda\).
\end{defi}

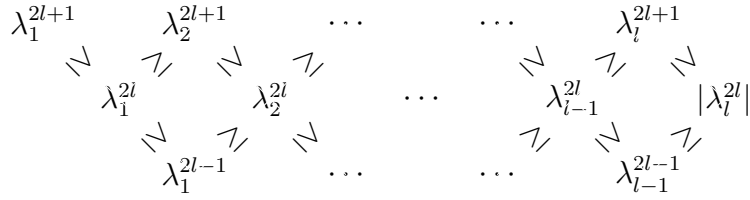
\begin{figure}[h]
  \begin{center}
    \begin{tikzpicture}[auto=left]
      \node at (0,0) {\(\lambda^{2l+1}_1\)};
      \node at (2,0) {\(\lambda^{2l+1}_2\)};
      \node at (4,0) {\(\cdots\)};
      \node at (6,0) {\(\cdots\)};
      \node at (8,0) {\(\lambda^{2l+1}_l\)};
      \node at (1,-1) {\(\lambda^{2l}_1\)};
      \node at (3,-1) {\(\lambda^{2l}_2\)};
      \node at (5,-1) {\(\cdots\)};
      \node at (7,-1) {\(\lambda^{2l}_{l-1}\)};
      \node at (9,-1) {\(|\lambda^{2l}_l|\)};
      \node at (2,-2) {\(\lambda^{2l-1}_1\)};
      \node at (4,-2) {\(\cdots\)};
      \node at (6,-2) {\(\cdots\)};
      \node at (8,-2) {\(\lambda^{2l-1}_{l-1}\)};

      \draw[white] (1,-1) to node[sloped, midway, anchor=center, black] {\(\geq\)} ++(-1,1);
      \draw[white] (3,-1) to node[sloped, midway, anchor=center, black] {\(\geq\)} ++(-1,1);
      \draw[white] (7,-1) to node[sloped, midway, anchor=center, black] {\(\geq\)} ++(-1,1);
      \draw[white] (9,-1) to node[sloped, midway, anchor=center, black] {\(\geq\)} ++(-1,1);
      \draw[white] (2,-2) to node[sloped, midway, anchor=center, black] {\(\geq\)} ++(-1,1);
      \draw[white] (4,-2) to node[sloped, midway, anchor=center, black] {\(\geq\)} ++(-1,1);
      \draw[white] (8,-2) to node[sloped, midway, anchor=center, black] {\(\geq\)} ++(-1,1);

      \draw[white] (2,0) to node[sloped, midway, anchor=center, black] {\(\geq\)} ++(-1,-1);
      \draw[white] (4,0) to node[sloped, midway, anchor=center, black] {\(\geq\)} ++(-1,-1);
      \draw[white] (8,0) to node[sloped, midway, anchor=center, black] {\(\geq\)} ++(-1,-1);
      \draw[white] (3,-1) to node[sloped, midway, anchor=center, black] {\(\geq\)} ++(-1,-1);
      \draw[white] (7,-1) to node[sloped, midway, anchor=center, black] {\(\geq\)} ++(-1,-1);
      \draw[white] (9,-1) to node[sloped, midway, anchor=center, black] {\(\geq\)} ++(-1,-1);
    \end{tikzpicture}
  \end{center}
  \caption{Interlacing condition}\label{fig_interlacing_cond}
\end{figure}

\subsection{Partitions}\label{ssect_part}
We refer the reader to \cite{Ful97} for basics of partitions.
A \emph{partition} is a finite sequence of nonnegative integers.
Given a partition \(\lambda = (\lambda_1,\dots,\lambda_n)\), we use the following notations:
\begin{itemize}
  \item \(\ell(\lambda) := \#\{ i \mid \lambda_i \neq 0 \}\): the length of \(\lambda\),
  \item \(|\lambda| := \lambda_1+\dots+\lambda_n\): the size of \(\lambda\),
  \item \(\mathsf{D}(\lambda) := \{ (i,j) \in \mathbb{Z}_{\geq 1}^2 \mid 1 \leq i \leq n,\ 1 \leq j \leq \lambda_i  \}\): the Young diagram of \(\lambda\),
  \item \(\lambda'\): the transpose of \(\lambda\), i.e., the partition whose Young diagram is
  \[
    \{ (j,i) \mid (i,j) \in \mathsf{D}(\lambda) \}.
  \]
\end{itemize}

Given two partitions \(\lambda,\mu\), we write \(\mu \subseteq \lambda\) to mean that \(\mathsf{D}(\mu) \subseteq \mathsf{D}(\lambda)\).
Also, we write \(\mu \overset{\textsf{hor}}{\subseteq} \lambda\) to mean that \(\mu \subseteq \lambda\) and the skew partition \(\lambda/\mu\) forms a \emph{horizontal strip}, i.e.,
\[
  \#\{ i \mid (i,j) \in D(\lambda) \setminus D(\mu) \} \leq 1 \quad \text{ for all } j.
\]
We have \(\mu \overset{\textsf{hor}}{\subseteq} \lambda\) if and only if
\[
  \lambda_1 \geq \mu_1 \geq \lambda_2 \geq \mu_2 \geq \cdots.
\]

\begin{ex}[Partitions]
  \ytableausetup{smalltableaux}
  Set \(\lambda := (5,4,3,3,1)\) and \(\mu := (4,4,3,1)\).
  Then, we have
  \begin{itemize}
    \item \(\ell(\lambda) = 5\),
    \item \(|\lambda| = 16\),
    \item \(\lambda' = (5,4,4,2,1)\),
    \item \(\mu \overset{\textsf{hor}}{\subseteq} \lambda\):
    \begin{align*}
      D(\lambda) = \ydiagram{5,4,3,3,1},\quad D(\mu) = \ydiagram{4,4,3,1},\quad D(\lambda) \setminus D(\mu) = \ydiagram{4+1,4+0,3+0,1+2,0+1}
    \end{align*}
  \end{itemize}
\end{ex}

Let \(N \in \mathbb{Z}_{\geq 2}\), and set \(n := \lfloor \frac{N}{2} \rfloor\).
For each \(\lambda \in X^+_{N}\), define a partition \(\mathsf{par}(\lambda)\) by
\[
  \mathsf{par}(\lambda) := (\lceil \lambda_1 \rceil,\dots,\lceil \lambda_{n-1} \rceil,\lceil |\lambda_n| \rceil).
\]
Two elements \(\lambda \in X^+_{N+1}\) and \(\mu \in X^+_{N}\) satisfy the interlacing condition (Definition \ref{def_interlace_cond}) if and only if
\[
  \mathsf{s}(\lambda) = \mathsf{s}(\mu) \text{ and } \mathsf{par}(\mu) \overset{\textsf{hor}}{\subseteq} \mathsf{par}(\lambda).
\]
Therefore, the branching rule \eqref{eq_branching_so} can be rewritten as follows:
\begin{align}\label{eq_branching_so_partition}
  V_{N+1}(\lambda)|_{\mathfrak{so}_N} \simeq \bigoplus_{\substack{\mu \in X^+_{N, \mathsf{s}(\lambda)} \\ \mathsf{par}(\mu) \overset{\textsf{hor}}{\subseteq} \mathsf{par}(\lambda)}} V_N(\mu).
\end{align}

\subsection{Crystals}\label{ssect_crystal}
Let \(N \in \mathbb{Z}_{\geq 2}\), and set
\begin{itemize}
  \item \(n := \lfloor \frac{N}{2} \rfloor\),
  \item \(I :=
  \begin{cases}
    \{ 1,\dots,n \} & \text{ if } N \geq 3, \\
    \emptyset & \text{ if } N = 2.
  \end{cases}
  \)
\end{itemize}
For each \(i \in I\), define \(\alpha_i \in \mathfrak{h}_N^*\) by
\[
  \alpha_i :=
  \begin{cases}
    \epsilon_i-\epsilon_{i+1} & \text{ if } i \in I \setminus \{ n \}, \\
    \epsilon_{n-1}+\epsilon_n & \text{ if } N = 2n \geq 4 \text{ and } i = n, \\
    \epsilon_n & \text{ if } N = 2n+1 \text{ and } i = n.
  \end{cases}
\]

A \emph{crystal of type \(\mathfrak{so}_N\)}, or an \emph{\(\mathfrak{so}_N\)-crystal}, is a set \(\mathcal{B}\) equipped with maps
\begin{itemize}
  \item \(\widetilde{E}_i, \widetilde{F}_i \colon \mathcal{B} \to \mathcal{B} \sqcup \{ 0 \}\) (\(i \in I\)),
  \item \(\mathsf{wt} \colon \mathcal{B} \to X\),
\end{itemize}
where \(0\) denotes a formal symbol, satisfying the following:
\begin{itemize}
  \item \(\mathsf{wt}(\widetilde{E}_i b) := \mathsf{wt}(b) + \alpha_i\) if \(\widetilde{E}_i b \neq 0\), for all \(i \in I\), \(b \in \mathcal{B}\),
  \item \(\widetilde{E}_i b_1 = b_2\) if and only if \(b_1 = \widetilde{F}_i b_2\) for all \(i \in I\), \(b_1,b_2 \in \mathcal{B}\).
\end{itemize}
An \(\mathfrak{so}_N\)-crystal is also called a crystal of \emph{type \(D_n\)} or a \emph{\(D_n\)-crystal} if \(N = 2n\), and a crystal of \emph{type \(B_n\)} or a \(B_n\)-crystal if \(N = 2n+1\).

A \emph{morphism of crystals} \(\mathcal{B}_1 \to \mathcal{B}_2\) is a map \(\mathcal{B}_1 \sqcup \{ 0 \} \to \mathcal{B}_2 \sqcup \{ 0 \}\) commuting with the maps \(\widetilde{E}_i, \widetilde{F}_i\) and preserving the map \(\mathsf{wt}\).
A morphism is said to be an \emph{isomorphism} if it is bijective. 

\begin{rem}[Crystals and morphisms]
  Our crystals and morphisms are called \emph{seminormal crystals} and \emph{strict morphisms}, respectively in \cite{BuSc17}.
\end{rem}

\begin{defi}[Subcrystals]
  A subset of a crystal is said to be a \emph{subcrystal} if the inclusion map is a crystal morphism.
\end{defi}

Let \(\mathcal{A}^{D_n} := \{ 1,\dots,n,\overline{n},\dots,\overline{1} \}\) denote the poset defined by
\[
  1 \prec \cdots \prec n, \overline{n} \prec \cdots \prec \overline{1};
\]
note that the letters \(n\) and \(\overline{n}\) are incomparable.
It has the following \(D_n\)-crystal structure:
\begin{align*}
  &\mathsf{wt}(a) :=
  \begin{cases}
    \epsilon_a & \text{ if } a \preceq n, \\
    -\epsilon_b & \text{ if } a = \overline{b} \text{ for some } b \preceq n,
  \end{cases} \\
  &\begin{tikzpicture}[auto=left, xscale=2, yscale=1.5]
    \node (1) at (1,0) {\(1\)};
    \node (dots_t) at (2,0) {\(\cdots\)};
    \node (n-1) at (3,0) {\(n-1\)};
    \node (n) at (4,0) {\(n\)};
    \node (nb) at (3,-1) {\(\overline{n}\)};
    \node (n-1b) at (4,-1) {\(\overline{n-1}\)};
    \node (dots_b) at (5,-1) {\(\cdots\)};
    \node (1b) at (6,-1) {\(\overline{1}\)};
    \draw[->] (1) to node[font=\scriptsize] {\(1\)} (dots_t);
    \draw[->] (dots_t) to node[font=\scriptsize] {\(n-2\)} (n-1);
    \draw[->] (n-1) to node[font=\scriptsize] {\(n-1\)} (n);
    \draw[->, left] (n-1) to node[font=\scriptsize] {\(n\)} (nb);
    \draw[->] (n) to node[font=\scriptsize] {\(n\)} (n-1b);
    \draw[->, below] (nb) to node[font=\scriptsize] {\(n-1\)} (n-1b);
    \draw[->, below] (n-1b) to node[font=\scriptsize] {\(n-2\)} (dots_b);
    \draw[->, below] (dots_b) to node[font=\scriptsize] {\(1\)} (1b);
  \end{tikzpicture}
\end{align*}
Here, as usual, we represent the maps \(\widetilde{E}_i, \widetilde{F}_i\) by the crystal graph (\cite[\S 2.2]{BuSc17}).
Also, let \(\mathcal{A}^{B_n} := \{ 1,\dots,n,0,\overline{n},\dots,\overline{1} \}\) denote the totally ordered set defined by
\[
  1 \prec \cdots \prec n \prec 0 \prec \overline{n} \prec \cdots \prec \overline{1}.
\]
It has the following \(B_n\)-crystal structure:
\begin{align*}
  &\mathsf{wt}(a) :=
  \begin{cases}
    \epsilon_a & \text{ if } a \preceq n, \\
    0 & \text{ if } a = 0, \\
    -\epsilon_b & \text{ if } a = \overline{b} \text{ for some } b \preceq n,
  \end{cases} \\
  &\begin{tikzpicture}[auto=left, xscale=2, yscale=1.5]
    \node (1) at (1,0) {\(1\)};
    \node (dots_t) at (2,0) {\(\cdots\)};
    \node (n) at (3,0) {\(n\)};
    \node (0) at (3,-1) {\(0\)};
    \node (nb) at (3,-2) {\(\overline{n}\)};
    \node (dots_b) at (4,-2) {\(\cdots\)};
    \node (1b) at (5,-2) {\(\overline{1}\)};
    \draw[->] (1) to node[font=\scriptsize] {\(1\)} (dots_t);
    \draw[->] (dots_t) to node[font=\scriptsize] {\(n-1\)} (n);
    \draw[->] (n) to node[font=\scriptsize] {\(n\)} (0);
    \draw[->] (0) to node[font=\scriptsize] {\(n\)} (nb);
    \draw[->, below] (nb) to node[font=\scriptsize] {\(n-1\)} (dots_b);
    \draw[->, below] (dots_b) to node[font=\scriptsize] {\(1\)} (1b);
  \end{tikzpicture}
\end{align*}

Given two crystals \(\mathcal{B}_1, \mathcal{B}_2\), its tensor product \(\mathcal{B}_1 \otimes \mathcal{B}_2 := \{ b_1 \otimes b_2 \mid b_1 \in \mathcal{B}_1,\ b_2 \in \mathcal{B}_2 \}\) is equipped with the following crystal structure:
\begin{align*}
  &\widetilde{E}_i(b_1 \otimes b_2) :=
  \begin{cases}
    \widetilde{E}_i(b_1) \otimes b_2 & \text{ if } \varepsilon_i(b_1) > \varphi_i(b_2), \\
    b_1 \otimes \widetilde{E}_i(b_2) & \text{ if } \varepsilon_i(b_1) \leq \varphi_i(b_2),
  \end{cases} \\
  &\widetilde{F}_i(b_1 \otimes b_2) :=
  \begin{cases}
    b_1 \otimes \widetilde{F}_i(b_2) & \text{ if } \varepsilon_i(b_1) < \varphi_i(b_2), \\
    \widetilde{F}_i(b_1) \otimes b_2 & \text{ if } \varepsilon_i(b_1) \geq \varphi_i(b_2),
  \end{cases} \\
  &\mathsf{wt}(b_1 \otimes b_2) := \mathsf{wt}(b_1) + \mathsf{wt}(b_2),
\end{align*}
where
\[
  \varepsilon_i(b) := \max \{ k \in \mathbb{Z}_{\geq 0} \mid \widetilde{E}_i^k b \neq 0 \}, \quad \varphi_i(b) := \max \{ k \in \mathbb{Z}_{\geq 0} \mid \widetilde{F}_i^k b \neq 0 \}.
\]

\begin{rem}[Tensor product of crystals]
  Our tensor product of crystals is the same as \cite{BuSc17}, but opposite to \cite{KaNa94} and \cite{Lec03}.
\end{rem}

Let \(\mathcal{W}^{D_n} := \bigsqcup_{\leq 0} \mathcal{W}^{D_n}_l\) denote the free monoid of words on \(\mathcal{A}^{D_n}\), where \(\mathcal{W}^{D_n}_l := (\mathcal{A}^{D_n})^l\).
Define \(\mathcal{W}^{B_n}\) similarly.
They are equipped with a \(D_n\)-crystal and \(B_n\)-crystal structure, respectively by identifying \(\mathbf{a} = (a_l,\dots,a_1)\) with \(a_l \otimes \cdots \otimes a_1\).
The multiplication (concatenation) is denoted by \(*\):
\[
  (a_l,\dots,a_1) * (b_m,\dots,b_1) := (a_l,\dots,a_1,b_m,\dots,b_1).
\]
We often identify a letter \(a\) and the word \((a)\) of length \(1\).
In particular, we understand that
\[
  (a_l,\dots,a_1) * a = (a_l,\dots,a_1,a), \quad a*(a_l,\dots,a_1) = (a,a_l,\dots,a_1).
\]

When \(N \geq 4\), an \(\mathfrak{so}_N\)-crystal \(\mathcal{B}\) has the following natural structure of \(\mathfrak{so}_{N-2}\)-crystal:
\begin{align*}
  &\widetilde{E}'_j b := \widetilde{E}_{j+1} b, \quad \widetilde{F}'_j b := \widetilde{F}_{j+1} b, \\
  &\langle d'_j, \mathsf{wt}'(b) \rangle := \langle d_{j+1}, \mathsf{wt}(b)  \rangle.
\end{align*}
Here, we put the prime symbol on the structure maps for the crystal of type \(\mathfrak{so}_{N-2}\) in order to distinguish them from those for \(\mathfrak{so}_N\).
However, there is no natural structure of \(\mathfrak{so}_{N-1}\)-crystal on \(\mathcal{B}\).

\section{Kashiwara--Nakashima tableaux of type \(D\)}\label{sect_KNT_D}
In this section, we recall the notion of \emph{Kashiwara--Nakashima tableaux of type \(D_n\)}, which is a crystal of type \(D_n\) modelling the finite-dimensional simple \(\mathfrak{so}_{2n}\)-modules.
It was first defined in \cite{KaNa94}, but we employ slightly different description for our purpose.

\subsection{Columns}
Since a Kashiwara--Nakashima tableau is a finite sequence of Kashiwara--Nakashima columns, we first introduce the columns and related notions.

Let \(n \in \mathbb{Z}_{\geq 1}\), and recall from \S \ref{ssect_crystal} the \(D_n\)-crystals \(\mathcal{A}^{D_n}\) and \(\mathcal{W}^{D_n}\).

\begin{defi}[Vacant numbers and occupied numbers]\label{def_vac_occ}
  Let \(\mathbf{a} \in \mathcal{W}^{D_n}\) and \(a,b \in \{ 1,\dots,n \}\).
  \begin{enumerate}
    \item We say that \(a\) is \emph{vacant} in \(\mathbf{a}\) if neither \(a\) nor \(\overline{a}\) occurs in \(\mathbf{a}\):
    \[
      a \notin \mathbf{a} \text{ and } \overline{a} \notin \mathbf{a}.
    \]
    Let \(\mathsf{V}(\mathbf{a})\) denote the set of vacant numbers in \(\mathbf{a}\).
    \item We say that \(b\) is \emph{occupied} in \(\mathbf{a}\) if both \(b\) and \(\overline{b}\) occur in \(\mathbf{a}\):
    \[
      b \in \mathbf{a} \text{ and } \overline{b} \in \mathbf{a}.
    \]
    Let \(\mathsf{O}(\mathbf{a})\) denote the set of occupied numbers in \(\mathbf{a}\).
  \end{enumerate}
\end{defi}

\begin{defi}[Quasi-increasing words]
  A word \(\mathbf{a} = (a_l,\dots,a_1) \in \mathcal{W}^{D_n}\) is said to be \emph{quasi-increasing} (from right to left) if either \(a_r \prec a_{r+1}\) or \((a_r,a_{r+1}) \in \{ (n,\overline{n}),\ (\overline{n}, n) \}\) for all \(r \in \{ 1,\dots,l-1 \}\).
\end{defi}

A quasi-increasing word \(\mathbf{a} = (a_l,\dots,a_1)\) is usually represented by a one-column shaped tableau
\[
  \ytableausetup{boxsize=normal}
  \begin{ytableau}
    a_1 \\
    \vdots \\
    a_l
  \end{ytableau}.
\]
In this paper, we often describe a quasi-increasing word \(\mathbf{a}\) by two rows of beads by putting one bead on the \(a\)-th position (from the left) in the upper row if \(a \in \mathbf{a}\), and by putting one bead on the \(a\)-th position in the lower row if \(\overline{a} \in \mathbf{a}\).
For example, when \(n = 6\), the quasi-increasing word \((\overline{2},\overline{4},\overline{5},5,3,2)\) is described as Figure \ref{fig_quasi_inc_word}.
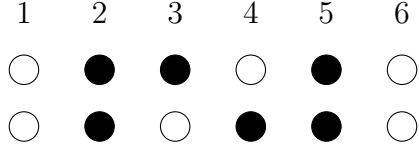
\begin{figure}[h]
  \begin{center}
    \begin{tikzpicture}
      \node at (1,0) {\(1\)};
      \node at (2,0) {\(2\)};
      \node at (3,0) {\(3\)};
      \node at (4,0) {\(4\)};
      \node at (5,0) {\(5\)};
      \node at (6,0) {\(6\)};
      \node[draw, circle] at (1,-0.75) {};
      \node[draw, circle, fill] at (2,-0.75) {};
      \node[draw, circle, fill] at (3,-0.75) {};
      \node[draw, circle] at (4,-0.75) {};
      \node[draw, circle, fill] at (5,-0.75) {};
      \node[draw, circle] at (6,-0.75) {};
      \node[draw, circle] at (6,-1.5) {};
      \node[draw, circle, fill] at (5,-1.5) {};
      \node[draw, circle, fill] at (4,-1.5) {};
      \node[draw, circle] at (3,-1.5) {};
      \node[draw, circle, fill] at (2,-1.5) {};
      \node[draw, circle] at (1,-1.5) {};
    \end{tikzpicture}
  \end{center}
  \caption{A quasi-increasing word}\label{fig_quasi_inc_word}
\end{figure}
Note that when \(n \in \mathsf{O}(\mathbf{a})\), we cannot describe \(\mathbf{a}\) correctly in this way since letters \(n\) and \(\overline{n}\) may alternate.

For each \(a \in \mathcal{A}^{D_n}\), set
\[
  |a| := \begin{cases}
    a & \text{ if } a \preceq n, \\
    b & \text{ if } a = \overline{b} \text{ for some } b \preceq n.
  \end{cases}
\]
Given a word \(\mathbf{a} = (a_l,\dots,a_1) \in \mathcal{W}^{D_n}\), we use the following notations:
\begin{itemize}
  \item \(|\mathbf{a}| := l\): the length of \(\mathbf{a}\),
  \item \(\mathbf{a}_{< a}\): the subword of \(\mathbf{a}\) consisting of \(b \in \mathcal{A}^{D_n}\) such that \(|b| < a\); the subwords \(\mathbf{a}_{\leq a}\), \(\mathbf{a}_{> a}\), \(\mathbf{a}_{\geq a}\) are defined similarly,
  \item \(\mathbf{a}^{\prec a}\): the subword of \(\mathbf{a}\) consisting of \(b \in \mathcal{A}^{D_n}\) such that \(b \prec a\); the subwords \(\mathbf{a}^{\preceq a}\), \(\mathbf{a}^{\succ a}\), \(\mathbf{a}^{\succeq a}\) are defined similarly,
  \item \(\mathbf{a}_n\): the subword of \(\mathbf{a}\) consisting of \(n, \overline{n}\).
\end{itemize}

\begin{ex}[Notations]
  Set \(n := 6\) and \(\mathbf{a} := (\overline{2},\overline{4},\overline{5},6,\overline{6},6,\overline{6},5,3,2)\).
  Then, we have
  \begin{itemize}
    \item \(|\mathbf{a}| = 10\),
    \item \(\mathbf{a}_{< 5} = (\overline{2},\overline{4},3,2)\),
    \item \(\mathbf{a}^{\prec 5} = (3,2)\),
    \item \(\mathbf{a}_n = (6,\overline{6},6,\overline{6})\).
  \end{itemize}
\end{ex}

\begin{defi}[Deep numbers and shallow numbers]\label{def_deep_shallow}
  Let \(\mathbf{a} \in \mathcal{W}^{D_n}\) and \(b \in \{ 1,\dots,n \}\).
  We say that \(b\) is \emph{deep} in \(\mathbf{a}\) if
  \[
    b \in \mathsf{O}(\mathbf{a}) \text{ and } |\mathbf{a}_{\leq b}| > b.
  \]
  Otherwise, we say that \(b\) is \emph{shallow} in \(\mathbf{a}\).
\end{defi}

\begin{ex}[Deep numbers and shallow numbers]
  In the quasi-increasing word of Figure \ref{fig_quasi_inc_word}, the number \(5\) is deep, and the others are shallow.
\end{ex}

\begin{defi}[Signs of quasi-increasing words]
  Let \(\mathbf{a} = (a_l,\dots,a_1) \in \mathcal{W}^{D_n}\) be a quasi-increasing word.
  \begin{enumerate}
    \item The \emph{sign} of \(\mathbf{a}\) is \(\mathsf{e}(\mathbf{a}) \in \{ +,-,0 \}\) defined by
    \[
      \mathsf{e}(\mathbf{a}) :=
      \begin{cases}
        + & \text{ if either \(a_r = n\) for some \(r \equiv n \pmod{2}\)}, \\
        & \text{ or \(a_r = \overline{n}\) for some \(r \not\equiv n \pmod{2}\)}, \\
        - & \text{ if either \(a_r = n\) for some \(r \not\equiv n \pmod{2}\)}, \\
        & \text{ or \(a_r = \overline{n}\) for some \(r \equiv n \pmod{2}\)}, \\
        0 & \text{ if } n \in \mathsf{V}(\mathbf{a}).
      \end{cases}
    \]
    \item The \emph{anti-sign} of \(\mathbf{a}\) is \(\overline{\mathsf{e}}(\mathbf{a}) \in \{ +,-,0 \}\) defined by
    \[
      \overline{\mathsf{e}}(\mathbf{a}) :=
      \begin{cases}
        + & \text{ if } \mathsf{e}(\mathbf{a}) = -, \\
        - & \text{ if } \mathsf{e}(\mathbf{a}) = +, \\
        0 & \text{ if } \mathsf{e}(\mathbf{a}) = 0.
      \end{cases}
    \]
  \end{enumerate}
\end{defi}

\begin{defi}[Kashiwara--Nakashima columns {\cite[(6.3.2)]{KaNa94}}]
  A word \(\mathbf{a} \in \mathcal{W}^{D_n}\) is said to be a \emph{Kashiwara--Nakashima column} if the following conditions are satisfied:
  \begin{description}
    \item[Length] \(|\mathbf{a}| \leq n\),
    \item[Quasi-increasing] \(\mathbf{a}\) is quasi-increasing,
    \item[Shallowness] every number is shallow in \(\mathbf{a}\), i.e.,
    \[
      |\mathbf{a}_{\leq b}| \leq b \quad \text{ for all } b \in \mathsf{O}(\mathbf{a}).
    \]
  \end{description}
  Let
  \[
    \mathsf{Col}^{D_n}_{l}
  \]
  denote the set of Kashiwara--Nakashima columns of type \(D_n\) and length \(l\), and set
  \[
    \mathsf{Col}^{D_n} := \bigsqcup_{l=0}^n \mathsf{Col}^{D_n}_{l}.
  \]
  Also for each \(e \in \{ +,-,0 \}\), set
  \[
    \mathsf{Col}^{D_n}_{l,e} := \{ \mathbf{a} \in \mathsf{Col}^{D_n}_{l} \mid \mathsf{e}(\mathbf{a}) = e \}.
  \]
\end{defi}

\begin{ex}[Kashiwara--Nakashima columns]
  Let \(n \geq 3\).
  The quasi-increasing words of length at most \(3\) that are not Kashiwara--Nakashima columns are the following:
  \begin{align*}
    \begin{ytableau}
      1 \\
      *(lightgray) 1
    \end{ytableau},\ 
    \begin{ytableau}
      1 \\
      a \\
      *(lightgray) 1
    \end{ytableau},\ 
    \begin{ytableau}
      1 \\
      2 \\
      *(lightgray) 2
    \end{ytableau},\ 
    \begin{ytableau}
      2 \\
      *(lightgray) 2 \\
      *(lightgray) 1
    \end{ytableau},
  \end{align*}
  for \(1 \prec a \prec \overline{1}\).
  Here, we represent a box with entry \(\overline{a}\) for some \(a \preceq n\) by a shaded box with entry \(a\).
\end{ex}

\begin{rem}[Crystal structure of columns]
  The columns \(\mathsf{Col}^{D_n}_{l}\) and \(\mathsf{Col}^{D_n}_{n,\pm}\) are subcrystals of \(\mathcal{W}^{D_n}\); see \cite[Propositions 6.3.2, 6.3.4]{KaNa94}.
\end{rem}

\begin{defi}[Spin columns {\cite[(6.4.6)]{KaNa94}}]
  A column \(\mathbf{a} = (a_n,\dots,a_1) \in \mathsf{Col}^{D_n}_{n}\) of length \(n\) is said to be a \emph{spin column} if
  \[
    \{ |a_n|,\dots,|a_1| \} = \{ 1,\dots,n \}.
  \]
  Let
  \[
    \mathsf{Col}^{D_n}_{\mathsf{spin}}
  \]
  denote the set of spin columns, and set
  \[
    \mathsf{Col}^{D_n}_{\mathsf{spin}, \pm} := \mathsf{Col}^{D_n}_{\mathsf{spin}} \cap \mathsf{Col}^{D_n}_{n, \pm}.
  \]
\end{defi}

\begin{rem}[Crystal structure of spin columns]
  The spin columns \(\mathsf{Col}^{D_n}_{\mathsf{spin},\pm}\) have structures of crystal of type \(D_n\), but they are not subcrystals of \(\mathcal{W}^{D_n}\); see \cite[\S 6.4]{KaNa94}.
\end{rem}

\subsection{Tableaux}
\begin{defi}[Adjacent condition {\cite[\S\S 6.5--6.6]{KaNa94}}]
  A pair of two Kashiwara--Nakashima columns \(\mathbf{a} = (a_k,\dots,a_1)\) and \(\mathbf{b} = (b_l,\dots,b_1)\) are said to satisfy the \emph{adjacent condition} if the following hold:
  \begin{description}
    \item[Length] \(|\mathbf{a}| \geq |\mathbf{b}| > 0\),
    \item[Weakly Increasing] \(a_r \preceq b_r\) for all \(1 \leq r \leq |\mathbf{b}|\),
    \item[\((a,b)\)-Configurations] \((s-r) + (u-t) < b-a\) for all \(1 \preceq a \preceq b \prec n\) and \(1 \leq r \leq s < t \leq u \leq |\mathbf{b}|\) satisfying one of the following:
    \begin{itemize}
      \item \(a_r = a\), \(b_s = b\), \(b_t = \overline{b}\), and \(b_u = \overline{a}\),
      \item \(a_r = a\), \(a_s = b\), \(a_t = \overline{b}\), and \(b_u = \overline{a}\),
    \end{itemize}
    \item[\((a,n)\)-Configurations] \(u-r-1 < n-a\) for all \(1 \preceq a \prec n\) and \(1 \leq r < s < u \leq |\mathbf{b}|\) satisfying one of the following:
    \begin{itemize}
      \item \(a_r = a\), \(b_s = n\), \(b_{s+1} = \overline{n}\), and \(b_u = \overline{a}\),
      \item \(a_r = a\), \(b_s = \overline{n}\), \(b_{s+1} = n\), and \(b_u = \overline{a}\),
      \item \(a_r = a\), \(a_s = n\), \(a_{s+1} = \overline{n}\), and \(b_u = \overline{a}\),
      \item \(a_r = a\), \(a_s = \overline{n}\), \(a_{s+1} = n\), and \(b_u = \overline{a}\),
    \end{itemize}
    \item[\((n,n)\)-Configurations] there are no \(1 \leq s < t \leq |\mathbf{b}|\) satisfying the one of the following:
    \begin{itemize}
      \item \(a_s = n\) and \(b_t = n\),
      \item \(a_s = n\) and \(b_t = \overline{n}\),
      \item \(a_s = \overline{n}\) and \(b_t = n\),
      \item \(a_s = \overline{n}\) and \(b_t = \overline{n}\),
    \end{itemize}
    \item[\(a\)-Odd-configurations] \(u-r < n-a\) for all \(1 \preceq a \prec n\) and \(1 \leq r \leq s < t \leq u \leq |\mathbf{b}|\) such that \(t-s+1\) is odd and they satisfy one of the following:
    \begin{itemize}
      \item \(a_r = a\), \(b_s = n\), \(a_t = \overline{n}\), and \(b_u = \overline{a}\),
      \item \(a_r = a\), \(b_s = \overline{n}\), \(a_t = n\), and \(b_u = \overline{a}\),
    \end{itemize}
    \item[\(a\)-Even-configurations] \(u-r < n-a\) for all \(1 \preceq a \prec n\) and \(1 \leq r \leq s < t \leq u \leq |\mathbf{b}|\) such that \(t-s+1\) is even and they satisfy one of the following:
    \begin{itemize}
      \item \(a_r = a\), \(b_s = n\), \(a_t = n\), and \(b_u = \overline{a}\),
      \item \(a_r = a\), \(b_s = \overline{n}\), \(a_t = \overline{n}\), and \(b_u = \overline{a}\),
    \end{itemize}
  \end{description}
  We write \(\mathbf{a} \triangleleft \mathbf{b}\) to mean that \((\mathbf{a}, \mathbf{b})\) satisfy the adjacent condition.
\end{defi}

\begin{defi}[Kashiwara--Nakashima tableaux {\cite[\S 6.7]{KaNa94}}]
  Let \(\lambda \in X^+_{2n}\).
  A \emph{Kashiwara--Nakashima tableau} of type \(D_n\) and shape \(\lambda\) is a finite sequence \(T = (\mathbf{a}^1,\dots,\mathbf{a}^m)\) of Kashiwara--Nakashima columns satisfying the following:
  \begin{itemize}
    \item \(\mathbf{a}^2,\dots,\mathbf{a}^m \in \mathsf{Col}^{D_n}\),
    \item \(\mathbf{a}^1 \in \mathsf{Col}^{D_n}_{\mathsf{spin}}\) if \(\mathsf{s}(\lambda) = \frac{1}{2}\),
    \item \(\mathbf{a}^1 \triangleleft \cdots \triangleleft \mathbf{a}^m\).
    \item \((|\mathbf{a}^1|,\dots,|\mathbf{a}^m|) = \mathsf{par}(\lambda)'\),
    \item \(\mathsf{e}(\mathbf{a}^r) = \mathsf{e}(\lambda)\) for all \(r \in \{ 1,\dots,m \}\) such that \(|\mathbf{a}^r| = n\).
  \end{itemize}
  see \S \ref{ssect_part} for notations concerning partitions.
  Let
  \[
    \mathsf{KNT}^{D_n}(\lambda)
  \]
  denote the set of Kashiwara--Nakashima tableaux of type \(D_n\) and shape \(\lambda\), and set
  \[
    \mathsf{KNT}^{D_n}_{0} := \bigsqcup_{\lambda \in X^+_{2n,0}} \mathsf{KNT}^{D_n}(\lambda).
  \]
  For each \(T \in \mathsf{KNT}^{D_n}(\lambda)\), set
  \[
    \mathsf{sh}(T) := \lambda,
  \]
  and call it the \emph{shape} of \(T\).
\end{defi}

\begin{rem}[Crystal structure of tableaux]\label{rem_crystal_knt_D}
  The tableaux \(\mathsf{KNT}^{D_n}(\lambda)\) is a subcrystal of a tensor product of columns; a tableau \(T = (\mathbf{a}^1,\dots,\mathbf{a}^m)\) is identified with \(\mathbf{a}^1 \otimes \dots \otimes \mathbf{a}^m\).
\end{rem}

\subsection{Column-insertion}\label{ssect_col_ins}
Each connected component of the crystal \(\mathcal{W}^{D_n}\) is isomorphic to some \(\mathsf{KNT}^{D_n}(\lambda)\) with \(\lambda \in X^+_{2n,0}\) (\cite[\S 2.2]{Lec03}).
For each \(\mathbf{a} \in \mathcal{W}^{D_n}\), let \(\mathsf{P}(\mathbf{a})\) denote the corresponding tableau, and \(\mathsf{Q}(\mathbf{a})\) the connected component containing \(\mathbf{a}\).
The connected components of \(\mathcal{W}^{D_n}\) are parametrized by \emph{oscillating tableaux} (\cite[Definition 3.4.1]{Lec03}).
Hence, let \(\mathsf{OT}(\lambda)\) denote the set of connected components of \(\mathcal{W}^{D_n}\) isomorphic to \(\mathsf{KNT}^{D_n}(\lambda)\).
Then, the assignment \(\mathbf{a} \mapsto (\mathsf{P}(\mathbf{a}), \mathsf{Q}(\mathbf{a}))\) gives rise to an isomorphism
\[
  \mathsf{LRS} \colon \mathcal{W}^{D_n} \to \bigsqcup_{\lambda \in X^+_{2n,0}} \mathsf{KNT}^{D_n}(\lambda) \times \mathsf{OT}(\lambda)
\]
of crystals; the crystal structure of the right hand-side is given by only the first factors.
We call this isomorphism the \emph{Lecouvey--Robinson--Schensted correspondence}.
An algorithm calculating \(\mathsf{LRS}(\mathbf{a})\) can be found in \cite{Lec03}.

For each \(\lambda \in X^+_{2n,0}\), define a map
\[
  \mathsf{w}_\mathsf{col} \colon \mathsf{KNT}^{D_n}(\lambda) \to \mathcal{W}^{D_n};\ T = (\mathbf{a}^1,\dots,\mathbf{a}^m) \mapsto \mathbf{a}^1 * \cdots * \mathbf{a}^m.
\]
Clearly, it is an injective crystal morphism.

Consider the composite
\[
  * \colon \mathsf{KNT}^{D_n}_{0} \otimes \mathsf{KNT}^{D_n}_{0} \xrightarrow{\mathsf{w}_\mathsf{col} \otimes \mathsf{w}_\mathsf{col}} \mathcal{W}^{D_n} \otimes \mathcal{W}^{D_n} \xrightarrow{*} \mathcal{W}^{D_n} \xrightarrow{\mathsf{P}} \mathsf{KNT}^{D_n}_{0}
\]
of crystal morphisms.
In this way, we can transform two tableaux \((T,S)\) into a new one \(T*S\) while preserving the crystal structure.

\section{Manipulation of tableaux of type \(D\)}\label{sect_manipulation_D}
In this section, we introduce the notion of \emph{Kashiwara--Nakashima tableaux of type \(D\mathrm{II}_n\)} and various manipulations.
They will be related to representation theory of quantum symmetric pairs of type \(D\mathrm{II}_n\) in \S\ref{sect_qsp_DII}.

\subsection{Columns}
Since a Kashiwara--Nakashima tableau is a sequence of columns, we first focus on columns.

\begin{prop}[\(1\) is not occupied]\label{prop_non_occ_1_kn_col_d}
  Let \(\mathbf{a} \in \mathsf{Col}^{D_n}\).
  Then, we have
  \[
    1 \notin \mathsf{O}(\mathbf{a}).
  \]
\end{prop}
\begin{proof}
  Let us prove the contraposition.
  If \(1 \in \mathsf{O}(\mathbf{a})\), then we have
  \[
    |\mathbf{a}_{\leq 1}| = 2 > 1.
  \]
  This implies that \(\mathbf{a} \notin \mathsf{Col}^{D_n}\).
  Thus, we complete the proof.
\end{proof}

\begin{defi}[Kashiwara--Nakashima columns of type \(D\mathrm{II}\)]\label{def_col_DII}
  A column \(\mathbf{a} \in \mathsf{Col}^{D_n}\) is said to be of \emph{type \(D\mathrm{II}_n\)} if it satisfies the following:
  \begin{description}
    \item[Length] \(|\mathbf{a}| < n\),
    \item[Exclusion of \(1\)] \(1 \notin \mathbf{a}\),
    \item[Strict Shallowness] \(|(\mathbf{a}_{> 1})_{\leq b}| < b\) for all \(b \in \mathsf{O}(\mathbf{a})\).
  \end{description}

  The set of Kashiwara--Nakashima columns of type \(D\mathrm{II}_n\) is denoted by
  \[
    \mathsf{Col}^{D\mathrm{II}_n}.
  \]
  For each \(l \in \mathbb{Z}_{\geq 0}\), set
  \[
    \mathsf{Col}^{D\mathrm{II}_n}_{l} := \mathsf{Col}^{D\mathrm{II}_n} \cap \mathsf{Col}^{D_n}_{l}.
  \]
\end{defi}

\begin{ex}[Kashiwara--Nakashima columns of type \(D\mathrm{II}\)]
  Let \(n \geq 4\).
  The Kashiwara--Nakashima columns of type \(D_n\) of length at most \(3\) that are not of type \(D\mathrm{II}_n\) are the following:
  \begin{align*}
    \begin{ytableau}
      1
    \end{ytableau},\ 
    \begin{ytableau}
      1 \\
      a
    \end{ytableau},\ 
    \begin{ytableau}
      2 \\
      *(lightgray) 2
    \end{ytableau},\ 
    \begin{ytableau}
      1 \\
      b \\
      c
    \end{ytableau},\ 
    \begin{ytableau}
      2 \\
      d \\
      *(lightgray) 2
    \end{ytableau},
  \end{align*}
  for \(1 \prec a \prec \overline{1}\), \(1 \prec b \not\succeq c \prec \overline{1}\), \(2 \prec d \prec \overline{2}\).
\end{ex}

\begin{prop}[Entry \(\overline{1}\)]\label{prop_entry_1b_DII}
  Let \(\mathbf{a} \in \mathsf{Col}^{D_n}\).
  If \(|\mathbf{a}| < n\) and \(\overline{1} \in \mathbf{a}\), then \(\mathbf{a} \in \mathsf{Col}^{D\mathrm{II}_n}\).
\end{prop}
\begin{proof}
  By Proposition \ref{prop_non_occ_1_kn_col_d}, the assumption \(\overline{1} \in \mathbf{a}\) implies that
  \[
    1 \notin \mathbf{a}.
  \]
  For each \(b \in \mathsf{O}(\mathbf{a})\), we have
  \[
    |(\mathbf{a}_{> 1})_{\leq b}| = |\mathbf{a}_{\leq b}|-1 \leq b-1 < b;
  \]
  the first inequality follows from the assumption that \(\mathbf{a} \in \mathsf{Col}^{D_n}\).
  Hence, the assertion follows.
\end{proof}

\subsection{Connecting Maps}
The aim of this subsection is to define the connecting maps (Definition \ref{def_conn_map_DII}) and to prove its bijectivity (Theorem \ref{thm_bij_conn_map_DII}).
The connecting maps relate columns of different length.
A representation theoretic interpretation will be given in \S\ref{ssect_conn_hom_D}.

For a quasi-increasing word \(\mathbf{a}\), we use the following notations:
\begin{itemize}
  \item \(\mathbf{a} \Leftarrow a\): the unique quasi-increasing word of length \(|\mathbf{a}|+2\) obtained from \(\mathbf{a}\) by inserting \(a\) and \(\overline{a}\) in suitable positions (if \(a \in \mathsf{V}(\mathbf{a}) \setminus \{ n \}\)),
  \item \(\mathbf{a} \Leftarrow_{\pm} n\): the unique quasi-increasing word of length \(|\mathbf{a}|+2\) and sign \(\pm\) obtained from \(\mathbf{a}\) by inserting one \(n\) and one \(\overline{n}\) in suitable positions (if possible),
  \item \(\mathbf{a} \Rightarrow b\): the subword of \(\mathbf{a}\) consisting of letters other than \(b,\overline{b}\) (if \(b \in \mathsf{O}(\mathbf{a}) \setminus \{ n \}\)),
  \item \(\mathbf{a} \Rightarrow n\): the subword of \(\mathbf{a}\) obtained from \(\mathbf{a}\) by removing the first two letters (read from right to left) in \(\mathbf{a}_n\) (if \(n \in \mathsf{O}(\mathbf{a})\)),
\end{itemize}
The transformation \(\mathbf{a} \mapsto \mathbf{a} \Leftarrow_\pm n\) is possible for example when \(n \in \mathsf{V}(\mathbf{a})\) or \(\mathsf{e}(\mathbf{a}) = \pm\).

\begin{defi}[Weakly deep numbers]
  Let \(\mathbf{a} \in \mathcal{W}^{D_n}\) be a quasi-increasing word.
  We say that \(b \in \mathsf{O}(\mathbf{a})\) is \emph{weakly deep} if
  \[
    |\mathbf{a}_{\leq b}| \geq b.
  \]
  Let
  \[
    \mathsf{WD}(\mathbf{a})
  \]
  denote the set of weakly deep numbers in \(\mathbf{a}\).
  When \(\mathsf{WD}(\mathbf{a}) \neq \emptyset\), the \emph{least weakly deep number} is denoted by \(\mathsf{lwd}(\mathbf{a})\):
  \[
    \mathsf{lwd}(\mathbf{a}) := \min \mathsf{WD}(\mathbf{a}).
  \]
\end{defi}

\begin{lem}[Existence of weakly deep numbers]\label{lem_exist_wd_D}
  Let \(\mathbf{a} \in \mathsf{Col}^{D_n} \setminus \mathsf{Col}^{D\mathrm{II}_n}\).
  If \(1 \in \mathsf{V}(\mathbf{a})\), then
  \[
    \mathsf{WD}(\mathbf{a}) \neq \emptyset.
  \]
\end{lem}
\begin{proof}
  By the assumption that \(1 \in \mathsf{V}(\mathbf{a})\), we have \(\mathbf{a}_{> 1} = \mathbf{a}\).
  Since \(\mathbf{a} \notin \mathsf{Col}^{D\mathrm{II}_n}\), there exists \(b \in \mathbf{O}(\mathbf{a})\) such that \(|\mathbf{a}_{\leq b}| = |(\mathbf{a}_{> 1})_{\leq b}| \geq b\).
  This implies that \(b \in \mathsf{WD}(\mathbf{a})\).
  Hence, the assertion follows.
\end{proof}

\begin{prop}[Inserting \(1\)]\label{prop_ins_1_DII}
  Let \(l \in \{ 0,\dots,n-1 \}\) and \(\mathbf{a} \in \mathsf{Col}^{D_n}_{l}\).
  If \(1 \in \mathsf{V}(\mathbf{a})\), then the following hold\textup{:}
  \begin{enumerate}
    \item\label{item_nonDII_ins_1_DII} If \(\mathbf{a} \notin \mathsf{Col}^{D\mathrm{II}_n}_{l}\), then \(\mathbf{a}*1 \notin \mathsf{Col}^{D_n}_{l+1}\),
    \item\label{item_DII_ins_1_DII} If \(\mathbf{a} \in \mathsf{Col}^{D\mathrm{II}_n}_{l}\), then \(\mathbf{a}*1 \in \mathsf{Col}^{D_n}_{l+1, \overline{\mathsf{e}}(\mathbf{a})} \setminus \mathsf{Col}^{D\mathrm{II}_n}_{l+1}\),
    \item\label{item_inv_ins_1_DII} \((\mathbf{a}*1)_{> 1} = \mathbf{a}\).
  \end{enumerate}
\end{prop}
\begin{proof}
  Clearly, we have \(\mathsf{O}(\mathbf{a}*1) = \mathsf{O}(\mathbf{a})\) and
  \[
    |(\mathbf{a}*1)_{\leq b}| = |\mathbf{a}_{\leq b}|+1 \quad \text{ for all } b \in \mathsf{O}(\mathbf{a}).
  \]
  \begin{enumerate}
    \item By Lemma \ref{lem_exist_wd_D}, there exists \(b \in \mathsf{WD}(\mathbf{a})\).
    Then, we have \(|(\mathbf{a}*1)_{\leq b}| \geq b+1\).
    This implies that \(\mathbf{a}*1 \notin \mathsf{Col}^{D_n}\).
    \item By the assumption, we have \(|(\mathbf{a}*1)_{\leq b}| < b+1\) for all \(b \in \mathsf{O}(\mathbf{a}*1)\).
    This implies that \(\mathbf{a}*1 \in \mathsf{Col}^{D_n}_{l+1}\).
    It is clear that \(\mathsf{e}(\mathbf{a}*1) = \overline{\mathsf{e}}(\mathbf{a})\).
    Also, we have \(\mathbf{a}*1 \notin \mathsf{Col}^{D\mathrm{II}_n}\) because \(1 \in \mathbf{a}*1\).
    \item The assertion is obvious.
  \end{enumerate}
\end{proof}

\begin{prop}[Inserting \(\overline{1}\)]\label{prop_ins_1b_DII}
  Let \(\mathbf{a} \in \mathsf{Col}^{D_n}_{n-1}\).
  If \(1 \in \mathsf{V}(\mathbf{a})\), then the following hold\textup{:}
  \begin{enumerate}
    \item\label{item_nonDII_ins_1b_DII} If \(\mathbf{a} \notin \mathsf{Col}^{D\mathrm{II}_n}_{n-1}\), then \(\overline{1}*\mathbf{a} \notin \mathsf{Col}^{D_n}_{n}\),
    \item\label{item_DII_ins_1b_DII} If \(\mathbf{a} \in \mathsf{Col}^{D\mathrm{II}_n}_{n-1}\), then \(\overline{1}*\mathbf{a} \in \mathsf{Col}^{D_n}_{n, \mathsf{e}(\mathbf{a})}\),
    \item\label{item_inv_ins_1b_DII} \((\overline{1}*\mathbf{a})_{> 1} = \mathbf{a}\).
  \end{enumerate}
\end{prop}
\begin{proof}
  One can prove the assertion in a similar way to Proposition \ref{prop_ins_1_DII}.
\end{proof}

\begin{lem}[Vacant numbers]\label{lem_vac_num_DII}
  Let \(\mathbf{a} \in \mathsf{Col}^{D_n}\).
  Assume that \(\overline{1} \in \mathbf{a}\), and set \(\mathbf{b} := \mathbf{a}_{> 1}\).
  For each \(a \in \mathsf{V}(\mathbf{a})\), we have
  \[
    |\mathbf{b}_{< a}| \leq a-2.
  \]
\end{lem}
\begin{proof}
  Assume contrary; there exists \(a \in \mathsf{V}(\mathbf{a})\) such that \(|\mathbf{b}_{< a}| > a-2 = \#\{ 2,\dots,a-1 \}\).
  Then, we have \(\mathsf{O}(\mathbf{b}_{< a}) \neq \emptyset\).
  Set
  \[
    b := \max \mathsf{O}(\mathbf{b}_{< a}).
  \]
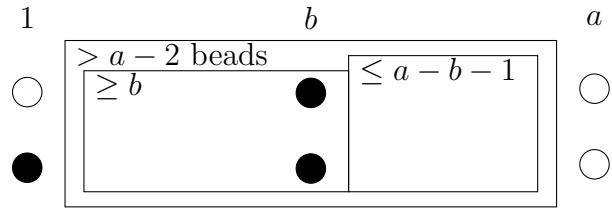
  \begin{center}
    \begin{tikzpicture}[xscale=2.5]
      \node (1) at (0,0) {\(1\)};
      \node (b) at (1.5,0) {\(b\)};
      \node (a) at (3,0) {\(a\)};

      \node[draw, circle] [below=0.5cm of 1] {};
      \node[draw, circle, fill] [below=1.5cm of 1] {};
      \node[draw, circle] [below=0.5cm of a] {};
      \node[draw, circle] [below=1.5cm of a] {};
      \node[draw, circle, fill] [below=0.5cm of b] {};
      \node[draw, circle, fill] [below=1.5cm of b] {};

      \draw (0.2,-0.3) rectangle (3-0.2,-2.5);
      \node[anchor=west] at (0.2,-0.5) {\(> a-2\) beads};
      \draw (1.5+0.2,-0.5) rectangle (3-0.3,-2.3);
      \node[anchor=west] at (1.7,-0.7) {\(\leq a-b-1\)};
      \draw (0.3,-0.7) rectangle (1.5+0.2,-2.3);
      \node[anchor=west] at (0.3,-0.9) {\(\geq b\)};
    \end{tikzpicture}
  \end{center}
  By the maximality of \(b\), it holds that
  \[
    |(\mathbf{b}_{< a})_{> b}| \leq a-b-1.
  \]
  Hence,
  \[
    |\mathbf{b}_{\leq b}| = |\mathbf{b}_{< a}| - |(\mathbf{b}_{< a})_{> b}| \geq (a-1)-(a-b-1) = b.
  \]
  Therefore
  \[
    |\mathbf{a}_{\leq b}| = |\mathbf{b}_{\leq b}|+1 > b.
  \]
  This contradicts that \(\mathbf{a} \in \mathsf{Col}^{D_n}\).
  Thus, we complete the proof.
\end{proof}

\begin{defi}[Properly vacant numbers]
  Let \(\mathbf{a} \in \mathsf{Col}^{D_n}\).
  Assume that \(\overline{1} \in \mathbf{a}\), and set
  \[
    \mathbf{b} := \mathbf{a}_{> 1}.
  \]
  We say that \(a \in \mathsf{V}(\mathbf{a})\) is \emph{proper} if
  \[
    |\mathbf{b}_{< a}| = a-2.
  \]
  Let
  \[
    \mathsf{PV}(\mathbf{a})
  \]
  denote the set of properly vacant numbers in \(\mathbf{a}\).
  When \(\mathsf{PV}(\mathbf{a}) \neq \emptyset\), the \emph{greatest properly vacant number} is denoted by \(\mathsf{gpv}(\mathbf{a})\):
  \[
    \mathsf{gpv}(\mathbf{a}) := \max \mathsf{PV}(\mathbf{a}).
  \]
\end{defi}

\begin{lem}[Existence of properly vacant numbers]\label{lem_exsit_pv_D}
  Let \(\mathbf{a} \in \mathsf{Col}^{D_n}\).
  If \(|\mathbf{a}| < n\) and \(\overline{1} \in \mathbf{a}\), then, we have
  \[
    \mathsf{PV}(\mathbf{a}) \neq \emptyset.
  \]
\end{lem}
\begin{proof}
  Since \(|\mathbf{a}| < n\), we have \(\mathsf{V}(\mathbf{a}) \neq \emptyset\).
  Set
  \[
    a := \min \mathsf{V}(\mathbf{a}).
  \]
  Let us show that \(a \in \mathsf{PV}(\mathbf{a})\).
  Since \(\overline{1} \in \mathbf{a}\), we have \(a > 1\).
  Set
  \[
    \mathbf{b} := \mathbf{a}_{> 1}.
  \]
  The minimality of \(a\) implies that
  \[
    |\mathbf{b}_{< a}| \geq a-2.
  \]
  By Lemma \ref{lem_vac_num_DII}, this inequality must be equality.
  Hence, the assertion follows.
\end{proof}

\begin{prop}[Inserting properly vacant numbers]\label{prop_ins_pv_DII}
  Let \(l \in \{ 0,\dots,n-1 \}\) and \(\mathbf{a} \in \mathsf{Col}^{D_n}_{l}\).
  Assume that \(\overline{1} \in \mathbf{a}\), and set
  \[
    \mathbf{b} := \mathbf{a}_{> 1}.
  \]
  Let \(a \in \mathsf{PV}(\mathbf{a})\)\textup{;} see Lemma \textup{\ref{lem_exsit_pv_D}}.
  Then, the following hold\textup{:}
  \begin{enumerate}
    \item\label{item_ins_pv_DII} If \(a < \mathsf{gpv}(\mathbf{a})\), then \(\mathbf{b} \Leftarrow a \notin \mathsf{Col}^{D_n}_{l+1}\).
    \item\label{item_ins_gpv_DII} If \(a = \mathsf{gpv}(\mathbf{a}) < n\), then
    \begin{itemize}
      \item \(\mathbf{b} \Leftarrow a \in \mathsf{Col}^{D_n}_{l+1, \overline{\mathsf{e}}(\mathbf{a})} \setminus \mathsf{Col}^{D\mathrm{II}_n}_{l+1}\),
      \item \(a = \mathsf{lwd}(\mathbf{b} \Leftarrow a)\),
      \item \(\overline{1}*((\mathbf{b} \Leftarrow a) \Rightarrow a) = \mathbf{a}\).
    \end{itemize}
    \item\label{item_ins_gpv_max_len_DII} If \(a = \mathsf{gpv}(\mathbf{a}) = n\), then
    \begin{itemize}
      \item \(l = n-1\)
      \item \(\mathbf{b} \Leftarrow_\pm n \in \mathsf{Col}^{D_n}_{n,\pm}\),
      \item \(n = \mathsf{lwd}(\mathbf{b} \Leftarrow_\pm n)\),
      \item \(\overline{1}*((\mathbf{b} \Leftarrow_\pm n) \Rightarrow n) = \mathbf{a}\)
    \end{itemize}
  \end{enumerate}
\end{prop}
\begin{proof}
  \hfill
  \begin{enumerate}
    \item Set \(a' := \mathsf{gpv}(\mathbf{a})\).
    Then, we have
    \[
      |(\mathbf{b}_{< a'})_{> a}| = |\mathbf{b}_{< a'}|-|\mathbf{b}_{< a}| = (a'-2)-(a-2) = a'-a > \#\{ a+1,\dots,a'-1 \}.
    \]
    This implies that \(\mathsf{O}((\mathbf{b}_{< a'})_{> a}) \neq \emptyset\).
    Set
    \[
      b := \max \mathsf{O}((\mathbf{b}_{< a'})_{> a}).
    \]
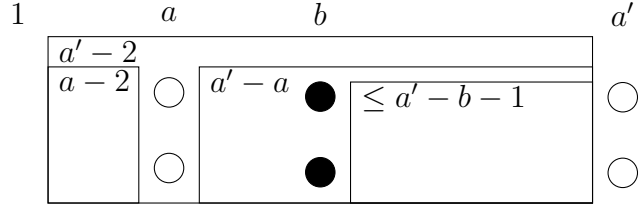
    \begin{center}
      \begin{tikzpicture}[xscale=2]
        \begin{scope}
          \node (1) at (0,0) {\(1\)};
          \node (a) at (1,0) {\(a\)};
          \node (b) at (2,0) {\(b\)};
          \node (ad) at (4,0) {\(a'\)};

          \node[draw, circle] [below=0.6cm of a] {};
          \node[draw, circle] [below=1.6cm of a] {};
          \node[draw, circle, fill] [below=0.6cm of b] {};
          \node[draw, circle, fill] [below=1.6cm of b] {};
          \node[draw, circle] [below=0.6cm of ad] {};
          \node[draw, circle] [below=1.6cm of ad] {};

          \draw (0.2,-0.3) rectangle (4-0.2,-2.5);
          \node[anchor=west] at (0.2,-0.5) {\(a'-2\)};
          \draw (0.2,-0.7) rectangle (1-0.2,-2.5);
          \node[anchor=west] at (0.2,-0.9) {\(a-2\)};
          \draw (1+0.2,-0.7) rectangle (4-0.2,-2.5);
          \node[anchor=west] at (1+0.2,-0.9) {\(a'-a\)};
          \draw (2+0.2,-0.9) rectangle (4-0.2,-2.5);
          \node[anchor=west] at (2+0.2,-1.1) {\(\leq a'-b-1\)};
        \end{scope}
      \end{tikzpicture}
    \end{center}
    By the maximality of \(b\), we have
    \[
      |(\mathbf{b}_{< a'})_{> b}| \leq a'-b-1.
    \]
    This implies that
    \[
      |\mathbf{b}_{\leq b}| = |\mathbf{b}_{< a'}| - |(\mathbf{b}_{< a'})_{> b}| \geq (a'-2)-(a'-b-1) = b-1.
    \]
    Since \(a < b\), we obtain
    \[
      |(\mathbf{b} \Leftarrow a)_{\leq b}| = |\mathbf{b}_{\leq b}|+2 > b.
    \]
    Hence, the assertion follows.
    \item We have \(\mathsf{O}(\mathbf{b} \Leftarrow a) = \mathsf{O}(\mathbf{a}) \sqcup \{ a \}\), and
    \begin{align}\label{eq_estimate_b_ins_a}
      |(\mathbf{b} \Leftarrow a)_{\leq b}| =
      \begin{cases}
        |\mathbf{a}_{\leq b}|-1 & \text{ if } b < a, \\
        a & \text{ if } b = a, \\
        |\mathbf{a}_{\leq b}|+1 & \text{ if } b > a,
      \end{cases}
    \end{align}
    for all \(b \in \mathsf{O}(\mathbf{b} \Leftarrow a)\).

    In order to show that \(\mathbf{b} \Leftarrow a \in \mathsf{Col}^{D_n}\), assume contrary.
    Then, there exists \(b \in \mathsf{O}(\mathbf{b} \Leftarrow a)\) such that \(|(\mathbf{b} \Leftarrow a)_{\leq b}| > b\).
    By equation \eqref{eq_estimate_b_ins_a}, this can happen only if \(b > a\) and \(|\mathbf{b}_{\leq b}| = b\).
    Hence, we have
    \[
      |\mathbf{b}_{> b}| = |\mathbf{b}|-|\mathbf{b}_{\leq b}| = (l-1)-(b-1) = l-b < n-b = \#\{ b+1,\dots,n \}.
    \]
    This implies that \(\mathsf{V}(\mathbf{b}) \cap \{ b+1,\dots,n \} \neq \emptyset\).
    Set
    \[
      a' := \min (\mathsf{V}(\mathbf{b}) \cap \{ b+1,\dots,n \})
    \]
    \begin{center}
      \begin{tikzpicture}[xscale=2]
        \node (1) at (0,0) {\(1\)};
        \node (a) at (1,0) {\(a\)};
        \node (b) at (2,0) {\(b\)};
        \node (ad) at (4,0) {\(a'\)};
        \node (n) at (4.7,0) {\(n\)};

        \node[draw, circle] [below=0.6cm of a] {};
        \node[draw, circle] [below=1.6cm of a] {};
        \node[draw, circle, fill] [below=0.6cm of b] {};
        \node[draw, circle, fill] [below=1.6cm of b] {};
        \node[draw, circle] [below=0.6cm of ad] {};
        \node[draw, circle] [below=1.6cm of ad] {};

        \draw (0.2,-0.3) rectangle (2+0.2,-2.5);
        \node[anchor=west] at (0.2,-0.5) {\(b-1\)};
        \draw (2+0.2,-0.3) rectangle (5-0.2,-2.5);
        \node[anchor=west] at (2+0.2,-0.5) {\(l-b\)};
        \draw (2+0.2,-0.7) rectangle (4-0.2,-2.5);
        \node[anchor=west] at (2+0.2,-0.9) {\(\geq a'-b-1\)};
      \end{tikzpicture}
    \end{center}
    By the minimality of \(a'\), we have
    \[
      |(\mathbf{b}_{< a'})_{> b}| \geq a'-b-1.
    \]
    Hence,
    \[
      |\mathbf{b}_{< a'}| = |(\mathbf{b}_{< a'})_{> b}|+|\mathbf{b}_{\leq b}| \geq (a'-b-1)+(b-1) = a'-2.
    \]
    This, together with Lemma \ref{lem_vac_num_DII}, implies that
    \[
      |\mathbf{b}_{< a'}| = a'-2.
    \]
    Therefore, we obtain
    \[
      a' \in \mathsf{PV}(\mathbf{a}).
    \]
    However, this contradicts the maximality of \(a\).
    Thus, we have proved that
    \[
      \mathbf{b} \Leftarrow a \in \mathsf{Col}^{D_n}_{l+1}.
    \]

    It is obvious that
    \[
      \mathsf{e}(\mathbf{b} \Leftarrow a) = \overline{\mathsf{e}}(\mathbf{a}).
    \]
    Also, we have \(\mathbf{b} \Leftarrow a \notin \mathsf{Col}^{D\mathrm{II}_n}\) because
    \[
      |(\mathbf{b} \Leftarrow a)_{\leq a}| = a.
    \]

    Let us show that \(a = \mathsf{lwd}(\mathbf{b} \Leftarrow a)\).
    Assume contrary that there exists \(b \in \mathsf{WD}(\mathbf{b} \Leftarrow a)\) such that \(b < a\).
    Then, we have \(b \in \mathsf{O}(\mathbf{a})\) and
    \[
      |\mathbf{a}_{\leq b}| = |(\mathbf{b} \Leftarrow a)_{\leq b}| + 1 = b+1.
    \]
    This contradicts that \(\mathbf{a} \in \mathsf{Col}^{D_n}\).

    Now, the remaining assertion is clear.
    \item Noting that the equality \(\mathsf{gpv}(\mathbf{a}) = n\) holds only if \(|\mathbf{b}| = n-2\), one can prove the assertion in a similar way to assertion \eqref{item_ins_gpv_DII}.
  \end{enumerate}
\end{proof}

\begin{prop}[Removing \(1\)]\label{prop_rm_1_DII}
  Let \(l \in \{ 0,\dots,n \}\) and \(\mathbf{a} \in \mathsf{Col}^{D_n}_{l}\).
  If \(1 \in \mathbf{a}\), then the following hold\textup{:}
  \begin{enumerate}
    \item\label{item_rm_1_DII} \(\mathbf{a}_{> 1} \in \mathsf{Col}^{D\mathrm{II}_n}_{l-1, \overline{\mathsf{e}}(\mathbf{a})}\),
    \item\label{item_inv_rm_1_DII} \((\mathbf{a}_{> 1})*1 = \mathbf{a}\).
  \end{enumerate}
\end{prop}
\begin{proof}
  The assertions are obvious.
\end{proof}

\begin{prop}[Removing \(\overline{1}\)]\label{prop_rm_1b_DII}
  Let \(\mathbf{a} \in \mathsf{Col}^{D_n}_{n}\).
  If \(\overline{1} \in \mathbf{a}\), then the following hold\textup{:}
  \begin{enumerate}
    \item\label{item_rm_1b_DII} \(\mathbf{a}_{> 1} \in \mathsf{Col}^{D\mathrm{II}_n}_{n-1, \mathsf{e}(\mathbf{a})}\),
    \item\label{item_inv_rm_1b_DII} \(\overline{1} * (\mathbf{a}_{> 1}) = \mathbf{a}\).
  \end{enumerate}
\end{prop}
\begin{proof}
  The assertions are obvious.
\end{proof}

\begin{prop}[Removing the least weakly deep number]\label{prop_rm_lwd_DII}
  Let \(l \in \{ 0,\dots,n \}\) and \(\mathbf{a} \in \mathsf{Col}^{D_n}_{l} \setminus \mathsf{Col}^{D\mathrm{II}_n}_{l}\).
  Assume that \(1 \in \mathsf{V}(\mathbf{a})\), and set \(b := \mathsf{lwd}(\mathbf{a})\), \(\mathbf{b} := \mathbf{a} \Rightarrow b\)\textup{;} see Lemma \textup{\ref{lem_exist_wd_D}}.
  Then, the following hold\textup{:}
  \begin{enumerate}
    \item\label{item_rm_lwd_DII} \(\overline{1}*\mathbf{b} \in \mathsf{Col}^{D\mathrm{II}_n}_{l-1}\),
    \item\label{item_gpv_rm_lwd_DII} \(b = \mathsf{gpv}(\overline{1}*\mathbf{b})\),
    \item\label{item_inv_rm_lwd_DII} \((\overline{1}*\mathbf{b})_{> 1} \Leftarrow b = \mathbf{a}\) if \(b < n\),
    \item\label{item_inv_max_len_rm_lwd_DII} \((\overline{1}*\mathbf{b})_{> 1} \Leftarrow_{\mathsf{e}(\mathbf{a})} n = \mathbf{a}\) if \(b = n\).
  \end{enumerate}
\end{prop}
\begin{proof}
  \hfill
  \begin{enumerate}
    \item Clearly, we have the following:
    \begin{itemize}
      \item \(|\overline{1}*\mathbf{b}| = |\mathbf{a}|-1 = l-1 < n\),
      \item \(1 \notin \overline{1}*\mathbf{b}\),
      \item \(\mathsf{O}(\mathbf{b}) = \mathsf{O}(\mathbf{a}) \setminus \{ b \}\).
    \end{itemize}
    For each \(b' \in \mathsf{O}(b)\), we have
    \[
      |((\overline{1}*\mathbf{b})_{> 1})_{\leq b'}| = |\mathbf{b}_{\leq b'}| =
      \begin{cases}
        |\mathbf{a}_{\leq b'}| & \text{ if } b' < b, \\
        |\mathbf{a}_{\leq b'}|-2 & \text{ if } b' > b.
      \end{cases}
    \]
    We have \(|\mathbf{a}_{\leq b'}| \leq b'\) for all \(b'\).
    Also, since \(b = \mathsf{lwd}(\mathbf{a})\), it holds that
    \[
      |\mathbf{a}_{\leq b'}| < b' \quad \text{ if } b' < b.
    \]
    Therefore, we obtain
    \[
      |((\overline{1}*\mathbf{b})_{> 1})_{\leq b'}| < b'.
    \]
    Hence, we complete the proof.
    \item It is clear that \(b \in \mathsf{PV}(\overline{1}*\mathbf{b})\).
    In order to prove the assertion, assume contrary: there exists \(a \in \mathsf{PV}(\overline{1}*\mathbf{b})\) such that \(a > b\).
    Then, we have
    \[
      |(\mathbf{b}_{> b})_{< a}| = |\mathbf{b}_{< a}| - |\mathbf{b}_{< b}| = (a-2)-(b-2) = a-b > \#\{ b+1,\dots,a-1 \}.
    \]
    This implies that
    \[
      \mathsf{O}((\mathbf{b}_{> b})_{< a}) \neq \emptyset.
    \]
    Set
    \[
      b' := \max \mathsf{O}((\mathbf{b}_{> b})_{< a}).
    \]
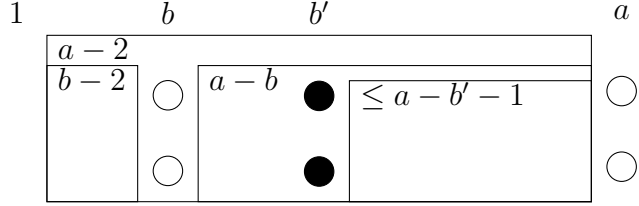
    \begin{center}
      \begin{tikzpicture}[xscale=2]
        \begin{scope}
          \node (1) at (0,0) {\(1\)};
          \node (a) at (1,0) {\(b\)};
          \node (b) at (2,0) {\(b'\)};
          \node (ad) at (4,0) {\(a\)};

          \node[draw, circle] [below=0.6cm of a] {};
          \node[draw, circle] [below=1.6cm of a] {};
          \node[draw, circle, fill] [below=0.6cm of b] {};
          \node[draw, circle, fill] [below=1.6cm of b] {};
          \node[draw, circle] [below=0.6cm of ad] {};
          \node[draw, circle] [below=1.6cm of ad] {};

          \draw (0.2,-0.3) rectangle (4-0.2,-2.5);
          \node[anchor=west] at (0.2,-0.5) {\(a-2\)};
          \draw (0.2,-0.7) rectangle (1-0.2,-2.5);
          \node[anchor=west] at (0.2,-0.9) {\(b-2\)};
          \draw (1+0.2,-0.7) rectangle (4-0.2,-2.5);
          \node[anchor=west] at (1+0.2,-0.9) {\(a-b\)};
          \draw (2+0.2,-0.9) rectangle (4-0.2,-2.5);
          \node[anchor=west] at (2+0.2,-1.1) {\(\leq a-b'-1\)};
        \end{scope}
      \end{tikzpicture}
    \end{center}
    By the maximality of \(b'\), we have
    \[
      |(\mathbf{b}_{> b'})_{< a}| \leq a-b'-1.
    \]
    Hence,
    \[
      |\mathbf{b}_{\leq b'}| = |\mathbf{b}_{< a}| - |(\mathbf{b}_{> b'})_{< a}| \geq (a-2)-(a-b'-1) = b'-1.
    \]
    Therefore,
    \[
      |\mathbf{a}_{\leq b'}| = |\mathbf{b}_{\leq b'}|+2 \geq b'+1.
    \]
    This contradicts that \(\mathbf{a} \in \mathsf{Col}^{D_n}_{l}\).
    Hence, the assertion follows.
  \end{enumerate}
  The rest of the assertions are obvious.
\end{proof}

\begin{defi}[Connecting maps]\label{def_conn_map_DII}
  \hfill
  \begin{enumerate}
    \item For each \(l \in \{ 0,\dots,n-2 \}\), the \emph{connecting map \(\delta_l\)} is the map
    \[
      \delta_l \colon \mathsf{Col}^{D\mathrm{II}_n}_{l} \to \mathsf{Col}^{D_n}_{l+1} \setminus \mathsf{Col}^{D\mathrm{II}_n}_{l+1}
    \]
    defined by
    \[
      \delta_l(\mathbf{a}) :=
      \begin{cases}
        \mathbf{a}*1 & \text{ if } \overline{1} \notin \mathbf{a}, \\
        \mathbf{a}_{> 1} \Leftarrow \mathsf{gpv}(\mathbf{a}) & \text{ if } \overline{1} \in \mathbf{a}.
      \end{cases}
    \]
    \item The \emph{connecting map \(\delta_{n-1,\pm}\)} is the map
    \[
      \delta_{n-1,\pm} \colon \mathsf{Col}^{D\mathrm{II}_n}_{n-1} \to \mathsf{Col}^{D_n}_{n,\pm}
    \]
    defined by
    \[
      \delta_{n-1,\pm}(\mathbf{a}) := \begin{cases}
        \mathbf{a}*1 & \text{ if } \overline{1} \notin \mathbf{a} \text{ and } e(\mathbf{a}) = \mp, \\
        \overline{1}*\mathbf{a} & \text{ if } \overline{1} \notin \mathbf{a} \text{ and } e(\mathbf{a}) = \pm, \\
        \mathbf{a}_{> 1} \Leftarrow \mathsf{gpv}(\mathbf{a}) & \text{ if } \overline{1} \in \mathbf{a},\ \mathsf{gpv}(\mathbf{a}) < n, \text{ and } e(\mathbf{a}) = \mp, \\
        \mathbf{a}_{> 1} \Leftarrow_\pm n & \text{ if } \overline{1} \in \mathbf{a},\ \mathsf{gpv}(\mathbf{a}) < n, \text{ and } e(\mathbf{a}) = \pm, \\
        \mathbf{a}_{> 1} \Leftarrow_\pm n & \text{ if } \overline{1} \in \mathbf{a} \text{ and } \mathsf{gpv}(\mathbf{a}) = n.
      \end{cases}
    \]
  \end{enumerate}
\end{defi}

\begin{rem}[Well-definedness of connecting maps]
  The connecting maps \(\delta_l\) for \(l \in \{ 0,\dots,n-2 \}\) are well-defined by Propositions \ref{prop_ins_1_DII} \eqref{item_DII_ins_1_DII} and \ref{prop_ins_pv_DII} \eqref{item_ins_gpv_DII}.
  And the connecting map \(\delta_{n,\pm}\) is well-defined by Propositions \ref{prop_ins_1_DII} \eqref{item_DII_ins_1_DII}, \ref{prop_ins_1b_DII} \eqref{item_DII_ins_1b_DII}, and \ref{prop_ins_pv_DII} \eqref{item_ins_gpv_DII}--\eqref{item_ins_gpv_max_len_DII}.
\end{rem}

\begin{thm}[Bijectivity of connecting maps]\label{thm_bij_conn_map_DII}
  The connecting maps are bijective.
\end{thm}
\begin{proof}
  Let us construct the inverses of connecting maps.
  For each \(l \in \{ 0,\dots,n-2 \}\), define the map
  \[
    \delta_l' \colon \mathsf{Col}^{D_n}_{l+1} \setminus \mathsf{Col}^{D\mathrm{II}_n}_{l+1} \to \mathsf{Col}^{D\mathrm{II}_n}_{l}
  \]
  by
  \[
    \delta_l'(\mathbf{a}) :=
    \begin{cases}
      \mathbf{a}_{> 1} & \text{ if } 1 \in \mathbf{a}, \\
      \overline{1}*(\mathbf{a} \Rightarrow \mathsf{lwd}(\mathbf{a})) & \text{ if } 1 \notin \mathbf{a}.
    \end{cases}
  \]
  This is well-defined and the inverse of \(\delta_l\) by Propositions \ref{prop_rm_1_DII}, \ref{prop_rm_lwd_DII}, \ref{prop_ins_1_DII}, and \ref{prop_ins_pv_DII}.

  Also, define the map
  \[
    \delta_{n-1,\pm}' \colon \mathsf{Col}^{D_n}_{n,\pm} \to \mathsf{Col}^{D\mathrm{II}_n}_{n-1}
  \]
  by
  \[
    \delta_{n-1,\pm}'(\mathbf{a}) :=
    \begin{cases}
      \mathbf{a}_{> 1} & \text{ if } 1 \notin \mathsf{V}(\mathbf{a}), \\
      \overline{1}*(\mathbf{a} \Rightarrow \mathsf{lwd}(\mathbf{a})) & \text{ if } 1 \in \mathsf{V}(\mathbf{a}).
    \end{cases}
  \]
  This is well-defined and the inverse of \(\delta_{n,\pm}\) by Propositions \ref{prop_rm_1_DII}, \ref{prop_rm_lwd_DII}, \ref{prop_ins_1_DII}, \ref{prop_ins_1b_DII}, and \ref{prop_ins_pv_DII}.
\end{proof}

\subsection{Reductions}
In this section, we transform columns of type \(D\mathrm{II}_n\) into \(D_{n-1}\), and type \(D_n\) into \(D\mathrm{II}_n\).
Given a word \(\mathbf{a} \in \mathcal{W}^{D_n}\), we use the following notations:
\begin{itemize}
  \item \(\mathbf{a}[+1]\): the word obtained by replacing the letters \(1,\dots,n-1,\overline{n-1},\dots,\overline{1}\) with \(2,\dots,n,\overline{n},\dots,\overline{2}\), respectively (if \(n,\overline{n} \notin \mathbf{a}\)).
  \item \(\mathbf{a}[-1]\): the word obtained by replacing the letters \(2,\dots,n,\overline{n},\dots,\overline{2}\) with \(1,\dots,n-1,\overline{n-1},\dots,\overline{1}\), respectively (if \(1,\overline{1} \notin \mathbf{a}\)).
\end{itemize}

\begin{defi}[Reduction of columns from \(D\mathrm{II}_n\) to \(D_{n-1}\)]
  Let \(\mathbf{a} \in \mathsf{Col}^{D\mathrm{II}_n}\).
  The \emph{reduction of \(\mathbf{a}\) from \(D\mathrm{II}_n\) to \(D_{n-1}\)} is the word \(\mathbf{a} \downarrow^{D\mathrm{II}_n}_{D_{n-1}} \in \mathcal{W}^{D_{n-1}}\) defined by
  \[
    \mathbf{a} \downarrow^{D\mathrm{II}_n}_{D_{n-1}} := \mathbf{a}_{> 1}[-1].
  \]
\end{defi}

\begin{thm}[Reduction of columns from \(D\mathrm{II}_n\) to \(D_{n-1}\)]\label{thm_red_DII_D_col}
  Let \(l \in \{ 0,\dots,n-1 \}\).
  The assignment \(\mathbf{a} \mapsto \mathbf{a} \downarrow^{D\mathrm{II}_n}_{D_{n-1}}\) gives rise to a bijection
  \[
    \cdot \downarrow^{D\mathrm{II}_n}_{D_{n-1}} \colon \mathsf{Col}^{D\mathrm{II}_n}_{l} \to \mathsf{Col}^{D_{n-1}}_{l} \sqcup \mathsf{Col}^{D_{n-1}}_{l-1}.
  \]
\end{thm}
\begin{proof}
  Let \(\mathbf{a} \in \mathsf{Col}^{D\mathrm{II}_n}_{l}\).
  By the definitions of columns, we see that
  \[
    \mathbf{a} \downarrow^{D\mathrm{II}_n}_{D_{n-1}} \in
    \begin{cases}
      \mathsf{Col}^{D_{n-1}}_{l} & \text{ if } \overline{1} \notin \mathbf{a}, \\
      \mathsf{Col}^{D_{n-1}}_{l-1} & \text{ if } \overline{1} \in \mathbf{a}.
    \end{cases}
  \]
  Consider the map
  \begin{align}\label{eq_inv_red_DII_D}
    \mathsf{Col}^{D_{n-1}}_{l} \sqcup \mathsf{Col}^{D_{n-1}}_{l-1} \to \mathcal{W}^{D_n};\ \mathbf{b} \mapsto
    \begin{cases}
      \mathbf{b}[+1] & \text{ if } |\mathbf{b}| = l, \\
      \overline{1}*(\mathbf{b}[+1]) & \text{ if } |\mathbf{b}| = l-1.
    \end{cases}
  \end{align}
  Then, its image lies in \(\mathsf{Col}^{D\mathrm{II}_n}_{l}\).
  Also, this map is the inverse of \(\cdot \downarrow^{D\mathrm{II}_n}_{D_{n-1}}\).
  Hence, the assertion follows.
\end{proof}

\begin{cor}[\(D_{n-1}\)-crytsal structure of \(\mathsf{Col}^{D\mathrm{II}_n}\)]\label{cor_crystal_col_DII}
  Let \(l \in \{ 0,\dots,n-1 \}\).
  The subset \(\mathsf{Col}^{D\mathrm{II}_n}_{l} \subseteq \mathsf{Col}^{D_n}_{l}\) forms a subcrystal of type \(D_{n-1}\)\textup{;} see \textup{\S\ref{ssect_crystal}} for the \(D_{n-1}\)-crystal structure of \(\mathsf{Col}^{D_n}_{l}\).
  Moreover, the reduction map \(\cdot \downarrow^{D\mathrm{II}_n}_{D_{n-1}}\) is an isomorphism of crystals.
\end{cor}
\begin{proof}
  Clearly, the map \eqref{eq_inv_red_DII_D} is a morphism of crystals of type \(D_{n-1}\).
  Hence, the assertion follows from Theorem \ref{thm_red_DII_D_col}.
\end{proof}

\begin{defi}[Reduction of columns from \(D_n\) to \(D\mathrm{II}_n\)]\label{def_red_col_D_DII}
  Let \(\mathbf{a} \in \mathsf{Col}^{D_n}_{l}\).
  The \emph{reduction of \(\mathbf{a}\) from \(D_n\) to \(D\mathrm{II}_n\)} is the word \(\mathbf{a} \downarrow^{D_n}_{D\mathrm{II}_n} \in \mathcal{W}^{D_n}\) defined by
  \[
    \mathbf{a} \downarrow^{D_n}_{D\mathrm{II}_n} :=
    \begin{cases}
      \mathbf{a} & \text{ if } \mathbf{a} \in \mathsf{Col}^{D\mathrm{II}_n}_{l}, \\
      \mathbf{a}_{> 1} & \text{ if } \mathbf{a} \notin \mathsf{Col}^{D\mathrm{II}_n}_{l} \text{ and } 1 \notin \mathsf{V}(\mathbf{a}), \\
      \overline{1}*(\mathbf{a} \Rightarrow \mathsf{lwd}(\mathbf{a})) & \text{ if } \mathbf{a} \notin \mathsf{Col}^{D\mathrm{II}_n}_{l} \text{ and } 1 \in \mathsf{V}(\mathbf{a}).
    \end{cases}
  \]
\end{defi}

\begin{thm}[Reduction of columns from \(D_n\) to \(D\mathrm{II}_n\)]\label{thm_red_D_DII_col}
  \hfill
  \begin{enumerate}
    \item Let \(l \in \{ 0,\dots,n-1 \}\).
    Then, the assignment \(\mathbf{a} \mapsto \mathbf{a} \downarrow^{D_n}_{D\mathrm{II}_n}\) gives rise to a bijection
    \[
      \mathsf{Col}^{D_n}_{l} \to \mathsf{Col}^{D\mathrm{II}_n}_{l} \sqcup \mathsf{Col}^{D\mathrm{II}_n}_{l-1}.
    \]
    \item The assignment \(\mathbf{a} \mapsto \mathbf{a} \downarrow^{D_n}_{D\mathrm{II}_n}\) gives rise to a bijection
    \[
      \mathsf{Col}^{D_n}_{n,\pm} \to \mathsf{Col}^{D\mathrm{II}_n}_{n-1}.
    \]
  \end{enumerate}
\end{thm}
\begin{proof}
  For each \(l \in \{ 0,\dots,n \}\) and \(\mathbf{a} \in \mathsf{Col}^{D_n}_{l}\), we have
  \[
    \mathbf{a} \downarrow^{D_n}_{D\mathrm{II}_n} =
    \begin{cases}
      \mathbf{a} & \text{ if } \mathbf{a} \in \mathsf{Col}^{D\mathrm{II}_n}_{l} \\
      \delta_{l-1}^{-1}(\mathbf{a}) & \text{ if } \mathbf{a} \notin \mathsf{Col}^{D\mathrm{II}_n}_{l} \text{ and } l < n, \\
      \delta_{n-1,\mathbf{e}(\mathbf{a})}^{-1}(\mathbf{a}) & \text{ if } l = n.
    \end{cases}
  \]
  Hence, the assertion follows from Theorem \ref{thm_bij_conn_map_DII}.
\end{proof}

\subsection{Tableaux}
\begin{defi}[Kashiwara--Nakashima tableaux of type \(D\mathrm{II}\)]\label{def_knt_DII}
  A Kashiwara--Nakashima tableau \(T = (\mathbf{a}^1,\dots,\mathbf{a}^m)\) of type \(D_n\) and shape \(\lambda \in X^+_{2n,0}\) is said to be of \emph{type \(D\mathrm{II}_n\)} if its first column \(\mathbf{a}^1\) is of type \(D\mathrm{II}_n\).
  The set of Kashiwara--Nakashima tableaux of type \(D\mathrm{II}_n\) and shape \(\lambda\) is denoted by
  \[
    \mathsf{KNT}^{D\mathrm{II}_n}(\lambda).
  \]
  Also, set
  \[
    \mathsf{KNT}^{D\mathrm{II}_n}_{0} := \bigsqcup_{\lambda \in X^+_{2n,0}} \mathsf{KNT}^{D_n}(\lambda).
  \]
\end{defi}

\begin{rem}[Shape of tableaux of type \(D\mathrm{II}\)]
  Let \(\lambda \in X^+_{2n,0}\) be such that \(\mathsf{KNT}^{D\mathrm{II}_n}(\lambda) \neq \emptyset\).
  Since every Kashiwara--Nakashima columns of type \(D\mathrm{II}_n\) has length less than \(n\), so does \(\mathsf{par}(\lambda)\).
  Hence, we can regard \(\lambda\) as an element of \(X^+_{2n-1,0}\).
\end{rem}

Let us first transform a tableau of type \(D\mathrm{II}_n\) into type \(D_{n-1}\).

\begin{prop}[Adjacent columns]\label{prop_adj_col_DII}
  Let \(\mathbf{a} = (a_k,\dots,a_1), \mathbf{b} = (b_l,\dots,b_1) \in \mathsf{Col}^{D_n}\).
  \begin{enumerate}
    \item\label{item_propagate_adj_col_DII} If \(\mathbf{a} \triangleleft \mathbf{b}\) and \(\mathbf{a} \in \mathsf{Col}^{D\mathrm{II}_n}\), then \(\mathbf{b} \in \mathsf{Col}^{D\mathrm{II}_n}\).
    \item\label{item_equivalence_adj_col_DII} If \(\mathbf{a}, \mathbf{b} \in \mathsf{Col}^{D\mathrm{II}_n}\), then we have \(\mathbf{a} \triangleleft \mathbf{b}\) if and only if \((\mathbf{a} \downarrow^{D\mathrm{II}_n}_{D_{n-1}}) \triangleleft (\mathbf{b} \downarrow^{D\mathrm{II}_n}_{D_{n-1}})\).
  \end{enumerate}
\end{prop}
\begin{proof}
  \hfill
  \begin{enumerate}
    \item The following can be immediately deduced:
    \begin{itemize}
      \item \(l \leq k < n\),
      \item \(1 \notin \mathbf{b}\) since \(1 \prec a_1 \preceq b_1\).
    \end{itemize}
    Hence, we only need to show that
    \[
      |(\mathbf{b}_{> 1})_{\leq b}| < b \quad \text{ for all } b \in \mathsf{O}(\mathbf{b}).
    \]
    This claim is clear when \(\overline{1} \in \mathbf{b}\); see Proposition \ref{prop_entry_1b_DII}.
    In the rest of the proof, we assume that \(\overline{1} \notin \mathbf{b}\).
    For the proof of our claim, assume contrary.
    Then, there exists \(b \in \mathsf{O}(\mathbf{b})\) such that
    \[
      |\mathbf{b}_{\leq b}| = b;
    \]
    note that we have \(\mathbf{b}_{> 1} = \mathbf{b}\) and \(|\mathbf{b}_{\leq b'}| \leq b'\) for all \(b' \in \mathsf{O}(\mathbf{b})\).
    Since \(|\mathbf{b}| = l < n\), it follows that
    \[
      b = |\mathbf{b}_{\leq b}| \leq |\mathbf{b}| < n.
    \]
    Hence, there exist \(1 \leq s < t \leq l\) such that
    \[
      b_s = b, \quad b_t = \overline{b}.
    \]
    Since we have
    \begin{itemize}
      \item \(1 \prec a_1 \prec \cdots \prec a_s \preceq b_s = b\),
      \item \(\overline{b} = b_t \prec \cdots \prec b_l \prec \overline{1}\),
      \item \(s + (l-t+1) = |\mathbf{b}_{\leq b}| = b\),
    \end{itemize}
    there exist \(1 \leq r \leq s\), \(t \leq u \leq l\), and \(a \in \{ 2,\dots,b \}\) such that
    \[
      a_r = a, \quad b_u = \overline{a}.
    \]
    Let \(a\) be the minimum among such numbers.
    By the minimality of \(a\), we must have
    \[
      (r-1) + (l-u) = |\mathbf{a}^{\prec a}| + |\mathbf{b}^{\succ \overline{a}}| \leq a-2.
    \]
    On the other hand, by the adjacent condition, we also have
    \begin{align*}
      (r-1) + (l-u)
      &= (s+(l-t+1))-((s-r)+(u-t))-2 \\
      &> b-(b-a)-2 \\
      &= a-2.
    \end{align*}
    Thus, we obtain a contradiction, and complete the proof.
    \item Set \(\mathbf{a}' := \mathbf{a} \downarrow^{D\mathrm{II}_n}_{D_{n-1}}\) and \(\mathbf{b}' := \mathbf{b} \downarrow^{D\mathrm{II}_n}_{D_{n-1}}\), and write
    \[
      \mathbf{a}' = (a'_{k'},\dots,a'_1), \quad \mathbf{b}' = (b'_{l'},\dots,b'_1).
    \]
    First, let us show that if \(\mathbf{a} \triangleleft \mathbf{b}\) then \(k' \geq l'\).
    Assume contrary.
    Then, we must have
    \[
      k = l,\ \overline{1} \in \mathbf{a}, \text{ and } \overline{1} \notin \mathbf{b}.
    \]
    In this case, it holds that
    \[
      \overline{1} = a_k \preceq b_k = b_l \prec \overline{1}.
    \]
    This is a contradiction.
    Hence, we have proved that \(k' \geq l'\).

    The rest of the assertion can be verified straightforwardly.
    Thus, we complete the proof.
  \end{enumerate}
\end{proof}

\begin{cor}[Columns of tableaux of type \(D\mathrm{II}\)]\label{cor_col_tab_DII}
  Let \(T = (\mathbf{a}^1,\dots,\mathbf{a}^m) \in \mathsf{KNT}^{D\mathrm{II}_n}_0\).
  Then, the columns \(\mathbf{a}^1,\dots,\mathbf{a}^m\) are of type \(D\mathrm{II}_n\).
\end{cor}
\begin{proof}
  The assertion is immediate from Proposition \ref{prop_adj_col_DII} \eqref{item_propagate_adj_col_DII}.
\end{proof}

\begin{defi}[Reduction from \(D\mathrm{II}_n\) to \(D_{n-1}\)]
  Let \(T = (\mathbf{a}^1,\dots,\mathbf{a}^m) \in \mathsf{KNT}^{D\mathrm{II}_n}_0\).
  The \emph{reduction of \(T\) from \(D\mathrm{II}_n\) to \(D_{n-1}\)} is the sequence \(T \downarrow^{D\mathrm{II}_n}_{D_{n-1}}\) of columns of type \(D_{n-1}\) defined by
  \[
    T \downarrow^{D\mathrm{II}_n}_{D_{n-1}} := (\mathbf{a}^1 \downarrow^{D\mathrm{II}_n}_{D_{n-1}}, \dots, \mathbf{a}^m \downarrow^{D\mathrm{II}_n}_{D_{n-1}}).
  \]
  This is well-defined by Corollary \ref{cor_col_tab_DII}.
\end{defi}

\begin{thm}[Reduction from \(D\mathrm{II}_n\) to \(D_{n-1}\)]\label{thm_red_map_DII_D}
  Let \(\nu \in X^+_{2n-1,0}\).
  The assignment \(T \mapsto T \downarrow^{D\mathrm{II}_n}_{D_{n-1}}\) gives rise to a bijection
  \[
    \cdot \downarrow^{D\mathrm{II}_n}_{D_{n-1}} \colon \mathsf{KNT}^{D\mathrm{II}_n}(\nu) \to \bigsqcup_{\substack{\mu \in X^+_{2n-2,0} \\ \mathsf{par}(\mu) \overset{\textsf{hor}}{\subseteq} \mathsf{par}(\nu)}} \mathsf{KNT}^{D_{n-1}}(\mu).
  \]
\end{thm}
\begin{proof}
  Let \(T = (\mathbf{a}^1,\dots,\mathbf{a}^m) \in \mathsf{KNT}^{D\mathrm{II}_n}(\nu)\).
  By Proposition \ref{prop_adj_col_DII} \eqref{item_equivalence_adj_col_DII}, we see that \(T \downarrow^{D\mathrm{II}_n}_{D_{n-1}} \in \mathsf{KNT}^{D_{n-1}}_{0}\).
  Since \(|\mathbf{a}^r|-|\mathbf{a}^r \downarrow^{D\mathrm{II}_n}_{D_{n-1}}| \in \{ 0,1 \}\) for all \(r \in \{ 1,\dots,m \}\), we have
  \[
    \mathsf{par}(\mathsf{sh}(T \downarrow^{D\mathrm{II}_n}_{D_{n-1}})) \overset{\textsf{hor}}{\subseteq} \mathsf{par}(\nu).
  \]

  Conversely, let \(\mu \in X^+_{2n-2,0}\) such that \(\mathsf{par}(\mu) \overset{\textsf{hor}}{\subseteq} \mathsf{par}(\nu)\), and \(S = (\mathbf{b}^1,\dots,\mathbf{b}^m) \in \mathsf{KNT}^{D_{n-1}}(\mu)\).
  By Theorem \ref{thm_red_DII_D_col}, there exist unique columns \(\mathbf{a}^1,\dots,\mathbf{a}^m\) of type \(D\mathrm{II}_n\) such that
  \[
    \mathbf{a}^r \downarrow^{D\mathrm{II}_n}_{D_{n-1}} = \mathbf{b}^r \quad \text{ for all } r \in \{ 1,\dots,m \}
  \]
  and
  \[
    (|\mathbf{a}^1|,\dots,|\mathbf{a}^m|) = \mathsf{par}(\mu).
  \]
  Again, Proposition \ref{prop_adj_col_DII} \eqref{item_equivalence_adj_col_DII} implies that \((\mathbf{a}^1,\dots,\mathbf{a}^m) \in \mathsf{KNT}^{D_n}(\mu)\).
  Thus, we complete the proof.
\end{proof}

\begin{cor}[\(D_{n-1}\)-crystal structure of \(\mathsf{KNT}^{D\mathrm{II}_n}(\nu)\)]
  The subset \(\mathsf{KNT}^{D\mathrm{II}_n}(\nu) \subseteq \mathsf{KNT}^{D_n}(\nu)\) is a subcrystal of type \(D_{n-1}\).
  Moreover, the reduction map \(\cdot \downarrow^{D\mathrm{II}_n}_{D_{n-1}}\) is an isomorphism of crystals.
\end{cor}
\begin{proof}
  The assertions are immediate from Theorem \ref{thm_red_map_DII_D} and Corollary \ref{cor_crystal_col_DII}.
\end{proof}

Next, we will transform a tableau of type \(D_n\) into type \(D\mathrm{II}_n\).

\begin{defi}[Successor]\label{def_suc_map_DII}
  Let \(T = (\mathbf{a}^1,\dots,\mathbf{a}^m) \in \mathsf{KNT}^{D_n}_{0}\).
  The \emph{successor} of \(T\) is the tableau \(\mathsf{suc}(T)\) defined by
  \[
    \mathsf{suc}(T) := \mathsf{P}(\mathbf{a}^1 \downarrow^{D_n}_{D\mathrm{II}_n} * \mathbf{a}^2 * \cdots * \mathbf{a}^m).
  \]
  see \S\ref{ssect_col_ins} for notations.
\end{defi}

\begin{prop}[Successor]\label{prop_suc_map_DII}
  Let \(T \in \mathsf{KNT}^{D_n}_{0}\).
  \begin{enumerate}
    \item \(\mathsf{suc}(T) = T\) if and only if \(T \in \mathsf{KNT}^{D\mathrm{II}_n}_{0}\).
    \item There exists \(k \in \mathbb{Z}_{\geq 0}\) such that \(\mathsf{suc}^{k+1}(T) = \mathsf{suc}^k(T)\).
  \end{enumerate}
\end{prop}
\begin{proof}
  We have \(\mathbf{a}^1 \downarrow^{D_n}_{D\mathrm{II}_n} = \mathbf{a}^1\) if and only if \(\mathbf{a}^1 \in \mathsf{Col}^{D\mathrm{II}_n}\).
  Hence, the first assertion is clear.

  Also, we have
  \[
    |\mathbf{a}^1 \downarrow^{D_n}_{D\mathrm{II}_n}| =
    \begin{cases}
      |\mathbf{a}^1| & \text{ if } \mathbf{a}_1 \in \mathsf{Col}^{D\mathrm{II}_n}, \\
      |\mathbf{a}^1|-1 & \text{ if } \mathbf{a}_1 \notin \mathsf{Col}^{D\mathrm{II}_n}.
    \end{cases}
  \]
  Therefore, the second assertion can be proved by induction on \(|\mathsf{par}(\mathsf{sh}(T))|\).
\end{proof}

\begin{defi}[Reduction from \(D_n\) to \(D\mathrm{II}_n\)]\label{def_red_tab_D_DII}
  Let \(T \in \mathsf{KNT}^{D_n}_{0}\).
  The \emph{reduction of \(T\) from \(D_n\) to \(D\mathrm{II}_n\)} is the Kashiwara--Nakashima tableau \(T \downarrow^{D_n}_{D\mathrm{II}_n}\) of type \(D\mathrm{II}_n\) obtained from \(T\) by applying the successor map iteratively until the result becomes type \(D\mathrm{II}_n\); see Proposition \ref{prop_suc_map_DII}.
\end{defi}

\begin{thm}[Reduction of tableaux from \(D_n\) to \(D\mathrm{II}_n\)]\label{thm_red_KNT_D_DII}
  Let \(\lambda \in X^+_{2n,0}\).
  The assignment \(T \mapsto T \downarrow^{D_n}_{D\mathrm{II}_n}\) gives rise to a bijection
  \[
    \mathsf{KNT}^{D_n}(\lambda) \to \bigsqcup_{\substack{\nu \in X^+_{2n-1,0} \\ \mathsf{par}(\nu) \overset{\textsf{hor}}{\subseteq} \mathsf{par}(\lambda)}} \mathsf{KNT}^{D\mathrm{II}_n}(\nu).
  \]
\end{thm}

The proof of Theorem \ref{thm_red_KNT_D_DII} will be given in \S\ref{ssect_red_D_DII} with the aid of representation theory of quantum symmetric pairs of type \(D\mathrm{II}\).

In the remainder of this section, we will assume that Theorem \ref{thm_red_KNT_D_DII} holds, and prove the last main result of this paper.
Let \(T \in \mathsf{KNT}^{D_n}_0\).
Define tableaux \(T \downarrow^{D_n}_{D_{k-1}}\) for \(k \in \{ 2,\dots,n \}\), and  \(T \downarrow^{D_n}_{D\mathrm{II}_k}\) for \(k \in \{ 2,\dots,n-1 \}\) inductively by
\[
  T \downarrow^{D_n}_{D_{k-1}} := (T \downarrow^{D_n}_{D\mathrm{II}_{k}}) \downarrow^{D\mathrm{II}_{k}}_{D_{k-1}}, \quad T \downarrow^{D_n}_{D\mathrm{II}_k} := (T \downarrow^{D_n}_{D_{k}}) \downarrow^{D_{k}}_{D\mathrm{II}_{k}}.
\]
Similarly, for each \(T \in \mathsf{KNT}^{D\mathrm{II}_n}_0\) define
\[
  T \downarrow^{D\mathrm{II}_n}_{D\mathrm{II}_{k}} := (T \downarrow^{D\mathrm{II}_n}_{D_{k}}) \downarrow^{D_{k}}_{D\mathrm{II}_{k}}, \quad T \downarrow^{D\mathrm{II}_n}_{D_{k-1}} := (T \downarrow^{D\mathrm{II}_n}_{D\mathrm{II}_{k}}) \downarrow^{D\mathrm{II}_{k}}_{D_{k-1}}.
\]
See \S\ref{ssect_results} for an example.

\begin{cor}[Bijections between tableaux and patterns]\label{cor_bij_D}
  \hfill
  \begin{enumerate}
    \item Let \(\lambda \in X^+_{2n,0}\).
    The assignment
    \[
      T \mapsto (\mathsf{sh}(T), \mathsf{sh}(T \downarrow^{D_n}_{D\mathrm{II}_n}), \mathsf{sh}(T \downarrow^{D_n}_{D_{n-1}}),\dots,\mathsf{sh}(T \downarrow^{D_n}_{D_1}))
    \]
    gives rise to a bijection
    \[
      \mathsf{KNT}^{D_n}(\lambda) \to \mathsf{GTP}_{2n}(\lambda).
    \]
    \item Let \(\nu \in X^+_{2n-1,0}\).
    The assignment
    \[
      T \mapsto (\mathsf{sh}(T), \mathsf{sh}(T \downarrow^{D\mathrm{II}_n}_{D_{n-1}}), \mathsf{sh}(T \downarrow^{D\mathrm{II}_n}_{D\mathrm{II}_{n-1}}),\dots,\mathsf{sh}(T \downarrow^{D\mathrm{II}_n}_{D_1}))
    \]
    gives rise to a bijection
    \[
      \mathsf{KNT}^{D\mathrm{II}_n}(\nu) \to \mathsf{GTP}_{2n-1}(\nu).
    \]
  \end{enumerate}
\end{cor}
\begin{proof}
  The assertions are immediate from Definition \ref{def_GTP} and Theorems \ref{thm_red_map_DII_D} and \ref{thm_red_KNT_D_DII}.
\end{proof}

\section{Quantum groups of type \(D\)}\label{sect_qg_D}
In this section, we briefly recall basics of quantum groups of type \(D\).
Our main reference is \cite{Lus93}.

\subsection{Quantum groups}
Let \(n \in \mathbb{Z}_{\geq 1}\), and consider the special orthogonal Lie algebra \(\mathfrak{so}_{2n}\).
Recall the notations \(d_i\), \(\epsilon_i\), \(I\), \(h_i\), and \(\alpha_i\) from \S\S \ref{ssect_so} and \ref{ssect_crystal}.
Set
\begin{itemize}
  \item \(Y := 
  \begin{cases}
    \mathsf{Span}_{\mathbb{Z}} \{ h_i \mid i \in I \} & \text{ if } n \geq 2, \\
    2\mathbb{Z} d_1 & \text{ if } n = 1,
  \end{cases}
  \)
  \item \(a_{i,j} := \langle h_i, \alpha_j \rangle\) for each \(i,j \in I\).
\end{itemize}
The \emph{quantum group} \(\mathbf{U} = \mathbf{U}(\mathfrak{so}_{2n})\) of \(\mathfrak{so}_{2n}\) is the unital associative \(\mathbb{Q}(q)\)-algebra with generators
\[
  \{ E_i,F_i,K_h \mid i \in I,\ h \in Y \}
\]
subject to the following relations:
\begin{align*}
  &K_0 = 1, \\
  &K_h K_{h'} = K_{h+h'}, \\
  &K_h E_i = q^{\langle h, \alpha_i \rangle} E_i K_h, \\
  &K_h F_i = q^{\langle h, -\alpha_i \rangle} F_i K_h, \\
  &E_i F_j - F_j E_i = \delta_{i,j} \frac{K_i-K_i^{-1}}{q-q^{-1}}, \\
  &\sum_{r+s = 1-a_{i,j}} E_i^{(r)} E_j E_i^{(s)} = 0 \quad \text{ if } i \neq j, \\
  &\sum_{r+s = 1-a_{i,j}} F_i^{(r)} F_j F_i^{(s)} = 0 \quad \text{ if } i \neq j,
\end{align*}
where
\[
  K_i := K_{h_i},\ E_i^{(r)} := \frac{1}{[r]!}E_i^r,\ F_i^{(r)} := \frac{1}{[r]!}F_i^r, \ [r]! := \prod_{a=1}^r [a],\ [r] := \frac{q^r-q^{-r}}{q-q^{-1}}.
\]

The quantum group \(\mathbf{U}\) has a Hopf algebra structure with comultiplication \(\Delta\) defined as follows:
\begin{align*}
  &\Delta(E_i) := E_i \otimes 1 + K_i \otimes E_1, \\
  &\Delta(F_i) := 1 \otimes F_i + F_i \otimes K_i^{-1}, \\
  &\Delta(K_h) := K_h \otimes K_h.
\end{align*}

Let \(\rho\) denote the involutive anti-algebra automorphism on \(\mathbf{U}\) such that
\[
  \rho(E_i) = q^{-1} F_i K_i, \quad \rho(F_i) = q K_i^{-1} E_i, \quad \rho(K_h) = K_h \ \text{ for all } i \in I,\ h \in Y.
\]
A \(\mathbb{Q}(q)\)-valued symmetric bilinear form \((,)\) on a \(\mathbf{U}\)-module \(M\) is said to be contragredient if
\[
  (x m_1, m_2) = (m_1, \rho(x) m_2) \ \text{ for all } x \in \mathbf{U},\ m_1, m_2 \in M.
\]
The finite-dimensional simple \(\mathbf{U}\)-module \(V_q(\lambda) = V_{2n,q}(\lambda)\) of highest weight \(\lambda \in X^+_{2n}\) has a unique (up to nonzero scalar multiple) nondegenerate contragredient symmetric bilinear form.

Set
\[
  \mathbf{A}_\infty := \mathbb{Q}(q) \cap \mathbb{Q}[\![q^{-1}]\!].
\]
A basis \(B\) of a \(\mathbf{U}\)-module \(M\) with a contragredient symmetric bilinear form \((,)\) is said to be \emph{almost orthonormal} if
\[
  (b_1, b_2) \in \delta_{b_1, b_2} + q^{-1}\mathbf{A}_\infty \ \text{ for all } b_1,b_2 \in B.
\]
For each \(m_1, m_2 \in M\), we write \(m_1 \equiv_\infty m_2\) to indicate that
\[
  m_1 - m_2 \in q^{-1} \mathsf{Span}_{\mathbf{A}_\infty} B.
\]

Given two \(\mathbf{U}\)-modules \(M\) and \(N\) with almost orthonormal bases \(B_M\) and \(B_N\) respectively, we see that the tensor product \(B_M \otimes B_N := \{ b_1 \otimes b_2 \mid (b_1,b_2) \in B_M \times B_N \}\) forms an almost orthonormal basis of \(M \otimes N\) with respect to the natural bilinear form.

For each \(\lambda \in X^+_{2n}\), there exists a distinguished basis \(\mathbf{B}(\lambda)\), called the \emph{canonical basis}, of \(V_q(\lambda)\).
It is of the form
\[
  \mathbf{B}(\lambda) = \{ b_T \mid T \in \mathsf{KNT}^{D_n}(\lambda) \}.
\]
The canonical basis forms an almost orthonormal basis with respect to the bilinear form mentioned above.

We say that a basis \(B\) of a \(\mathbf{U}\)-module \(M\) \emph{induces a crystal base} if the pair
\[
  (\mathcal{L} := \mathsf{Span}_{\mathbf{A}_\infty} B,\ \mathcal{B} := \{ b + q^{-1}\mathcal{L} \mid b \in B \})
\]
forms a crystal base of \(M\); see e.g. \cite[Definition 1.1.1]{KaNa94} for the precise definition of crystal bases (note that our \(q\) is \(q^{-1}\) there).
When we are interested in only the crystal \(\mathcal{B}\), we say that \(B\) induces a crystal \(\mathcal{B}\).
For example, the canonical basis \(\mathbf{B}(\lambda)\) induces a crystal \(\mathcal{B}(\lambda)\) isomorphic to \(\mathsf{KNT}^{D_n}(\lambda)\).

For each \(\lambda,\mu \in X^+_{2n,0}\), there exists a \(\mathbf{U}\)-module homomorphism
\[
  {}_{\lambda}\bigwedge{}_{\mu} \colon V_q(\lambda) \otimes V_q(\mu) \to V_q(\lambda+\mu)
\]
such that
\[
  {}_{\lambda}\bigwedge{}_{\mu}(b_T \otimes b_S) \equiv_\infty b_{T*S} \quad \text{ for all } T \in \mathsf{KNT}^{D_n}(\lambda),\ S \in \mathsf{KNT}^{D_n}(\mu);
\]
see \S\ref{ssect_col_ins} for the definition of \(T*S\).
Also, there exists a \(\mathbf{U}\)-module homomorphism
\[
  {}^{\lambda}\bigvee{}^{\mu} \colon V_q(\lambda+\mu) \to V_q(\lambda) \otimes V_q(\mu)
\]
such that
\[
  {}^{\lambda}\bigvee{}^{\mu}(b_U) \equiv_\infty b_T \otimes b_S \quad \text{ for all } U \in \mathsf{KNT}^{D_n}(\lambda+\mu),
\]
where \(T \in \mathsf{KNT}^{D_n}(\lambda)\) and \(S \in \mathsf{KNT}^{D_n}(\mu)\) are unique tableaux such that \(T*S = U\).
In particular, if \(T = (\mathbf{a}^1,\dots,\mathbf{a}^m) \in \mathsf{KNT}^{D_n}\), then we have
\[
  {}_{\lambda}\bigwedge{}_{\mu}(b_{\mathbf{a}^1} \otimes b_{(\mathbf{a}^2,\dots,\mathbf{a}^m)}) \equiv_\infty b_T, \quad {}^{\lambda}\bigvee{}^{\mu}(b_T) \equiv_\infty b_{\mathbf{a}^1} \otimes b_{(\mathbf{a}^2,\dots,\mathbf{a}^m)}
\]
for appropriate \(\lambda,\mu\).

\subsection{Schur type duality}
Let \(n \in \mathbb{Z}_{\geq 1}\), and consider the quantum group \(\mathbf{U} = \mathbf{U}(\mathfrak{so}_{2n})\).
In this section, we recall the Schur type duality between \(\mathbf{U}\) and the \emph{Birman--Murakami--Wenzl algebra} \cite{Hay92}.
It will be used in \S\ref{ssect_col_mod_D} to construct certain \(\mathbf{U}\)-modules.

Let \(V_\natural := \operatorname{\mathsf{Span}}_{\mathbb{Q}(q)}\{ v_a \mid a \in \mathcal{A}^{D_n} \}\) denote the vector representation of \(\mathbf{U}\).
There exists a unique nondegenerate contragredient symmetric bilinear form on \(V_\natural\) such that \((v_1,v_1) = 1\).
The standard basis \(\{ v_a \mid a \in \mathcal{A}^{D_n} \}\) induces a crystal isomorphic to \(\mathcal{A}^{D_n}\).

Let \(\check{R} \in \mathsf{Aut}_{\mathbf{U}}(V_\natural^{\otimes 2})\) denote the R-matrix:
\[
  \check{R}(v_a \otimes v_b) :=
  \begin{cases}
    q v_a \otimes v_a & \text{ if } a = b, \\
    v_b \otimes v_a & \text{ if } b \prec a \neq \overline{b}, \\
    v_b \otimes v_a + (q-q^{-1}) v_a \otimes v_b & \text{ if } a \prec b \neq \overline{a}, \\
    q^{-1} v_b \otimes v_{\overline{b}} - (q-q^{-1}) \sum_{c \prec b} (-q)^{-d(c,a)} v_c \otimes v_{\overline{c}} & \text{ if } a = \overline{b} \succeq \overline{n} \\
    q^{-1} v_{\overline{a}} \otimes v_a + (q-q^{-1}) v_a \otimes v_{\overline{a}} \\
    -(q-q^{-1}) \sum_{c \not\succeq \overline{a}} (-q)^{-d(c,\overline{a})} v_c \otimes v_{\overline{c}} & \text{ if } b = \overline{a} \succeq \overline{n},
  \end{cases}
\]
where
\begin{align*}
  &\overline{a} := b \quad \text{ if } a = \overline{b} \succeq \overline{n}, \\
  &d(a,b) :=
  \begin{cases}
    |b-a| & \text{ if } a,b \preceq n, \\
    |\overline{b}-\overline{a}| & \text{ if } a,b \preceq \overline{n}, \\
    2n-(a+\overline{b}) & \text{ if } a \preceq n \text{ and } b \succeq \overline{n}, \\
    2n-(\overline{a}+b) & \text{ if } b \preceq n \text{ and } a \succeq \overline{n}.
  \end{cases}
\end{align*}

Let \(\mathsf{T} = \bigoplus_{d \geq 0} V_\natural^{\otimes d}\) denote the tensor algebra of \(V_\natural\).
For each \(d \in \mathbb{Z}_{\geq 0}\), let \(\mathsf{T}^d\) denote the \(d\)-th homogeneous subspace \(V_\natural^{\otimes d}\).
For each \(\mathbf{a} = (a_d,\dots,a_1) \in \mathcal{W}^{D_n}\), set
\[
  v_{\mathbf{a}} := v_{a_d} \cdots v_{a_1} \in \mathsf{T}^d.
\]
Then, \(\{ v_\mathbf{a} \mid \mathbf{a} \in \mathcal{W}^{D_n} \}\) forms an almost orthonormal basis of \(\mathsf{T}\) that induces a crystal isomorphic to \(\mathcal{W}^{D_n}\).
Hence, the Lecouvey--Robinson--Schensted correspondence (\S\ref{ssect_col_ins}) can be lifted to a \(\mathbf{U}\)-module isomorphism
\[
  \mathsf{LRS} \colon \mathsf{T} \to \bigoplus_{\lambda \in X^+_{2n,0}} V_q(\lambda) \otimes \mathsf{Span}_{\mathbb{Q}(q)} \mathsf{OT}(\lambda)
\]
such that
\[
  \mathsf{LRS}(v_\mathbf{a}) \equiv_\infty b_{\mathsf{P}(\mathbf{a})} \otimes \mathsf{Q}(\mathbf{a}) \quad \text{ for all } \mathbf{a} \in \mathcal{W}^{D_n}.
\]
Similarly, the maps \(\mathsf{w}_{\mathsf{col}}\) and \(*\) can be lifted to \(\mathbf{U}\)-module homomorphisms
\[
  \mathsf{w}_{\mathsf{col}} \colon V_q(\lambda) \to \mathsf{T}, \quad * \colon V_q(\lambda) \otimes V_q(\mu) \to \bigoplus_{\nu \in X^+_{2n,0}} V_q(\nu)
\]
such that
\[
  \mathsf{w}_{\mathsf{col}}(b_T) \equiv_\infty v_{\mathsf{w}_{\mathsf{col}}(T)},\ (b_T * b_S) \equiv_\infty b_{T*S} \quad \text{ for all } T \in \mathsf{KNT}^{D_n}(\lambda),\ S \in \mathsf{KNT}^{D_n}(\mu).
\]

Let \(d \in \mathbb{Z}_{\geq 2}\), and \(\mathsf{BMW}^{D_n}_d\) denote the \emph{Birman--Murakami--Wenzl algebra} of degree \(d\) with parameters \((q,q^{2n-1})\).
Namely, \(\mathsf{BMW}^{D_n}_d\) is the unital associative \(\mathbb{Q}(q)\)-algebra with generators \(\{ T_i \mid i \in \{ 1,\dots,d-1 \} \}\) subject to the following relations:
\begin{align*}
  &(T_i-q)(T_i+q^{-1})(T_i-q^{-2n+1}), \\
  &T_i T_{i+1} T_i = T_{i+1} T_i T_{i+1} \quad \text{ if } i < d-1, \\
  &T_i T_j = T_j T_i \quad \text{ if } |i-j| > 1, \\
  &E_{i+1} T_i^{\pm 1} E_{i+1} = q^{\pm (2n-1)} E_{i+1} \quad \text{ if } i < d-1, \\
  &E_{i-1} T_i^{\pm 1} E_{i-1} = q^{\pm (2n-1)} E_{i-1} \quad \text{ if } i > 1, \\
  &E_i T_i = q^{-2n+1} E_i = T_i E_i,
\end{align*}
where
\[
  E_i := 1-\frac{T_i-T_i^{-1}}{q-q^{-1}}.
\]

\begin{thm}[Schur type duality {\cite[Theorem 4.3]{Hay92}}]\label{thm_Schur_duality_D}
  Let \(d \in \mathbb{Z}_{\geq 2}\).
  Then, there exists a unique algebra homomorphism
  \[
    \mathsf{BMW}^{D_n}_d \to \mathsf{End}_{\mathbf{U}}(\mathsf{T}^d)
  \]
  which sends each \(T_i\) to \(\check{R}\) acting on the \(i\)-th and \((i+1)\)-th factors.
  This makes \(\mathsf{T}^d\) into a semisimple \((\mathbf{U},\mathsf{BMW}^{D_n}_d)\)-bimodule.
  Moreover, for each simple \(\mathsf{BMW}^{D_n}_d\)-module \(S\), the \(\mathbf{U}\)-module \(\mathsf{Hom}_{\mathsf{BMW}^{D_n}_d}(S,\mathsf{T}^d)\) is isomorphic to either \(V_q(\lambda)\) for some \(\lambda \in X^+_{2n,0}\) such that \(\ell(\mathsf{par}(\lambda)) < n\), or isomorphic to \(V_q(\lambda) \oplus V_q(\mu)\) for some \(\lambda,\mu \in X^+_{2n,0}\) such that \(\ell(\mathsf{par}(\lambda)) = n\), \(\mathsf{par}(\lambda) = \mathsf{par}(\mu)\), and \(\lambda \neq \mu\).
  In each case, the multiplicity of \(V_q(\lambda)\) in \(\mathsf{T}^d\) equals \(\dim S\).
\end{thm}

\subsection{Column modules}\label{ssect_col_mod_D}
Let \(n \in \mathbb{Z}_{\geq 1}\), and consider the quantum group \(\mathbf{U} = \mathbf{U}(\mathfrak{so}_{2n})\).
For each \(l \in \{ 0,\dots,n-1 \}\), define \(\gamma_l \in X^+_{2n,0}\) by
\[
  \gamma_l := \sum_{i=1}^l \epsilon_i = (\overbrace{1,\dots,1}^l,0,\dots,0).
\]
Also, set
\[
  \gamma_{n,\pm} := \gamma_{n-1} \pm \epsilon_n = (1,\dots,1,\pm 1).
\]
We will give a concrete realization of the simple \(\mathbf{U}\)-modules \(V_q(\gamma_l)\) and \(V_q(\gamma_{n,\pm})\) by means of Theorem \ref{thm_Schur_duality_D}.

\begin{defi}[Column modules]
  Let \(\mathfrak{A}\) denote the two-sided ideal of \(\mathsf{T}\) generated by \((\check{R}+q^{-1})(\mathsf{T}^2)\).
  Set
  \[
    \mathsf{C} := \mathsf{T}/\mathfrak{A}.
  \]
  Also, for each \(l \in \mathbb{Z}_{\geq 0}\), set
  \[
    \mathsf{C}_l := \mathsf{T}^l/(\mathfrak{A} \cap \mathsf{T}^l).
  \]
  We call it the \emph{column module of length \(l\)}.
\end{defi}

\begin{rem}[\(\mathbf{U}\)-module structure of \(\mathsf{C}\)]
  Since \(\check{R}\) is a \(\mathbf{U}\)-module homomorphism, the algebra \(\mathsf{C}\) and its subspaces \(\mathsf{C}_l\) are \(\mathbf{U}\)-modules.
\end{rem}

\begin{rem}[\(\mathsf{T}\)-bimodule structure of \(\mathsf{C}\)]\label{rem_bimod_mod_D}
  By the construction, the algebra \(\mathsf{C}\) is a \(\mathsf{T}\)-bimodule.
  This module structure is compatible with the \(\mathbf{U}\)-module structure:
  \[
    x \cdot (v u w) = (x_{(1)} v) \cdot (x_{(2)} u) \cdot (x_{(3)} w) \quad \text{ for all } x \in \mathbf{U},\ v,w \in \mathsf{T},\ u \in \mathsf{C},
  \]
  where we use the Sweedler notation \((\Delta \otimes 1)(\Delta(x)) = x_{(1)} \otimes x_{(2)} \otimes x_{(3)}\).
\end{rem}

For each \(\mathbf{a} \in \mathcal{W}^{D_n}\), let \(u_{\mathbf{a}}\) denote the image of \(v_{\mathbf{a}}\) in \(\mathsf{C}\).

\begin{prop}[Column modules]\label{prop_col_mod_D}
  Let \(l \in \{ 0,\dots,n \}\).
  \begin{enumerate}
    \item\label{item_hw_col_mod_D} If \(l < n\), then we have
    \[
      \mathsf{C}_l \simeq V_q(\gamma_l).
    \]
    Also, \(\mathsf{C}_n\) is decomposed as
    \[
      \mathsf{C}_n = \mathsf{C}_{n,+} \oplus \mathsf{C}_{n,-}, \quad \mathsf{C}_{n,\pm} \simeq V_q(\gamma_{n,\pm}).
    \]
    \item The set \(\{ u_{\mathbf{a}} \mid \mathbf{a} \in \mathsf{Col}^{D_n}_{l} \}\) forms a basis of \(\mathsf{C}_l\) that induces a crystal base \((\mathcal{L}(\mathsf{C}_l), \mathcal{B}(\mathsf{C}_l))\).
    Moreover, the map
    \[
      \mathsf{Col}^{D_n}_{l} \to \mathcal{B}(\mathsf{C}_l);\ \mathbf{a} \mapsto u_{\mathbf{a}} + q^{-1} \mathcal{L}(\mathsf{C}_l)
    \]
    is an isomorphism of crystals.
    \item\label{item_other_crystal_elms_col_mod_D} For each \(\mathbf{a} \in \mathcal{W}^{D_n}_l \setminus \mathsf{Col}^{D_n}_{l}\), we have
    \[
      u_{\mathbf{a}} \in q^{-1} \mathcal{L}(\mathsf{C}_l).
    \]
  \end{enumerate}
\end{prop}
\begin{proof}
  The assertions are clear for \(l \in \{ 0,1 \}\);
  \[
    \mathsf{C}_0 = \mathsf{T}^0 \simeq V_q(\gamma_0), \quad \mathsf{C}_1 = \mathsf{T}^1 \simeq V_q(\gamma_1).
  \]
  Hence, let us assume that \(l \geq 2\).
  
  It is easily seen that
  \[
    \mathfrak{A} \cap \mathsf{T}^l = \sum_{i=1}^{l-1} \mathsf{T}^l \cdot (T_i+q^{-1}).
  \]
  For each \(j,k \in \{ 1,\dots,l-1 \}\), we have
  \[
    (T_k+q^{-1}) T_j = (T_k+q^{-1})(T_j+q^{-1}) - q^{-1}(T_k+q^{-1}) \in \sum_{i=1}^{l-1} \mathsf{BMW}^{D_n}_l \cdot (T_i+q^{-1}).
  \]
  Hence, \(\mathfrak{A} \cap \mathsf{T}^l\) is a \(\mathsf{BMW}^{D_n}_l\)-submodule of \(\mathsf{T}^l\).

  We claim that \(\mathfrak{A} \cap \mathsf{T}^l \neq \mathsf{T}^l\).
  In order to prove this claim, assume contrary.
  Then, we have \(v_{l,\dots,1} \in \mathfrak{A} \cap \mathsf{T}^l\).
  Let \(\mathfrak{S}_l\) denote the \(l\)-th symmetric group with simple reflections \(s_1,\dots,s_{l-1}\).
  For each \(\sigma \in \mathfrak{S}_l\) with a reduced expression \(\sigma = s_{i_1} \cdots s_{i_r}\), set \(\ell(\sigma) := r\) and \(T_\sigma := T_{i_1} \cdots T_{i_r}\).
  Define an element \(A \in \mathsf{BMW}^{D_n}_l\) by
  \[
    A := \sum_{\sigma \in \mathfrak{S}_l} (-q)^{-\ell(\sigma)} T_\sigma.
  \]
  Then, we have
  \[
    v_{l,\dots,1} \cdot A = \sum_{\sigma \in \mathfrak{S}_l} (-q)^{-\ell(\sigma)} v_{\sigma(l),\dots,\sigma(1)} \neq 0.
  \]
  On the other hand, by the definition of \(\check{R}\), we see that \(v_{l,\dots,1} \cdot \mathsf{BMW}^{D_n}_l\) is spanned by \(\{ v_{\sigma(l),\dots,\sigma(1)} \mid \sigma \in \mathfrak{S}_l \}\), and \((T_i-q)(T_i+q^{-1}) = 0\) on this subspace.
  Hence, noting that \(A = (T_1-q) A'\), where
  \[
    A' := \sum_{\substack{\sigma \in \mathfrak{S}_l \\ \ell(s_1 \sigma) = \ell(\sigma)+1}} (-q)^{-\ell(\sigma)-1} T_\sigma,
  \]
  we obtain
  \[
    v_{l,\dots,1} T_1 A = v_{l,\dots,1} T_1 (T_1-q) A' = -q^{-1} v_{l,\dots,1} (T_1-q) A' = -q^{-1} v_{l,\dots,1} A.
  \]
  Therefore, we have
  \begin{align*}
    0
    &\neq (-q^{-1}-q^{-2n+1}) v_{l,\dots,1} A \\
    &= v_{l,\dots,1} (T_1-q^{-2n+1}) A \\
    &\in \sum_{i=1}^{l-1} \mathsf{T}^l (T_i+q^{-1}) (T_1-q^{-2n+1}) A \\
    &=\sum_{i=1}^{l-1} \mathsf{T}^l (T_i+q^{-1}) (T_1-q^{-2n+1}) (T_1-q) A' \\
    &= 0.
  \end{align*}
  This is a contradiction.
  Hence, we obtain
  \[
    \mathfrak{A} \cap \mathsf{T}^l \neq \mathsf{T}^l.
  \]

  Since \(\mathsf{T}^l\) is a semisimple \((\mathbf{U},\mathsf{BMW}^{D_n}_l)\)-bimodule (Theorem \ref{thm_Schur_duality_D}), there exists a nonzero submodule \(M' \subseteq \mathsf{T}^l\) such that
  \[
    \mathsf{T}^l = M' \oplus (\mathfrak{A} \cap \mathsf{T}^l).
  \]
  Note that \(M' \simeq \mathsf{T}^l/(\mathfrak{A} \cap \mathsf{T}^l)\).
  Since \(\mathsf{T}^l/(\mathfrak{A} \cap \mathsf{T}^l)\) is the maximal quotient on which every \(T_i\) acts as \(-q^{-1}\), the \(\mathsf{BMW}^{D_n}_l\)-module \(M'\) is a direct sum of copies of a one-dimensional (hence simple) \(\mathsf{BMW}^{D_n}_l\)-module, say \(S\).
  Also, the submodule \(\mathfrak{A} \cap \mathsf{T}^l\) has no simple factors isomorphic to \(S\).
  Hence, \(M' \simeq \mathsf{Hom}_{\mathsf{BMW}^{D_n}_l}(S, \mathsf{T}^l)\) as \(\mathbf{U}\)-modules.

  We have seen that \(u_{l,\dots,1} \in \mathsf{C}_l \simeq M'\) is nonzero.
  Clearly, it is a highest weight vector of highest weight
  \[
    \gamma :=
    \begin{cases}
      \gamma_l & \text{ if } l < n, \\
      \gamma_{n,+} & \text{ if } l = n.
    \end{cases}
  \]
  Hence, Theorem \ref{thm_Schur_duality_D} implies that
  \[
    \mathsf{C}_l \simeq
    \begin{cases}
      V_q(\gamma_l) & \text{ if } l < n, \\
      V_q(\gamma_{n,+}) \oplus V_q(\gamma_{n,-}) & \text{ if } l = n,
    \end{cases}
  \]
  and the multiplicity of \(V_q(\gamma)\) equals \(1\).
  Therefore, the subcrystal \(\mathcal{W}^{D_n}_l \setminus \mathsf{Col}^{D_n}_{l}\) has no connected components isomorphic to \(\mathsf{Col}^{D_n}_{l}\) if \(l < n\), and \(\mathsf{Col}^{D_n}_{n,\pm}\) if \(l = n\).
  By the uniqueness of crystal bases, there exists a basis \(\{ v'_{\mathbf{a}} \mid \mathbf{a} \in \mathcal{W}^{D_n}_l \}\) of \(\mathsf{T}\) such that
  \begin{itemize}
    \item \(v'_{\mathbf{a}} \equiv_\infty v_{\mathbf{a}}\) for all \(\mathbf{a} \in \mathcal{W}^{D_n}_l\),
    \item \(\{ v'_{\mathbf{a}} \mid \mathbf{a} \in \mathsf{\mathsf{C}ol}^{D_n}_{l} \}\) is a basis of \(M'\) that induces a crystal base,
    \item \(\{ v'_{\mathbf{a}} \mid \mathbf{a} \in \mathcal{W}^{D_n}_l \setminus \mathsf{Col}^{D_n}_{l} \}\) is a basis of \(\mathfrak{A} \cap \mathsf{T}^l\) that induces a crystal base,
  \end{itemize}
  Hence, the rest of the assertions follow.
\end{proof}

In the sequel, we will often identify \(\mathsf{C}_l\) with \(V_q(\gamma_l)\) if \(l < n\), and with \(V_q(\gamma_{n,+}) \oplus V_q(\gamma_{n,-})\) if \(l = n\), under the isomorphism
\[
  \begin{cases}
    V_q(\gamma_l) \\
    V_q(\gamma_{n,+}) \oplus V_q(\gamma_{n,-})
  \end{cases}
  \xrightarrow{\mathsf{w}_{\mathsf{col}}} \mathsf{T}^l
  \xrightarrow{\text{quotient}} \mathsf{C}_l.
\]
In particular, we see that \(\mathsf{C}_l\) has a nondegenerate contragredient symmetric bilinear form for which \(\{ u_\mathbf{a} \mid \mathbf{a} \in \mathsf{Col}^{D_n}_{l} \}\) is an almost orthonormal basis.

\section{Quantum symmetric pairs of type \(D\mathrm{II}\)}\label{sect_qsp_DII}
In this section, we recall the notion of quantum symmetric pairs of type \(D\mathrm{II}\), and relate the combinatorics in \S\ref{sect_manipulation_D} to representation theory.
Then, we complete the proof of one of our main result; Theorem \ref{thm_red_KNT_D_DII}.
We refer the reader to \cite{Kol14} for basics of quantum symmetric pairs.

\subsection{Symmetric pairs}
Let \(n \in \mathbb{Z}_{\geq 2}\) and set \(\mathfrak{k}\) to be the Lie subalgebra of \(\mathfrak{so}_{2n}\) generated by the following elements:
\begin{itemize}
  \item \(e_j,f_j,h_j\) for \(j \in I_\bullet := \{ 2,\dots,n \}\),
  \item \(b_1 := Z_{2,1} + Z_{1,2n-1}\).
\end{itemize}

There exists a Lie algebra homomorphism
\[
  \psi_{2n-1} \colon \mathfrak{so}_{2n-1} \to \mathfrak{so}_{2n}
\]
such that
\begin{align*}
  &\psi_{2n-1}(e_i) =
  \begin{cases}
    e_{i+1} & \text{ if } 1 \leq i < n-1, \\
    Z_{n,1}+Z_{1,n+1} & \text{ if } i = n-1,
  \end{cases} \\
  &\psi_{2n-1}(f_i) =
  \begin{cases}
    f_{i+1} & \text{ if } 1 \leq i < n-1, \\
    Z_{n+1,1}+Z_{1,n} & \text{ if } i = n-1,
  \end{cases} \\
  &\psi_{2n-1}(h_i) =
  \begin{cases}
    h_{i+1} & \text{ if } 1 \leq i < n-1, \\
    h_n-h_{n-1} & \text{ if } i = n-1.
  \end{cases}
\end{align*}
This homomorphism restricts to an isomorphism \(\mathfrak{so}_{2n-1} \to \mathfrak{k}\).

Let \(\mathfrak{g}_\bullet\) denote the Lie subalgebra of \(\mathfrak{k}\) generated by
\[
  \begin{cases}
    \{ e_j,f_j,h_j \mid j \in I_\bullet \} & \text{ if } n \geq 3, \\
    d_2 & \text{ if } n = 2.
  \end{cases}
\]
Clearly, there exists a Lie algebra homomorphism
\[
  \psi_{2n-2}' \colon \mathfrak{so}_{2n-2} \to \mathfrak{k}
\]
such that
\[
  \begin{cases}
    \psi_{2n-2}'(x_i) = x_{i+1} \ \text{ for all } i \in \{ 1,\dots,n-1 \},\ x \in \{ e,f,h \} & \text{ if } n \geq 3, \\
    \psi_2'(d_1) = d_2 & \text{ if } n = 2.
  \end{cases}
\]
This homomorphism restricts to an isomorphism \(\mathfrak{so}_{2n-2} \to \mathfrak{g}_\bullet\).
Let \(\psi_{2n-2}\) denote the composition
\[
  \mathfrak{so}_{2n-2} \xrightarrow{\psi_{2n-2}'} \mathfrak{k} \xrightarrow{\psi_{2n-1}^{-1}} \mathfrak{so}_{2n-1}.
\]
Explicitly, we have
\begin{align*}
  &\psi_{2n-2}(e_i) =
  \begin{cases}
    e_i & \text{ if } 1 \leq i < n-1, \\
    Z_{n-1,n+2} & \text{ if } i = n-1,
  \end{cases} \\
  &\psi_{2n-2}(f_i) =
  \begin{cases}
    f_i & \text{ if } 1 \leq i < n-1, \\
    Z_{n+1,n-2} & \text{ if } i = n-1,
  \end{cases} \\
  &\psi_{2n-2}(h_i) =
  \begin{cases}
    h_i & \text{ if } 1 \leq i < n-1, \\
    h_{n-2}+h_{n-1} & \text{ if } i = n-1,
  \end{cases}
\end{align*}
if \(n \geq 3\), and
\[
  \psi_2(d_1) = d_1
\]
if \(n = 2\).

\begin{prop}[Equivalences of embeddings]\label{prop_equiv_embd_D}
  \hfill
  \begin{enumerate}
    \item\label{item_B_D_prop_equiv_embd_D} Let \(\lambda \in X^+_{2n}\).
    The \(\mathfrak{so}_{2n-1}\)-modules \(V_{2n}(\lambda)\) defined via \(\phi_{2n-1}\) and \(\psi_{2n-1}\) are isomorphic.
    \item Let \(\nu \in X^+_{2n-1}\).
    The \(\mathfrak{so}_{2n-2}\)-modules \(V_{2n-1}(\nu)\) defined via \(\phi_{2n-2}\) and \(\psi_{2n-2}\) are isomorphic.
  \end{enumerate}
\end{prop}
\begin{proof}
  For each \(i \in \{ 1,\dots,n-1 \}\), we have
  \[
    \phi_{2n-1}(d_i) = d_i, \quad \psi_{2n-1}(d_i) = d_{i+1}, \quad \psi_{2n-2}(d_i) = d_i = \psi_{2n-2}(d_i).
  \]
  This implies that, in each case, the two modules in question have the same character.
  Hence, they are isomorphic.
\end{proof}

\begin{cor}[Branching rule for symmetric pairs of type \(D\mathrm{II}\)]\label{cor_branching_rule_classical_DII}
  \hfill
  \begin{enumerate}
    \item Let \(\lambda \in X^+_{2n}\).
    Then, we have
    \[
      V_{2n}(\lambda)|_{\mathfrak{k}} \simeq \bigoplus_{\substack{\nu \in X_{2n-1, \mathsf{s}(\lambda)} \\ \mathsf{par}(\nu) \overset{\textsf{hor}}{\subseteq} \mathsf{par}(\lambda)}} V_{2n-1}(\nu).
    \]
    Here, we regard the \(\mathfrak{so}_{2n-1}\)-module \(V_{2n-1}(\nu)\) as a \(\mathfrak{k}\)-module via the isomorphism \(\psi_{2n-1}^{-1} \colon \mathfrak{k} \to \mathfrak{so}_{2n-1}\).
    \item Let \(\nu \in X^+_{2n-1}\).
    Then, we have
    \[
      V_{2n-1}(\nu)|_{\mathfrak{g}_\bullet} \simeq \bigoplus_{\substack{\mu \in X_{2n-2, \mathsf{s}(\nu)} \\ \mathsf{par}(\mu) \overset{\textsf{hor}}{\subseteq} \mathsf{par}(\nu)}} V_{2n-2}(\mu).
    \]
    Here, we regard the \(\mathfrak{so}_{2n-2}\)-module \(V_{2n-2}(\mu)\) as a \(\mathfrak{g}_\bullet\)-module via the isomorphism \(\psi_{2n-2}^{-1} \colon \mathfrak{g}_\bullet \to \mathfrak{so}_{2n-2}\).
  \end{enumerate}
\end{cor}
\begin{proof}
  The assertions follow from Proposition \ref{prop_equiv_embd_D} and equation \eqref{eq_branching_so_partition}.
\end{proof}

\subsection{Quantum symmetric pairs}
Let \(n \in \mathbb{Z}_{\geq 2}\) and set \(\mathbf{U} := \mathbf{U}(\mathfrak{so}_{2n})\).

When \(n \geq 3\), the \emph{\(\imath\)quantum group of type \(D\mathrm{II}_n\)} is the subalgebra \(\mathbf{U}^\imath\) of \(\mathbf{U}\) generated by the following elements:
\begin{itemize}
  \item \(E_j,F_j,K_j^{\pm 1}\) for \(j \in I_\bullet := \{ 2,\dots,n \}\),
  \item \(B_1 := F_1 + (-q)^{n-2} T_2 \cdots T_{n-2} T_n T_{n-1} T_{n-2} \cdots T_2(E_1) K_1^{-1}\),
\end{itemize}
where \(T_i \in \mathsf{Aut}_{\textsf{\(\mathbb{Q}(q)\)-alg}}(\mathbf{U})\) denotes the Lusztig braid group action \(T''_{i,1}\) \cite[\S 37.1]{Lus93}:
\begin{align*}
  &T_i(E_j) :=
  \begin{cases}
    -F_iK_i & \text{ if } a_{i,j}=2, \\
    E_j & \text{ if } a_{i,j} = 0, \\
    E_iE_j - q^{-1}E_jE_i & \text{ if } a_{i,j} = -1,
  \end{cases} \\
  &T_i(F_j) :=
  \begin{cases}
    -K_i^{-1}E_i & \text{ if } a_{i,j} = 2,\\
    F_j & \text{ if } a_{i,j} = 0, \\
    F_jF_i - q F_iF_j & \text{ if } a_{i,j} = -1,
  \end{cases} \\
  &T_i(K_{h}) := K_{h-\langle h, \alpha_i \rangle h_i}.
\end{align*}
Also, we set \(\mathbf{U}_\bullet\) to be the subalgebra of \(\mathbf{U}^\imath\) generated by \(\{ E_j,F_j,K_j^{\pm 1} \mid j \in I_\bullet \}\).
Clearly, we have an algebra isomorphism
\begin{align}\label{eq_isom_qg_Udot}
  \mathbf{U}(\mathfrak{so}_{2n-2}) \to \mathbf{U}_\bullet;\ X_i \mapsto X_{i+1} \quad (X = E,F,K).
\end{align}

When \(n = 2\), we define \(\mathbf{U}^\imath\) to be the subalgebra of \(\mathbf{U}\) generated by the following elements:
\[
  F_1 + E_2K_1^{-1},\ F_2 + E_1K_2^{-1},\ K_{2d_2}^{\pm 1}.
\]
Also, we set \(\mathbf{U}_\bullet\) to be the subalgebra of \(\mathbf{U}^\imath\) generated by \(K_{2d_2}^{\pm 1}\).
Clearly, we have an algebra isomorphism
\begin{align}\label{eq_isom_qg_Udot_2}
  \mathbf{U}(\mathfrak{so}_2) \to \mathbf{U}_\bullet;\ K_{2d_1} \mapsto K_{2d_2}.
\end{align}

The \(\imath\)quantum group \(\mathbf{U}^\imath\) is a right coideal of \(\mathbf{U}\):
\[
  \Delta(\mathbf{U}^\imath) \subset \mathbf{U}^\imath \otimes \mathbf{U}
\]
Also, it tends to the universal enveloping algebra \(U(\mathfrak{so}_{2n-1})\) as \(q\) tends to \(1\) (see \cite[\S 10]{Kol14} for a precise meaning).

The \(\imath\)quantum group \(\mathbf{U}^\imath\) is closed under the involution \(\rho\) (\cite[Proposition 4.6]{BaWa18}).
Hence, we can consider the notions of contragredient symmetric bilinear forms and almost orthonormal bases of \(\mathbf{U}^\imath\)-modules.

\begin{thm}[Branching rules for quantum symmetric pairs of type \(D\mathrm{II}\)]\label{thm_branching_rule_DII}
  \hfill
  \begin{enumerate}
    \item For each \(\nu \in X^+_{2n-1}\), there exists a unique simple classical weight \textup{(}in the sense of \cite{Wat21b}\textup{)} \(\mathbf{U}^\imath\)-module \(V^\imath(\nu)\) which tends to the simple \(\mathfrak{so}_{2n-1}\)-module \(V_{2n-1}(\nu)\) as \(q\) tends to \(1\) \textup{(}in the sense of \cite[\S 3.4]{HoKa02} and \cite[\S 10]{Kol14}\textup{)}.
    \item\label{item_D_to_DII_thm_branching_rule_DII} Let \(\lambda \in X^+_{2n}\).
    Then, we have
    \[
      V_q(\lambda)|_{\mathbf{U}^\imath} \simeq \bigoplus_{\substack{\nu \in X^+_{2n-1, \mathsf{s}(\lambda)} \\ \mathsf{par}(\nu) \overset{\textsf{hor}}{\subseteq} \mathsf{par}(\lambda)}} V^\imath(\nu).
    \]
    \item Let \(\nu \in X^+_{2n-1}\).
    Then, we have
    \[
      V^\imath(\nu)|_{\mathbf{U}_\bullet} \simeq \bigoplus_{\substack{\mu \in X^+_{2n-2, \mathsf{s}(\nu)} \\ \mathsf{par}(\mu) \overset{\textsf{hor}}{\subseteq} \mathsf{par}(\nu)}} V_{2n-2,q}(\mu).
    \]
    Here, we regard the simple \(\mathbf{U}(\mathfrak{so}_{2n-2})\)-module \(V_{2n-2,q}(\mu)\) as a \(\mathbf{U}_\bullet\)-module via the isomorphism \eqref{eq_isom_qg_Udot} or \eqref{eq_isom_qg_Udot_2}.
  \end{enumerate}
\end{thm}
\begin{proof}
  When \(n \geq 3\), the assertions follow from \cite[Corollary 6.8, Theorem 6.11]{KoSt24} and Corollary \ref{cor_branching_rule_classical_DII}.

  In the rest of the proof, let us assume that \(n = 2\).
  There exists an algebra isomorphism \(\mathbf{U}(\mathfrak{so}_3) \to \mathbf{U}^\imath\) (see \S\ref{ssect_qg_B} for the definition of \(\mathbf{U}(\mathfrak{so}_3)\)) such that
  \[
    E_1 \mapsto F_1+E_2K_1^{-1},\ F_1 \mapsto F_2+E_1K_2^{-1},\ K_{2d_1} \mapsto K_{2d_2}.
  \]
  Hence, for each \(\nu \in X^+_3 = \frac{1}{2}\mathbb{Z}_{\geq 0}\), there exists a (\(2\nu+1\))-dimensional \(\mathbf{U}^\imath\)-module \(V^\imath(\nu)\) on which \(K_{2d_1}\) acts diagonally with eigenvalues \(q^{2\nu}, q^{2\nu-2},\dots,q^{-2\nu}\).
  This proves the first assertion.
  The others follow from Corollary \ref{cor_branching_rule_classical_DII}.
\end{proof}

\subsection{Connecting homomorphisms}\label{ssect_conn_hom_D}
Let \(n \in \mathbb{Z}_{\geq 2}\), and consider the quantum symmetric pair \((\mathbf{U} = \mathbf{U}(\mathfrak{so}_{2n}), \mathbf{U}^{\imath})\), the vector representation \(V_\natural\) of \(\mathbf{U}\), and column modules \(\mathsf{C}_l\), \(\mathsf{C}_{n,\pm}\) from \S\ref{sect_qg_D}.
\begin{prop}[Vector \(v^\imath_0\)]\label{prop_vi0_D}
  Set
  \[
    v^\imath_0 := v_1 + (-q)^{-n+1} v_{\overline{1}} \in V_\natural.
  \]
  Then, it spans a \(\mathbf{U}^\imath\)-submodule of \(V_\natural\) isomorphic to the trivial module.
\end{prop}
\begin{proof}
  The assertion can be verified by direct calculation (\textit{cf}.\, \cite[Lemma 4.4.2]{Wat23b}).
\end{proof}

\begin{defi}[Connecting homomorphisms]
  Let \(l \in \{ 0,\dots,n-1 \}\).
  The \emph{connecting homomorphism \(\delta_l\)} is the map
  \[
    \delta_l \colon \mathsf{C}_l \to \mathsf{C}_{l+1};\ u \mapsto (-q)^l v^\imath_0 \cdot u,
  \]
  where \(v^\imath_0\) is as in Proposition \ref{prop_vi0_D}.
  Also, the \emph{connecting homomorphism \(\delta_{n-1,\pm}\)} is the map
  \[
    \delta_{n-1,\pm} \colon \mathsf{C}_{n-1} \xrightarrow{\delta_{n-1}} \mathsf{C}_n = \mathsf{C}_{n,+} \oplus \mathsf{C}_{n,-} \to \mathsf{C}_{n,\pm},
  \]
  where the second map is the projection.
\end{defi}

\begin{prop}
  The connecting homomorphisms are \(\mathbf{U}^\imath\)-module homomorphisms.
\end{prop}
\begin{proof}
  The assertion follows from Proposition \ref{prop_vi0_D}; see also Remark \ref{rem_bimod_mod_D}.
\end{proof}

\begin{lem}[Relations in \(\mathsf{C}\)]\label{lem_rel_col_mod_D}
  In the algebra \(\mathsf{C}\), the following hold\textup{:}
  \begin{enumerate}
    \item \(u_{a,a} = 0\) for all \(a \in \mathcal{A}^{D_n}\),
    \item \(u_{a,b} = -q^{-1} u_{b,a}\) for all \(a,b \in \mathcal{A}^{D_n}\) such that \(a \prec b \neq \overline{a}\),
    \item \(u_{1,\overline{1}} = \sum_{a=2}^{n-1} (-q)^{-a+1} u_{\overline{a},a} + \frac{(-q)^{-n+1}}{1-q^{-2}}(u_{\overline{n},n} + u_{n,\overline{n}})\),
    \item \(u_{\overline{1},1} = -u_{1,\overline{1}}\).
  \end{enumerate}
\end{lem}
\begin{proof}
  The assertions can be straightforwardly deduced from the definition of \(\mathsf{C}\).
\end{proof}

\begin{prop}[Connecting homomorphisms]\label{prop_conn_hom_D}
  Let \(l \in \{ 0,\dots,n-1 \}\).
  For each \(\mathbf{a} \in \mathsf{Col}^{D_n}_{l}\), the following hold\textup{:}
  \begin{enumerate}
    \item If \(l < n-1\), then
    \[
      \delta_l(u_{\mathbf{a}}) \equiv_\infty
      \begin{cases}
        u_{\delta_l(\mathbf{a})} & \text{ if } \mathbf{a} \in \mathsf{Col}^{D\mathrm{II}_n}_{l}, \\
        0 & \text{ if } \mathbf{a} \notin \mathsf{Col}^{D\mathrm{II}_n}_{l}.
      \end{cases}
    \]
    \item If \(l = n-1\), then
    \begin{align*}
      &\delta_{n-1}(u_{\mathbf{a}}) \equiv_\infty
      \begin{cases}
        u_{\delta_{n-1,+}(\mathbf{a})} + u_{\delta_{n-1,-}(\mathbf{a})} & \text{ if } \mathbf{a} \in \mathsf{Col}^{D\mathrm{II}_n}_{n-1}, \\
        0 & \text{ if } \mathbf{a} \notin \mathsf{Col}^{D\mathrm{II}_n}_{n-1},
      \end{cases} \\
      &\delta_{n-1,\pm}(u_{\mathbf{a}}) \equiv_\infty
      \begin{cases}
        u_{\delta_{n-1,\pm}(\mathbf{a})} & \text{ if } \mathbf{a} \in \mathsf{Col}^{D\mathrm{II}_n}_{n-1}, \\
        0 & \text{ if } \mathbf{a} \notin \mathsf{Col}^{D\mathrm{II}_n}_{n-1}.
      \end{cases} \\
    \end{align*}
  \end{enumerate}
\end{prop}
\begin{proof}
  We will compute the connecting homomorphisms by means of Lemma \ref{lem_rel_col_mod_D}.
  
  First, let us assume that \(1 \in \mathsf{V}(\mathbf{a})\).
  Then, we have
  \begin{align*}
    \delta_l(u_{\mathbf{a}})
    &= (-q)^l u_{1*\mathbf{a}} + (-q)^{-n+l+1} u_{\overline{1}*\mathbf{a}} \\
    &= u_{\mathbf{a}*1} + (-q)^{-n+l+1} u_{\overline{1}*\mathbf{a}} \\
    &\equiv_\infty
    \begin{cases}
      u_{\mathbf{a}*1} & \text{ if } l < n-1, \\
      u_{\mathbf{a}*1} + u_{\overline{1}*\mathbf{a}} & \text{ if } l = n-1,
    \end{cases} \\
    &\equiv_\infty
    \begin{cases}
      0 & \text{ if } \mathbf{a} \notin \mathsf{Col}^{D\mathrm{II}_n}_{l}, \\
      u_{\mathbf{a}*1} & \text{ if } \mathbf{a} \in \mathsf{Col}^{D\mathrm{II}_n}_{l} \text{ and } l < n-1, \\
      u_{\mathbf{a}*1} + u_{\overline{1}*\mathbf{a}} & \text{ if } \mathbf{a} \in \mathsf{Col}^{D\mathrm{II}_n}_{l} \text{ and } l = n-1,
    \end{cases} \\
    &=
    \begin{cases}
      0 & \text{ if } \mathbf{a} \notin \mathsf{Col}^{D\mathrm{II}_n}_{l}, \\
      u_{\delta_l(\mathbf{a})} & \text{ if } \mathbf{a} \in \mathsf{Col}^{D\mathrm{II}_n}_{l} \text{ and } l < n-1, \\
      u_{\delta_{n-1,+}(\mathbf{a})} + u_{\delta_{n-1,-}(\mathbf{a})} & \text{ if } \mathbf{a} \in \mathsf{Col}^{D\mathrm{II}_n}_{l} \text{ and } l = n-1.
    \end{cases}
  \end{align*}
  Here, we used Propositions \ref{prop_ins_1_DII}, \ref{prop_ins_1b_DII}, and \ref{prop_col_mod_D} \eqref{item_other_crystal_elms_col_mod_D}.

  Next, let us assume that \(\overline{1} \in \mathbf{a}\).
  Set \(\mathbf{b} := \mathbf{a}_{> 1}\).
  Then, we have
  \begin{align*}
    \delta(u_{\mathbf{a}})
    &= (-q)^l u_{1*\mathbf{a}} + (-q)^{-n+l+1} u_{\overline{1}*\mathbf{a}} \\
    &= (-q)^l u_{1*\mathbf{a}} \\
    &= (-q)^l u_{1*\overline{1}*\mathbf{b}} \\
    &= (-q)^l u_{\mathbf{b}^{\succ \overline{n}}*\mathbf{b}_n*1*\overline{1}*\mathbf{b}^{\prec n}} \\
    &= \sum_{1 < a < n} (-q)^{l-a+1} u_{\mathbf{b}^{\succ \overline{n}}*\mathbf{b}_n*\overline{a}*a*\mathbf{b}^{\prec n}} + \frac{(-q)^{l-n+1}}{1-q^{-2}}(u_{\mathbf{b} \Leftarrow_+ n} + u_{\mathbf{b} \Leftarrow_- n}) \\
    &= \sum_{\substack{1 < a < n \\ a \in \mathsf{V}(\mathbf{a})}} (-q)^{|\mathbf{b}_{< a}|-a+2} u_{\mathbf{b} \Leftarrow a} + \frac{(-q)^{l-n+1}}{1-q^{-2}}(u_{\mathbf{b} \Leftarrow_+ n} + u_{\mathbf{b} \Leftarrow_- n}) \\
    &\equiv_\infty \sum_{\substack{1 < a < n \\ a \in \mathsf{PV}(\mathbf{a})}} u_{\mathbf{b} \Leftarrow a} + \frac{(-q)^{l-n+1}}{1-q^{-2}}(u_{\mathbf{b} \Leftarrow_+ n} + u_{\mathbf{b} \Leftarrow_- n}) \\
    &\equiv_\infty
    \begin{cases}
      u_{\mathbf{b} \Leftarrow \mathsf{gpv}(\mathbf{a})} & \text{ if } l < n-1, \\
      u_{\mathbf{b} \Leftarrow \mathsf{gpv}(\mathbf{a})} + u_{\mathbf{b} \Leftarrow_{\mathsf{e}(\mathbf{b})} n} & \text{ if } l = n-1 \text{ and } \mathsf{gpv}(\mathbf{a}) < n, \\
      u_{\mathbf{b} \Leftarrow_+ n} + u_{\mathbf{b} \Leftarrow_- n} & \text{ if } l = n-1 \text{ and } \mathsf{gpv}(\mathbf{a}) = n,
    \end{cases} \\
    &=
    \begin{cases}
      u_{\delta_l(\mathbf{a})} & \text{ if } l < n-1, \\
      u_{\delta_{n-1,+}(\mathbf{a})} + u_{\delta_{n-1,-}(\mathbf{a})} & \text{ if } l= n-1.
    \end{cases}
  \end{align*}
  Here, we used Lemma \ref{lem_vac_num_DII} and Proposition \ref{prop_ins_pv_DII}.
  When \(l = n\), since \(\delta_{n-1,\pm}(\mathbf{a}) \in \mathsf{Col}^{D_n}_{n,\pm}\), we further obtain
  \[
    \delta_{n-1,\pm}(u_\mathbf{a}) \equiv_\infty u_{\delta_{n-1,\pm}(\mathbf{a})}.
  \]

  Finally, let us assume that \(1 \in \mathbf{a}\).
  Then, we have
  \begin{align*}
    \delta_l(u_{\mathbf{a}})
    &= (-q)^l u_{1*\mathbf{a}} + (-q)^{-n+l+1} u_{\overline{1}*\mathbf{a}} \\
    &= (-q)^{-n+l+1} u_{\overline{1}*\mathbf{a}} \\
    &= (-q)^{-n+2l} u_{\overline{1}*1*\mathbf{a}_{> 1}} \\
    &= -(-q)^{-n+2l} u_{1*\overline{1}*\mathbf{a}_{> 1}} \\
    &= -(-q)^{-n+l} \cdot (-q)^l u_{1*\overline{1}*\mathbf{a}_{> 1}}.
  \end{align*}
  We have already seen that \((-q)^l u_{1*\overline{1}*\mathbf{a}_{> 1}} \in \mathcal{L}(\mathsf{C}_{l+1})\).
  Since \(l < n\), we obtain \(\delta_l(u_{\mathbf{a}}) \equiv_\infty 0\).

  Thus, we complete the proof.
\end{proof}

\begin{prop}[Decomposition of \(\mathsf{C}_l|_{\mathbf{U}^\imath}\)]\label{prop_decomp_col_mod_D}
  Let \(l \in \{ 0,\dots,n-1 \}\).
  The column module \(\mathsf{C}_l\) has a basis of the form \(\{ u'_{\mathbf{a}} \mid \mathbf{a} \in \mathsf{Col}^{D_n}_{l} \}\) satisfying the following\textup{:}
  \begin{itemize}
    \item \(u'_{\mathbf{a}} \equiv_\infty u_{\mathbf{a}}\) for all \(\mathbf{a} \in \mathsf{Col}^{D_n}_{l}\),
    \item \(\{ u'_{\mathbf{a}} \mid \mathbf{a} \in \mathsf{Col}^{D\mathrm{II}_n}_{l} \}\) spans a \(\mathbf{U}^\imath\)-submodule \(\mathsf{C}'_l\) isomorphic to \(V^\imath(\gamma_l)\),
    \item \(\{ u'_{\mathbf{a}} \mid \mathbf{a} \in \mathsf{Col}^{D_n}_{l} \setminus \mathsf{Col}^{D\mathrm{II}_n}_{l} \}\) spans a \(\mathbf{U}^\imath\)-submodule \(\mathsf{C}''_l\) isomorphic to \(V^\imath(\gamma_{l-1})\).
  \end{itemize}
  Moreover, the connecting homomorphisms restricts to \(\mathbf{U}^\imath\)-module isomorphisms
  \[
    \mathsf{C}'_l \xrightarrow{\delta_l} \mathsf{C}''_{l+1}, \quad \mathsf{C}'_{n-1} \xrightarrow{\delta_{n,\pm}} \mathsf{C}_{n,\pm}.
  \]
\end{prop}
\begin{proof}
  The assertion can be proved by induction on \(l\) by means of Proposition \ref{prop_conn_hom_D} and the contragredient symmetric bilinear form on \(\mathsf{C}_l\); see after Proposition \ref{prop_col_mod_D}.
\end{proof}

In the sequel, we often identify the simple \(\mathbf{U}^\imath\)-module \(V^\imath(\gamma_l)\) with the submodule \(\mathsf{C}'_l \subseteq \mathsf{C}_l\) in Proposition \ref{prop_decomp_col_mod_D}.
In particular, we see that \(V^\imath(\gamma_l)\) has a nondegenerate contragredient symmetric bilinear form and an almost orthonormal basis \(\{ b^\imath_{\mathbf{a}} \mid \mathbf{a} \in \mathsf{Col}^{D\mathrm{II}_n}_{l} \}\) corresponding to the basis \(\{ u'_{\mathbf{a}} \mid \mathbf{a} \in \mathsf{Col}^{D\mathrm{II}_n}_{l} \}\) of \(\mathsf{C}_l'\).

\begin{thm}[Reduction homomorphisms from \(D_n\) to \(D\mathrm{II}_n\)]\label{thm_red_col_D_DII}
  \hfill
  \begin{enumerate}
    \item Let \(l \in \{ 0,\dots,n-1 \}\). Then, there exists a \(\mathbf{U}^\imath\)-module isomorphism
    \[
      \cdot \downarrow^{D_n}_{D\mathrm{II}_n} \colon V_q(\gamma_l) \to V^\imath(\gamma_l) \oplus V^\imath(\gamma_{l-1})
    \]
    such that
    \[
      b_\mathbf{a} \downarrow^{D_n}_{D\mathrm{II}_n} \equiv_\infty b^\imath_{\mathbf{a} \downarrow^{D_n}_{D\mathrm{II}_n}} \quad \text{ for all } \mathbf{a} \in \mathsf{Col}^{D_n}_{l}.
    \]
    \item There exists a \(\mathbf{U}^\imath\)-module isomorphism
    \[
      \cdot \downarrow^{D_n}_{D\mathrm{II}_n} \colon V_q(\gamma_{n,\pm}) \to V^\imath(\gamma_{n-1})
    \]
    such that
    \[
      b_\mathbf{a} \downarrow^{D_n}_{D\mathrm{II}_n} \equiv_\infty b^\imath_{\mathbf{a} \downarrow^{D_n}_{D\mathrm{II}_n}} \quad \text{ for all } \mathbf{a} \in \mathsf{Col}^{D_n}_{n,\pm}.
    \]
  \end{enumerate}
  
\end{thm}
\begin{proof}
  With notations in Proposition \ref{prop_decomp_col_mod_D}, let us consider the \(\mathbf{U}^\imath\)-module homomorphism
  \[
    \mathsf{C}_l = \mathsf{C}'_l \oplus \mathsf{C}''_l \to \mathsf{C}'_l \oplus \mathsf{C}'_{l-1},
  \]
  which is the identity map on \(\mathsf{C}'_l\), and the inverse of the isomorphism \(\delta_{l-1} \colon \mathsf{C}'_{l-1} \to \mathsf{C}''_l\) on \(\mathsf{C}''_l\).
  Then, the first assertion follows.

  The second assertion can be proved similarly.
  Hence, we complete the proof.
\end{proof}

\subsection{Almost orthonormal bases}
In this section, we construct an almost orthonormal basis of each \(V^\imath(\nu)\) parametrized by \(\mathsf{KNT}^{D\mathrm{II}_n}(\nu)\).
\begin{lem}[Decomposition of \(V_{2n,q}(\nu)\)]\label{lem_compatible_basis_DII}
  Let \(\nu \in X^+_{2n-1,0}\).
  Then, there exists a basis \(B'(\nu)\) of \(V_{2n,q}(\nu)\) of the form
  \[
    B'(\nu) = \{ b'_T \mid T \in \mathsf{KNT}^{D_n}(\nu) \}
  \]
  satisfying the following\textup{:}
  \begin{itemize}
    \item \(b'_T \equiv_\infty b_T\) for all \(T \in \mathsf{KNT}^{D_n}(\nu)\).
    \item \(M' := \operatorname{\mathsf{Span}}_{\mathbb{Q}(q)} \{ b'_T \mid T \in \mathsf{KNT}^{D\mathrm{II}_n}(\nu) \}\) is a \(\mathbf{U}^\imath\)-submodule of \(V_q(\nu)|_{\mathbf{U}^\imath}\) isomorphic to \(V^\imath(\nu)\).
    \item \(V_{2n,q}(\nu)|_{\mathbf{U}^\imath} = M' \oplus M''\), where
    \[
      M'' := \operatorname{\mathsf{Span}}_{\mathbb{Q}(q)} \{ b'_T \mid T \in \mathsf{KNT}^{D_n}(\nu) \setminus \mathsf{KNT}^{D\mathrm{II}_n}(\nu) \}.
    \]
  \end{itemize}
\end{lem}
\begin{proof}
  The assertion is clear when \(\nu = 0\).
  Hence, let us assume that \(\nu \neq 0\).
  Let \(l\) denote the length of \(\mathsf{par}(\nu)\).
  Then, we have \(\nu-\gamma_l \in X^+_{2n,0}\).
  Consider the composite
  \begin{align*}
    g \colon V_{2n,q}(\nu)
    &\xrightarrow{{}^{\gamma_l}\bigvee{}^{\xi}} V_{q}(\gamma_l) \otimes V_q(\nu-\gamma_l) \\
    &\xrightarrow{\cdot \downarrow^{D_n}_{D\mathrm{II}_n} \otimes 1} (V^\imath(\gamma_l) \oplus V^\imath(\gamma_{l-1})) \otimes V_q(\nu-\gamma_l) \\
    &\xrightarrow{\text{projection}} V^\imath(\gamma_{l-1}) \otimes V_q(\nu-\gamma_l) \\
    &\xrightarrow{\delta_{l-1} \otimes 1} V_{q}(\gamma_l) \otimes V_q(\nu-\gamma_l) \\
    &\xrightarrow{{}_{\gamma_l}\bigwedge{}_{\xi}} V_{2n,q}(\nu).
  \end{align*}
  By Theorem \ref{thm_red_col_D_DII} and Proposition \ref{prop_conn_hom_D}, for each \(T \in \mathsf{KNT}^{D_n}(\nu)\), we have
  \begin{align}\label{eq_g_at_infty}
    g(b_T) \equiv_\infty
    \begin{cases}
      0 & \text{ if } T \in \mathsf{KNT}^{D\mathrm{II}_n}(\nu), \\
      b_T & \text{ if } T \notin \mathsf{KNT}^{D\mathrm{II}_n}(\nu).
    \end{cases}
  \end{align}

  Recall from Theorem \ref{thm_branching_rule_DII} \eqref{item_D_to_DII_thm_branching_rule_DII} that
  \[
    V_{2n,q}(\nu)|_{\mathbf{U}^\imath} \simeq V^\imath(\nu) \oplus \bigoplus_{\substack{\xi \in X^+_{2n-1,0} \\ \xi \overset{\textsf{hor}}{\subset} \nu}} V^\imath(\xi).
  \]
  Let \(V_{2n,q}(\nu)|_{\mathbf{U}^\imath} = M' \oplus M''\) denote the corresponding decomposition.
  Since the \(\mathbf{U}^\imath\)-module \((V_{2n,q}(\gamma_{l-1}) \otimes V_{2n,q}(\nu-\gamma_l))|_{\mathbf{U}^\imath}\) does not have \(V^\imath(\nu)\) as an irreducible component, we have
  \begin{align}\label{eq_lower_bound_ker}
    \dim \operatorname{\mathsf{Ker}} g \geq \dim M' = |\mathsf{KNT}^{D\mathrm{II}_n}(\nu)|,
  \end{align}
  where the equality follows from Theorem \ref{thm_red_map_DII_D}.
  On the other hand, by equation \eqref{eq_g_at_infty}, we have
  \begin{align}\label{eq_lower_bound_im}
    \dim \operatorname{\mathsf{Im}} g \geq |\mathsf{KNT}^{D_n}(\nu) \setminus \mathsf{KNT}^{D\mathrm{II}_n}(\nu)| = \dim M''.
  \end{align}
  Therefore, the inequalities in \eqref{eq_lower_bound_ker} and \eqref{eq_lower_bound_im} must be equalities.
  In particular, the image of \(g\), which coincides with \(M''\), has a basis
  \[
    \{ b'_T := g(b_T) \mid T \in \mathsf{KNT}^{D_n}(\nu) \setminus \mathsf{KNT}^{D\mathrm{II}_n}(\nu) \}.
  \]
  Hence, the orthogonal complement of \(M''\), which coincides with \(M'\), has a basis of the form
  \[
    \{ b'_T \mid T \in \mathsf{KNT}^{D\mathrm{II}_n}(\nu) \}
  \]
  such that
  \[
    b'_T \equiv_\infty b_T \quad \text{ for all } T \in \mathsf{KNT}^{D\mathrm{II}_n}(\nu).
  \]
  Thus, we complete the proof.
\end{proof}

\begin{thm}[Almost orthonormal basis of \(V^\imath(\nu)\)]\label{thm_alm_orthn_basis_Vinu_DII}
  Let \(\nu \in X^+_{2n-1,0}\).
  Then, the \(\mathbf{U}^\imath\)-module \(V^\imath(\nu)\) has a basis \(B^\imath(\nu)\) of the form
  \[
    B^\imath(\nu) = \{ b^\imath_T \mid T \in \mathsf{KNT}^{D\mathrm{II}_n}(\nu) \}
  \]
  satisfying the following:
  \begin{enumerate}
    \item There exists a \(\mathbf{U}^\imath\)-module homomorphism \(\iota_\nu \colon V^\imath(\nu) \to V_{2n,q}(\nu)|_{\mathbf{U}^\imath}\) such that
    \[
      \iota_\nu(b^\imath_T) \equiv_\infty b_T \quad \text{ for all } T \in \mathsf{KNT}^{D\mathrm{II}_n}(\nu).
    \]
    \item There exists a contragredient symmetric bilinear form \((,)\) on \(V^\imath(\nu)\) for which the basis \(B^\imath(\nu)\) is almost orthonormal.
    \item There exists a \(\mathbf{U}^\imath\)-module homomorphism \(p_\nu \colon V_{2n,q}(\nu) \to V^\imath(\nu)\) such that
    \[
      p_\nu(b_T) \equiv_\infty
      \begin{cases}
        b^\imath_T & \text{ if } T \in \mathsf{KNT}^{D\mathrm{II}_n}(\nu), \\
        0 & \text{ if } T \notin \mathsf{KNT}^{D\mathrm{II}_n}(\nu),
      \end{cases}
    \]
    for all \(T \in \mathsf{KNT}^{D_n}(\nu)\).
  \end{enumerate}
\end{thm}
\begin{proof}
  As for \(V^\imath(\gamma_l)\)'s, let us identify \(V^\imath(\nu)\) with the submodule \(M' \subseteq V_{2n,q}(\nu)\) in Lemma \ref{lem_compatible_basis_DII}.
  Let \(B^\imath(\nu)\) denote the basis of \(V^\imath(\nu)\) corresponding to the basis \(\{ b'_T \mid T \in \mathsf{KNT}^{D\mathrm{II}_n}(\nu) \}\) of \(M'\).
  Then, the assertions are immediate from Lemma \ref{lem_compatible_basis_DII}.
\end{proof}

\begin{cor}[Reduction from \(D\mathrm{II}_n\) to \(D_{n-1}\)]
  Let \(\nu \in X^+_{2n-1,0}\).
  Then, there exists a \(\mathbf{U}_\bullet\)-module isomorphism
  \[
    \cdot \downarrow^{D\mathrm{II}_n}_{D_{n-1}} \colon V^\imath(\nu)|_{\mathbf{U}_\bullet} \to \bigoplus_{\substack{\mu \in X^+_{2n-2,0} \\ \mathsf{par}(\mu) \overset{\textsf{hor}}{\subseteq} \mathsf{par}(\nu)}} V_{2n-2,q}(\mu)
  \]
  such that
  \[
    b^\imath_T \downarrow^{D\mathrm{II}_n}_{D_{n-1}} \equiv_\infty b_{T \downarrow^{D\mathrm{II}_n}_{D_{n-1}}} \quad \text{ for all } T \in \mathsf{KNT}^{D\mathrm{II}_n}(\nu).
  \]
\end{cor}
\begin{proof}
  The basis of \(V^\imath(\nu)\) in Theorem \ref{thm_alm_orthn_basis_Vinu_DII} induces a \(D_{n-1}\)-crystal isomorphic to \(\mathsf{KNT}^{D\mathrm{II}_n}(\nu)\).
  Hence, the assertion follows from Theorem \ref{thm_red_map_DII_D}.
\end{proof}

\subsection{Reduction from \(D_n\) to \(D\mathrm{II}_n\)}\label{ssect_red_D_DII}
In this section, we complete the proof of Theorem \ref{thm_red_KNT_D_DII} by lifting the map to a \(\mathbf{U}^{\imath}\)-module isomorphism.

\begin{defi}[Successor homomorphisms]
  Let \(\lambda \in X^+_{2n,0}\).
  Set \(l := \ell(\mathsf{par}(\lambda))\).
  When \(l < n\), the \emph{successor homomorphism} is the \(\mathbf{U}^\imath\)-module homomorphism defined to be the composite
  \begin{align*}
    \mathsf{suc} \colon V_q(\lambda)
    &\xrightarrow{{}^{\gamma_l}\bigvee{}^{\lambda-\gamma_l}} V_q(\gamma_l) \otimes V_q(\lambda-\gamma_l) \\
    &\xrightarrow{\cdot \downarrow^{D_n}_{D\mathrm{II}_n} \otimes 1} (V^\imath(\gamma_l) \oplus V^\imath(\gamma_{l-1})) \otimes V_q(\lambda-\gamma_l) \\
    &\hookrightarrow (V_q(\gamma_l) \oplus V_q(\gamma_{l-1})) \otimes V_q(\lambda-\gamma_l) \\
    &\xrightarrow{*} \bigoplus_{\mu \in X^+_{2n,0}} V_q(\mu).
  \end{align*}
  When \(l = n\), the \emph{successor homomorphism} is defined similarly by replacing \(\gamma_l\) with \(\gamma_{n,\mathsf{e}(\lambda)}\).
\end{defi}

\begin{prop}[Successor homomorphism at \(q = \infty\)]\label{prop_suc_hom_infty_DII}
  For each \(T \in \mathsf{KNT}^{D_n}_{0}\), we have
  \[
    \operatorname{\mathsf{suc}}(b_T) \equiv_\infty b_{\operatorname{\mathsf{suc}}(T)}.
  \]
\end{prop}
\begin{proof}
  The assertion is clear from definitions.
\end{proof}

\begin{defi}[Reduction homomorphisms from \(D_n\) to \(D\mathrm{II}_n\)]
  Let \(\lambda \in X^+_{2n,0}\), and take \(k \in \mathbb{Z}_{\geq 0}\) such that \(\mathsf{suc}^k(T) \in \mathsf{KNT}^{D\mathrm{II}_n}_{0}\) for all \(T \in \mathsf{KNT}^{D_n}(\lambda)\); see Proposition \ref{prop_suc_map_DII}.
  The \emph{reduction homomorphism} is the \(\mathbf{U}^\imath\)-module homomorphism \(\cdot \downarrow^{D_n}_{D\mathrm{II}_n}\) defined to be the composite
  \begin{align*}
    \cdot \downarrow^{D_n}_{D\mathrm{II}_n} \colon V_q(\lambda)
    &\xrightarrow{\mathsf{suc}^k} \bigoplus_{\nu \in X^+_{2n-1,0}} V_{2n,q}(\nu) \\
    &\xrightarrow{\bigoplus_{\nu} p_\nu} \bigoplus_{\nu \in X^+_{2n-1,0}} V^\imath(\nu) \\
    &\xrightarrow{\text{projection}} \bigoplus_{\substack{\nu \in X^+_{2n-1,0} \\ \mathsf{par}(\nu) \overset{\textsf{hor}}{\subseteq} \mathsf{par}(\lambda)}} V^\imath(\nu).
  \end{align*}
\end{defi}

\begin{prop}[Reduction homomorphisms at \(q = \infty\)]\label{prop_red_hom_D_DII_infty}
  Let \(\lambda \in X^+_{2n,0}\).
  Then, we have
  \[
    b_T \downarrow^{D_n}_{D\mathrm{II}_n} \equiv_\infty b_{T \downarrow^{D_n}_{D\mathrm{II}_n}} \quad \text{ for all } T \in \mathsf{KNT}^{D_n}(\lambda).
  \]
\end{prop}
\begin{proof}
  The assertion follows from Proposition \ref{prop_suc_hom_infty_DII} and Theorem \ref{thm_alm_orthn_basis_Vinu_DII}.
\end{proof}

\begin{thm}[Reduction homomorphisms from \(D_n\) to \(D\mathrm{II}_n\)]\label{thm_red_hom_D_DII}
  Let \(\lambda \in X^+_{2n,0}\).
  Then, the reduction homomorphism \(\cdot \downarrow^{D_n}_{D\mathrm{II}_n}\) is an isomorphism of \(\mathbf{U}^\imath\)-modules.
\end{thm}
\begin{proof}
  By Proposition \ref{prop_red_hom_D_DII_infty}, it suffices to prove Theorem \ref{thm_red_KNT_D_DII}, i.e., the map
  \[
    \cdot \downarrow^{D_n}_{D\mathrm{II}_n} \colon \mathsf{KNT}^{D_n}(\lambda) \to \bigsqcup_{\substack{\nu \in X^+_{2n-1,0} \\ \mathsf{par}(\nu) \overset{\textsf{hor}}{\subseteq} \mathsf{par}(\lambda)}} \mathsf{KNT}^{D\mathrm{II}_n}(\nu)
  \]
  is bijective.
  By the branching rule in Theorem \ref{thm_branching_rule_DII} \eqref{item_D_to_DII_thm_branching_rule_DII} and the bijection in Theorem \ref{thm_red_map_DII_D}, the domain and codomain of this map are finite sets of the same cardinality.
  Hence, we only need to show that it is injective.

  To this end, let us introduce some auxiliary maps.
  First, define a map
  \[
    \widetilde{\mathsf{suc}} \colon \mathsf{KNT}^{D_n}(\lambda) \to \bigsqcup_{\mu \in X^+_{2n,0}} \mathsf{KNT}^{D_n}(\mu) \times \mathsf{OT}(\mu)
  \]
  by
  \[
    \widetilde{\mathsf{suc}}(T) := \mathsf{LRS}(\mathbf{a}^1 \downarrow^{D_n}_{D\mathrm{II}_n} * \mathbf{a}^2 * \cdots * \mathbf{a}^m) \quad \text{ for all } T = (\mathbf{a}^1,\dots,\mathbf{a}^m) \in \mathsf{KNT}^{D_n}(\lambda).
  \]
  Clearly, it is injective.
  Also, for each \(T \in \mathsf{KNT}^{D_n}(\lambda)\), we have
  \[
    \widetilde{\mathsf{suc}}(T) = (\mathsf{suc}(T), \mathsf{Q}^{\mathsf{suc}}(T)) \quad \text{ for some } \mathsf{Q}^{\mathsf{suc}}(T) \in \mathsf{OT},
  \]
  where
  \[
    \mathsf{OT} := \bigsqcup_{\mu \in X^+_{2n,0}} \mathsf{OT}(\mu).
  \]

  Next, for each \(k \in \mathbb{Z}_{\geq 2}\), set
  \[
    \widetilde{\mathsf{suc}}^k := (\widetilde{\mathsf{suc}} \otimes 1^{k-1}) \circ \widetilde{\mathsf{suc}} \colon \mathsf{KNT}^{D_n}(\lambda) \to \bigsqcup_{\mu \in X^+_{2n,0}} \mathsf{KNT}^{D_n}(\mu) \times \mathsf{OT}^k.
  \]
  Namely, for each \(T \in \mathsf{KNT}^{D_n}(\lambda)\), we have
  \[
    \widetilde{\mathsf{suc}}^k(T) = (\mathsf{suc}^k(T), \mathbf{Q}^{\mathsf{suc}}(T)),
  \]
  where
  \[
    \mathbf{Q}^{\mathsf{suc}}(T) := (\mathsf{Q}^{\mathsf{suc}}(\mathsf{suc}^{k-1}(T)), \dots, \mathsf{Q}^{\mathsf{suc}}(\mathsf{suc}(T)), \mathsf{Q}^{\mathsf{suc}}(T)).
  \]
  Since \(\widetilde{\mathsf{suc}}\) is injective, so is \(\widetilde{\mathsf{suc}}^k\).

  The maps \(\widetilde{\mathsf{suc}}^k\) can be lifted to \(\mathbf{U}^\imath\)-module homomorphisms
  \[
    \widetilde{\mathsf{suc}}^k \colon V_q(\lambda) \to \bigoplus_{\mu \in X^+_{2n,0}} V_q(\mu) \otimes \mathsf{Span}_{\mathbb{Q}(q)} \mathsf{OT}^k
  \]
  such that
  \[
    \widetilde{\mathsf{suc}}^k(b_T) \equiv_\infty b_{\mathsf{suc}^k(T)} \otimes \mathbf{Q}^{\mathsf{suc}}(T) \quad \text{ for all } T \in \mathsf{KNT}^{D_n}(\lambda).
  \]

  Finally, define a \(\mathbf{U}^\imath\)-module homomorphism
  \[
    \cdot \widetilde{\downarrow^{D_n}_{D\mathrm{II}_n}} \colon V_q(\lambda) \to \bigoplus_{\substack{\nu \in X^+_{2n-1,0} \\ \mathsf{par}(\nu) \overset{\textsf{hor}}{\subseteq} \mathsf{par}(\lambda)}} V^\imath(\nu) \otimes \mathsf{Span}_{\mathbb{Q}(q)} \mathsf{OT}^k
  \]
  in a similar way to \(\cdot \downarrow^{D_n}_{D\mathrm{II}_n}\); use \(\widetilde{\mathsf{suc}}^k\) instead of \(\mathsf{suc}^k\).
  For each \(T \in \mathsf{KNT}^{D_n}(\lambda)\), we have
  \[
    b_T \widetilde{\downarrow^{D_n}_{D\mathrm{II}_n}} \equiv_\infty b_{T \downarrow^{D_n}_{D\mathrm{II}_n}} \otimes \mathbf{Q}^{\mathsf{suc}}(T).
  \]

  Now, we are ready to prove the injectivity of \(\cdot \downarrow^{D_n}_{D\mathrm{II}_n}\).
  Assume contrary, and take two distinct tableaux \(T,S \in \mathsf{KNT}^{D_n}(\lambda)\) such that
  \[
    T \downarrow^{D_n}_{D\mathrm{II}_n} = S \downarrow^{D_n}_{D\mathrm{II}_n}.
  \]
  Set \(\xi := \mathsf{sh}(T \downarrow^{D_n}_{D\mathrm{II}_n})\).
  Since the map \(\widetilde{\mathsf{suc}}^k\) is injective and we have
  \[
    \widetilde{\mathsf{suc}}^k(T) = (T \downarrow^{D_n}_{D\mathrm{II}_n}, \mathbf{Q}^{\mathsf{suc}}(T)), \quad \widetilde{\mathsf{suc}}^k(S) = (S \downarrow^{D_n}_{D\mathrm{II}_n}, \mathbf{Q}^{\mathsf{suc}}(S)),
  \]
  we obtain
  \[
    \mathbf{Q}^{\mathsf{suc}}(T) \neq \mathbf{Q}^{\mathsf{suc}}(S).
  \]
  This implies that the multiplicity of \(V^\imath(\xi)\) in the image of \(\cdot \widetilde{\downarrow^{D_n}_{D\mathrm{II}_n}}\) is more than one.
  However, this contradicts the branching rule in Theorem \ref{thm_branching_rule_DII} \eqref{item_D_to_DII_thm_branching_rule_DII}.
  Thus, we complete the proof.
\end{proof}

\section{Brief discussion on type \(B\)}\label{sect_B}
In this section, we briefly describe type \(B\) analogues of our results in \S\S\ref{sect_manipulation_D}--\ref{sect_qsp_DII}.

\subsection{Kashiwara--Nakashima tableaux of type \(B\)}\label{ssect_knt_B}
Let \(n \in \mathbb{Z}_{\geq 1}\).
For each \(a \in \mathcal{A}^{B_n}\), set
\[
  |a| := \begin{cases}
    a & \text{ if } a \preceq 0, \\
    b & \text{ if } a = \overline{b} \text{ for some } b \preceq n.
  \end{cases}
\]

The notions of \emph{vacant numbers} and \emph{occupied numbers} are defined in the same way as type \(D\) (Definition \ref{def_vac_occ}).
Note that the number \(0\) is always neither occupied nor vacant.

\begin{defi}[Quasi-increasing words]
  A word \(\mathbf{a} = (a_l,\dots,a_1) \in \mathcal{W}^{B_n}\) is said to be \emph{quasi-increasing} (from right to left) if either \(a_r \prec a_{r+1}\) or \((a_r,a_{r+1}) = (0,0)\) for all \(r \in \{ 1,\dots,l-1 \}\).
\end{defi}

Given a word \(\mathbf{a} = (a_l,\dots,a_1) \in \mathcal{W}^{B_n}\), let \(\mathbf{a}_0\) denote the subword of \(\mathbf{a}\) consisting of \(0\).

The notions of \emph{deep numbers} and \emph{shallow numbers} are defined in the same way as type \(D\) (Definition \ref{def_deep_shallow}).
Note that the number \(0\) is always shallow.

\begin{defi}[Kashiwara--Nakashima columns {\cite[(5.3.2)]{KaNa94}}]
  A word \(\mathbf{a} \in \mathcal{W}^{B_n}\) is said to be a \emph{Kashiwara--Nakashima column} if the following conditions are satisfied:
  \begin{description}
    \item[Length] \(|\mathbf{a}| \leq n\),
    \item[Quasi-increasing] \(\mathbf{a}\) is quasi-increasing,
    \item[Shallowness] every number is shallow in \(\mathbf{a}\), i.e.,
    \[
      |\mathbf{a}_{\leq b}| \leq b \quad \text{ for all } b \in \mathsf{O}(\mathbf{a}).
    \]
  \end{description}
  Let
  \[
    \mathsf{Col}^{B_n}_{l}
  \]
  denote the set of Kashiwara--Nakashima columns of type \(B_n\) and length \(l\), and set
  \[
    \mathsf{Col}^{B_n} := \bigsqcup_{l=0}^n \mathsf{Col}^{B_n}_{l}.
  \]
\end{defi}

\begin{defi}[Spin columns {\cite[Equation (5.4.6)]{KaNa94}}]
  A column \(\mathbf{a} = (a_n,\dots,a_1) \in \mathsf{Col}^{B_n}_{n}\) of length \(n\) is said to be a \emph{spin column} if
  \[
    \{ |\mathbf{a}_n|,\dots,|\mathbf{a}_1| \} = \{ 1,\dots,n \}.
  \]
  Let
  \[
    \mathsf{Col}^{B_n}_{\mathsf{spin}}
  \]
  denote the set of spin columns.
\end{defi}

\begin{defi}[Adjacent condition {\cite[\S\S 5.5--5.6]{KaNa94}}]
  A pair of two Kashiwara--Nakashima columns \(\mathbf{a} = (a_k,\dots,a_1)\) and \(\mathbf{b} = (b_l,\dots,b_1)\) are said to satisfy the \emph{adjacent condition} if the following hold:
  \begin{description}
    \item[Length] \(|\mathbf{a}| \geq |\mathbf{b}| > 0\),
    \item[Weakly Increasing] \(a_r \preceq b_r\) for all \(1 \leq r \leq |\mathbf{b}|\),
    \item[\((a,b)\)-Configurations] \((s-r) + (u-t) < b-a\) for all \(1 \preceq a \preceq b \prec n\) and \(1 \leq r \leq s < t \leq u \leq |\mathbf{b}|\) satisfying one of the following:
    \begin{itemize}
      \item \(a_r = a\), \(b_s = b\), \(b_t = \overline{b}\), and \(b_u = \overline{a}\),
      \item \(a_r = a\), \(a_s = b\), \(a_t = \overline{b}\), and \(b_u = \overline{a}\),
    \end{itemize}
    \item[\((a,n)\)-Configurations] \(u-r-1 < n-a\) for all \(1 \preceq a \prec n\) and \(1 \leq r < s < u \leq |\mathbf{b}|\) satisfying one of the following:
    \begin{itemize}
      \item \(a_r = a\), \(b_s = n\), \(b_{s+1} = 0\), and \(b_u = \overline{a}\),
      \item \(a_r = a\), \(b_s = n\), \(b_{s+1} = \overline{n}\), and \(b_u = \overline{a}\),
      \item \(a_r = a\), \(b_s = 0\), \(b_{s+1} = 0\), and \(b_u = \overline{a}\),
      \item \(a_r = a\), \(b_s = 0\), \(b_{s+1} = \overline{n}\), and \(b_u = \overline{a}\),
      \item \(a_r = a\), \(a_s = n\), \(a_{s+1} = 0\), and \(b_u = \overline{a}\),
      \item \(a_r = a\), \(a_s = n\), \(a_{s+1} = \overline{n}\), and \(b_u = \overline{a}\),
      \item \(a_r = a\), \(a_s = 0\), \(a_{s+1} = 0\), and \(b_u = \overline{a}\),
      \item \(a_r = a\), \(a_s = 0\), \(a_{s+1} = \overline{n}\), and \(b_u = \overline{a}\),
    \end{itemize}
    \item[\((n,n)\)-Configurations] there are no \(1 \leq s < t \leq |\mathbf{b}|\) satisfying the one of the following:
    \begin{itemize}
      \item \(a_s = n\) and \(b_t = 0\),
      \item \(a_s = n\) and \(b_t = \overline{n}\),
      \item \(a_s = 0\) and \(b_t = 0\),
      \item \(a_s = 0\) and \(b_t = \overline{n}\),
    \end{itemize}
  \end{description}
  We write \(\mathbf{a} \triangleleft \mathbf{b}\) to mean that \((\mathbf{a}, \mathbf{b})\) satisfy the adjacent condition.
\end{defi}

\begin{defi}[Kashiwara--Nakashima tableaux {\cite[\S 5.7]{KaNa94}}]
  Let \(\lambda \in X^+_{2n+1}\).
  A \emph{Kashiwara--Nakashima tableau} of type \(B_n\) and shape \(\lambda\) is a finite sequence \(T = (\mathbf{a}^1,\dots,\mathbf{a}^m)\) of Kashiwara--Nakashima columns satisfying the following:
  \begin{itemize}
    \item \(\mathbf{a}^2,\dots,\mathbf{a}^m \in \mathsf{Col}^{B_n}\),
    \item \(\mathbf{a}^1 \in \mathsf{Col}^{B_n}_{\mathsf{spin}}\) if \(\mathsf{s}(\lambda) = \frac{1}{2}\),
    \item \(\mathbf{a}^1 \triangleleft \cdots \triangleleft \mathbf{a}^m\).
    \item \((|\mathbf{a}^1|,\dots,|\mathbf{a}^m|) = \mathsf{par}(\lambda)'\),
  \end{itemize}
  see \S \ref{ssect_part} for notations concerning partitions.
  Let
  \[
    \mathsf{KNT}^{B_n}(\lambda)
  \]
  denote the set of Kashiwara--Nakashima tableaux of type \(B_n\) and shape \(\lambda\), and set
  \[
    \mathsf{KNT}^{B_n}_0 := \bigsqcup_{\lambda \in X^+_{2n+1,0}} \mathsf{KNT}^{B_n}(\lambda).
  \]
\end{defi}

\subsection{Kashiwara--Nakashima columns of type \(B\mathrm{II}\)}
\begin{defi}[Kashiwara--Nakashima columns of type \(B\mathrm{II}\)]
  A column \(\mathbf{a} \in \mathsf{Col}^{B_n}\) is said to be of \emph{type \(B\mathrm{II}_n\)} if it satisfies the following:
  \begin{description}
    \item[Exclusion of \(1\)] \(1 \notin \mathbf{a}\),
    \item[Strict Shallowness] \(|(\mathbf{a}_{> 1})_{\leq b}| < b\) for all \(b \in \mathsf{O}(\mathbf{a})\).
  \end{description}
  The set of Kashiwara--Nakashima columns of type \(B\mathrm{II}_n\) is denoted by
  \[
    \mathsf{Col}^{B\mathrm{II}_n}.
  \]
  For each \(l \in \mathbb{Z}_{\geq 0}\), set
  \[
    \mathsf{Col}^{B\mathrm{II}_n}_{l} := \mathsf{Col}^{B\mathrm{II}_n} \cap \mathsf{Col}^{B_n}_{l}.
  \]
\end{defi}

\begin{defi}[Signs of columns of type \(B\mathrm{II}\)]
  Let \(\mathbf{a} \in \mathsf{Col}^{B\mathrm{II}_n}_n\).
  The \emph{sign} of \(\mathbf{a}\) is \(\mathsf{e}(\mathbf{a}) \in \{ +,- \}\) defined by
  \[
    \mathsf{e}(\mathbf{a}) :=
    \begin{cases}
      + & \text{ if } \overline{1} \notin \mathbf{a}, \\
      - & \text{ if } \overline{1} \in \mathbf{a}.
    \end{cases}
  \]
  Let
  \[
    \mathsf{Col}^{B\mathrm{II}_n}_{n,\pm}
  \]
  denote the set of Kashiwara--Nakashima columns of type \(B\mathrm{II}_n\), length \(n\), and sign \(\pm\).
\end{defi}

\begin{rem}[Signs of columns of type \(B\mathrm{II}\)]
  The definition of sign for columns of type \(B\mathrm{II}\) is artificial.
  From the perspective of quantum symmetric pairs, at least to the author, it seems to be impossible to distinguish two columns of the form \(\overline{1}*\mathbf{a}*\mathbf{b}\) and \(\mathbf{a}*0*\mathbf{b}\) such that \(\mathbf{a} = \mathbf{a}^{\succ 0}\), \(\mathbf{b} = \mathbf{b}^{\preceq 0}\).
\end{rem}

\begin{defi}[Connecting maps]
  For each \(l \in \{ 0,\dots,n-1 \}\), the connecting map \(\delta_l\) is the map
  \[
    \delta_l \colon \mathsf{Col}^{B\mathrm{II}_n}_{l} \to \mathsf{Col}^{B_n}_{l+1} \setminus \mathsf{Col}^{B\mathrm{II}_n}_{l+1}
  \]
  defined by
  \[
    \delta_l(\mathbf{a}) :=
    \begin{cases}
      \mathbf{a}*1 & \text{ if } \overline{1} \notin \mathbf{a}, \\
      \mathbf{a}_{> 1} \Leftarrow \mathsf{gpv}(\mathbf{a}) & \text{ if } \overline{1} \in \mathbf{a}.
    \end{cases}
  \]
\end{defi}

\begin{thm}[Bijectivity of connecting maps]
  The connecting maps are bijective.
\end{thm}

Given a word \(\mathbf{a} \in \mathcal{W}^{B_n}\) with \(0 \in \mathbf{a}\), let \(\mathbf{a} - 0\) denote the subword obtained by removing \emph{exactly one} \(0\).
For a column \(\mathbf{a} \in \mathsf{Col}^{B\mathrm{II}_n}\), the following are equivalent:
\begin{itemize}
  \item \(|\mathbf{a}_{> 1}| = n\),
  \item \(|\mathbf{a}| = n\) and \(\overline{1} \notin \mathbf{a}\).
\end{itemize}
If one of the conditions above is satisfied, then we have \(0 \in \mathbf{a}\).

\begin{defi}[Reduction of columns from \(B\mathrm{II}_n\) to \(B_{n-1}\)]
  Let \(\mathbf{a} \in \mathsf{Col}^{B\mathrm{II}_n}\).
  The \emph{reduction of \(\mathbf{a}\) from \(B\mathrm{II}_n\) to \(B_{n-1}\)} is the word \(\mathbf{a} \downarrow^{B\mathrm{II}_n}_{B_{n-1}} \in \mathcal{W}^{B_{n-1}}\) defined by
  \[
    \mathbf{a} \downarrow^{B\mathrm{II}_n}_{B_{n-1}} :=
    \begin{cases}
      \mathbf{a}_{> 1}[-1] & \text{ if } |\mathbf{a}_{> 1}| < n, \\
      \mathbf{a}_{> 1}[-1] - 0 & \text{ if } |\mathbf{a}_{> 1}| = n.
    \end{cases}
  \]
\end{defi}

\begin{thm}[Reduction of columns from \(B\mathrm{II}_n\) to \(B_{n-1}\)]
  \hfill
  \begin{enumerate}
    \item Let \(l \in \{ 0,\dots,n-1 \}\).
    Then, the assignment \(\mathbf{a} \mapsto \mathbf{a} \downarrow^{B\mathrm{II}_n}_{B_{n-1}}\) gives rise to a bijection
    \[
      \cdot \downarrow^{B\mathrm{II}_n}_{B_{n-1}} \colon \mathsf{Col}^{B\mathrm{II}_n}_{l} \to \mathsf{Col}^{B_{n-1}}_{l} \sqcup \mathsf{Col}^{B_{n-1}}_{l-1}.
    \]
    \item The assignment \(\mathbf{a} \mapsto \mathbf{a} \downarrow^{B\mathrm{II}_n}_{B_{n-1}}\) gives rise to a bijection
    \[
      \cdot \downarrow^{B\mathrm{II}_n}_{B_{n-1}} \colon \mathsf{Col}^{B\mathrm{II}_n}_{n,\pm} \to \mathsf{Col}^{B_{n-1}}_{n-1}.
    \]
  \end{enumerate}
\end{thm}

\begin{cor}[\(B_{n-1}\)-crytsal structure of \(\mathsf{Col}^{B\mathrm{II}_n}\)]
  Let \(l \in \{ 0,\dots,n-1 \}\).
  The subsets \(\mathsf{Col}^{B\mathrm{II}_n}_{l} \subseteq \mathsf{Col}^{B_n}_{l}\) and \(\mathsf{Col}^{B\mathrm{II}_n}_{n,\pm} \subseteq \mathsf{Col}^{B_n}_{n}\) forms a subcrystal of type \(B_{n-1}\).
  Moreover, the reduction maps \(\cdot \downarrow^{B\mathrm{II}_n}_{B_{n-1}}\) are isomorphisms of crystals.
\end{cor}

\begin{defi}[Reduction of columns from \(B_n\) to \(B\mathrm{II}_n\)]
  Let \(\mathbf{a} \in \mathsf{Col}^{B_n}_{l}\).
  The \emph{reduction of \(\mathbf{a}\) from \(B_n\) to \(B\mathrm{II}_n\)} is the word \(\mathbf{a} \downarrow^{B_n}_{B\mathrm{II}_n} \in \mathcal{W}^{B_n}\) defined by
  \[
    \mathbf{a} \downarrow^{B_n}_{B\mathrm{II}_n} :=
    \begin{cases}
      \mathbf{a} & \text{ if } \mathbf{a} \in \mathsf{Col}^{B\mathrm{II}_n}_{l}, \\
      \mathbf{a}_{> 1} & \text{ if } \mathbf{a} \notin \mathsf{Col}^{B\mathrm{II}_n}_{l} \text{ and } \overline{1} \in \mathbf{a}, \\
      \overline{1}*(\mathbf{a} \Rightarrow \mathsf{lwd}(\mathbf{a})) & \text{ if } \mathbf{a} \notin \mathsf{Col}^{B\mathrm{II}_n}_{l} \text{ and } \overline{1} \notin \mathbf{a}.
    \end{cases}
  \]
\end{defi}

\begin{thm}[Reduction of columns from \(B_n\) to \(B\mathrm{II}_n\)]
  Let \(l \in \{ 0,\dots,n \}\).
  Then, the assignment \(\mathbf{a} \mapsto \mathbf{a} \downarrow^{B_n}_{B\mathrm{II}_n}\) gives rise to a bijection
  \[
    \mathsf{Col}^{B_n}_{l} \to \mathsf{Col}^{B\mathrm{II}_n}_{l} \sqcup \mathsf{Col}^{B\mathrm{II}_n}_{l-1}.
  \]
\end{thm}

\subsection{Kashiwara--Nakashima tableaux of type \(B\mathrm{II}\)}
\begin{defi}[Kashiwara--Nakashima tableaux of type \(B\mathrm{II}\)]
  A Kashiwara--Nakashima tableau \(T = (\mathbf{a}^1,\dots,\mathbf{a}^m)\) of type \(B_n\) and shape \(\lambda \in X^+_{2n,0}\) is said to be of \emph{type \(B\mathrm{II}_n\)} if its first column \(\mathbf{a}^1\) is of type \(B\mathrm{II}_n\) and sign \(\mathsf{e}(\lambda)\).
  The set of Kashiwara--Nakashima tableaux of type \(B\mathrm{II}_n\) and shape \(\lambda\) is denoted by
  \[
    \mathsf{KNT}^{B\mathrm{II}_n}(\lambda).
  \]
  Also, set
  \[
    \mathsf{KNT}^{B\mathrm{II}_n}_0 := \bigsqcup_{\lambda \in X^+_{2n,0}} \mathsf{KNT}^{B_n}(\lambda).
  \]
\end{defi}

\begin{thm}[Reduction of tableaux from \(B\mathrm{II}_n\) to \(B_{n-1}\)]
  Let \(\lambda \in X^+_{2n, 0}\).
  The assignment \(T \mapsto T \downarrow^{B\mathrm{II}_n}_{B_{n-1}}\) gives rise to a bijection
  \[
    \cdot \downarrow^{B\mathrm{II}_n}_{B_{n-1}} \colon \mathsf{KNT}^{B\mathrm{II}_n}(\lambda) \to \bigsqcup_{\substack{\xi \in X^+_{2n-1,0} \\ \mathsf{par}(\xi) \overset{\textsf{hor}}{\subseteq} \mathsf{par}(\lambda)}} \mathsf{KNT}^{B_{n-1}}(\xi).
  \]
\end{thm}

\begin{cor}[\(B_{n-1}\)-crystal structure of \(\mathsf{KNT}^{B\mathrm{II}_n}(\lambda)\)]
  The subset \(\mathsf{KNT}^{B\mathrm{II}_n}(\lambda) \subseteq \mathsf{KNT}^{B_n}(\lambda)\) is a subcrystal of type \(B_{n-1}\).
  Moreover, the reduction map \(\cdot \downarrow^{B\mathrm{II}_n}_{B_{n-1}}\) is an isomorphism of crystals.
\end{cor}

The map \(\cdot \downarrow^{B_n}_{B\mathrm{II}_n}\) is defined in a similar way to type \(D\) (Definition \ref{def_red_tab_D_DII}).

\begin{thm}[Reduction of tableaux from \(B_n\) to \(B\mathrm{II}_n\)]
  Let \(\nu \in X^+_{2n+1, 0}\).
  Under Assumption \ref{assum}, the assignment \(T \mapsto T \downarrow^{B_n}_{B\mathrm{II}_n}\) gives rise to a bijection
  \[
    \cdot \downarrow^{B_n}_{B\mathrm{II}_n} \colon \mathsf{KNT}^{B_n}(\nu) \to \bigsqcup_{\substack{\lambda \in X^+_{2n,0} \\ \mathsf{par}(\lambda) \overset{\textsf{hor}}{\subseteq} \mathsf{par}(\nu)}} \mathsf{KNT}^{B\mathrm{II}_n}(\lambda).
  \]
\end{thm}

\begin{cor}[Bijections between tableaux and patterns]
  Under Assumption \textup{\ref{assum}}, the following hold\textup{:}
  \begin{enumerate}
    \item Let \(\nu \in X^+_{2n+1,0}\).
    The assignment
    \[
      T \mapsto (\mathsf{sh}(T), \mathsf{sh}(T \downarrow^{B_n}_{B\mathrm{II}_n}), \mathsf{sh}(T \downarrow^{B_n}_{B_{n-1}}),\dots,\mathsf{sh}(T \downarrow^{B_n}_{B\mathrm{II}_1}))
    \]
    gives rise to a bijection
    \[
      \mathsf{KNT}^{B_n}(\nu) \to \mathsf{GTP}_{2n+1}(\nu).
    \]
    \item Let \(\lambda \in X^+_{2n,0}\).
    The assignment
    \[
      T \mapsto (\mathsf{sh}(T), \mathsf{sh}(T \downarrow^{B\mathrm{II}_n}_{B_{n-1}}), \mathsf{sh}(T \downarrow^{B\mathrm{II}_n}_{B\mathrm{II}_{n-1}}),\dots,\mathsf{sh}(T \downarrow^{B\mathrm{II}_n}_{B\mathrm{II}_1}))
    \]
    gives rise to a bijection
    \[
      \mathsf{KNT}^{B\mathrm{II}_n}(\lambda) \to \mathsf{GTP}_{2n}(\lambda).
    \]
  \end{enumerate}
\end{cor}

\subsection{Quantum groups of type \(B\)}\label{ssect_qg_B}
Let \(n \in \mathbb{Z}_{\geq 1}\), and consider the special orthogonal Lie algebra \(\mathfrak{so}_{2n+1}\).
Recall the notations \(d_i\), \(\epsilon_i\), \(I\), \(h_i\), and \(\alpha_i\) from \S\S\ref{ssect_so} and \ref{ssect_crystal}.
Set
\begin{itemize}
  \item \(Y := \mathsf{Span}_{\mathbb{Z}} \{ h_i \mid i \in I \}\),
  \item \(a_{i,j} := \langle h_i, \alpha_j \rangle\) for each \(i,j \in I\),
  \item \(l_i :=
  \begin{cases}
    2 & \text{ if } 1 \leq i < n, \\
    1 & \text{ if } i = n.
  \end{cases}
  \)
\end{itemize}
The \emph{quantum group} \(\mathbf{U} = \mathbf{U}(\mathfrak{so}_{2n+1})\) of \(\mathfrak{so}_{2n+1}\) is the unital associative \(\mathbb{Q}(q)\)-algebra with generators
\[
  \{ E_i,F_i,K_h \mid i \in I,\ h \in Y \}
\]
subject to the following relations:
\begin{align*}
  &K_0 = 1, \\
  &K_h K_{h'} = K_{h+h'}, \\
  &K_h E_i = q^{\langle h, \alpha_i \rangle} E_i K_h, \\
  &K_h F_i = q^{\langle h, -\alpha_i \rangle} F_i K_h, \\
  &E_i F_j - F_j E_i = \delta_{i,j} \frac{K_i-K_i^{-1}}{q_i-q_i^{-1}}, \\
  &\sum_{r+s = 1-a_{i,j}} E_i^{(r)} E_j E_i^{(s)} = 0 \quad \text{ if } i \neq j, \\
  &\sum_{r+s = 1-a_{i,j}} F_i^{(r)} F_j F_i^{(s)} = 0 \quad \text{ if } i \neq j,
\end{align*}
where
\[
  q_i := q^{l_i},\ K_i := K_{l_i h_i},\ E_i^{(r)} := \frac{1}{[r]_i!}E_i^r,\ F_i^{(r)} := \frac{1}{[r]_i!}F_i^r, \ [r]_i! := \prod_{a=1}^r [a]_i,\ [r]_i := \frac{q_i^r-q_i^{-r}}{q_i-q_i^{-1}}.
\]

\subsection{Column modules}
Let \(V_\natural := \mathsf{Span}_{\mathbb{Q}(q)} \{ v_a \mid a \in \mathcal{A}^{B_n} \}\) denote the vector representation of \(\mathbf{U}\).
Let \(\check{R} \in \mathsf{Aut}_{\mathbf{U}}(V_\natural^{\otimes 2})\) denote the R-matrix:
\[
  \check{R}(v_a \otimes v_b) :=
  \begin{cases}
    q^2 v_a \otimes v_a & \text{ if } a \neq 0, \\
    v_b \otimes v_a & \text{ if } b \prec a \neq \overline{b}, \\
    v_b \otimes v_a + (q^2-q^{-2})v_a \otimes v_b & \text{ if } a \prec b \neq \overline{a}, \\
    q^{-2} v_b \otimes v_{\overline{b}} - (q^2-q^{-2}) \sum_{c \prec b} (-q^2)^{-d(c,b)} v_c \otimes v_{\overline{c}} & \text{ if } a = \overline{b} \succ 0, \\
    v_0 \otimes v_0 - (q+q^{-1})(q^2-q^{-2}) \sum_{c=1}^n (-q^2)^{-n+c-1} v_c \otimes v_{\overline{c}} & \text{ if } a = b = 0, \\
    q^{-2} v_{\overline{a}} \otimes v_a + (q^2-q^{-2}) v_a \otimes v_{\overline{a}} \\
    -(q^2-q^{-2}) \sum_{0 \neq c \prec \overline{a}} (-q^2)^{-d(c,\overline{a})} v_c \otimes v_{\overline{c}} \\
    + (q-q^{-1})(-q^2)^{-n+a} v_0 \otimes v_0 & \text{ if } b = \overline{a} \succ 0.
  \end{cases}
\]

\begin{defi}[Column modules]
  Let \(\mathfrak{A}\) denote the two-sided ideal of \(\mathsf{T}\) generated by \((\check{R}+q^{-2})(\mathsf{T}^2)\).
  Set
  \[
    \mathsf{C} := \mathsf{T}/\mathfrak{A}.
  \]
  Also, for each \(l \in \mathbb{Z}_{\geq 0}\), set
  \[
    \mathsf{C}_l := \mathsf{T}^l/(\mathfrak{A} \cap \mathsf{T}^l).
  \]
  We call it the \emph{column module of length \(l\)}.
\end{defi}

\begin{prop}[Column modules]
  Let \(l \in \{ 0,\dots,n \}\).
  \begin{enumerate}
    \item We have
    \[
      \mathsf{C}_l \simeq V_q(\gamma_l).
    \]
    \item The set \(\{ u_{\mathbf{a}} \mid \mathbf{a} \in \mathsf{Col}^{B_n}_{l} \}\) forms a basis of \(\mathsf{C}_l\) that induces a crystal base \((\mathcal{L}(\mathsf{C}_l), \mathcal{B}(\mathsf{C}_l))\).
    Moreover, the map
    \[
      \mathsf{Col}^{B_n}_{l} \to \mathcal{B}(\mathsf{C}_l);\ \mathbf{a} \mapsto u_{\mathbf{a}} + q^{-1} \mathcal{L}(\mathsf{C}_l)
    \]
    is an isomorphism of crystals.
    \item For each \(\mathbf{a} = (a_l,\dots,a_1) \in \mathcal{W}^{B_n}_l \setminus \mathsf{Col}^{B_n}_{l}\), we have
    \[
      u_{\mathbf{a}} \in q^{-1} \mathcal{L}(\mathsf{C}_l).
    \]
  \end{enumerate}
\end{prop}

\subsection{Symmetric pairs}
Let \(n \in \mathbb{Z}_{\geq 2}\) and set \(\mathfrak{k}\) to be the Lie subalgebra of \(\mathfrak{so}_{2n+1}\) generated by the following elements:
\begin{itemize}
  \item \(e_j,f_j,h_j\) for \(j \in I_\bullet := \{ 2,\dots,n \}\),
  \item \(b_1 := Z_{2,1} + Z_{1,2n}\).
\end{itemize}

There exists a Lie algebra homomorphism
\[
  \psi_{2n} \colon \mathfrak{so}_{2n} \to \mathfrak{so}_{2n+1}
\]
such that
\begin{align*}
  &\psi_{2n}(e_i) =
  \begin{cases}
    e_{i+1} & \text{ if } 1 \leq i < n-1, \\
    \frac{1}{2}(\sqrt{2}Z_{n,n+1} + Z_{n,1} + Z_{1,n+2}) & \text{ if } i = n-1, \\
    \frac{1}{2}(\sqrt{2}Z_{n,n+1} - Z_{n,1} - Z_{1,n+2}) & \text{ if } i = n,
  \end{cases} \\
  &\psi_{2n}(f_i) =
  \begin{cases}
    f_{i+1} & \text{ if } 1 \leq i < n-1, \\
    \frac{1}{2}(\sqrt{2}Z_{n+1,n} + Z_{1,n} + Z_{n+2,1}) & \text{ if } i = n-1, \\
    \frac{1}{2}(\sqrt{2}Z_{n+1,n} - Z_{1,n} - Z_{n+2,1}) & \text{ if } i = n,
  \end{cases} \\
  &\psi_{2n}(h_i) =
  \begin{cases}
    h_{i+1} & \text{ if } 1 \leq i < n-1, \\
    2Z_{n,n} - \sqrt{2}(Z_{1,n+1} + Z_{n+1,1}) & \text{ if } i = n-1, \\
    2Z_{n,n} + \sqrt{2}(Z_{1,n+1} + Z_{n+1,1}) & \text{ if } i = n-.
  \end{cases}
\end{align*}
This homomorphism restricts to an isomorphism \(\mathfrak{so}_{2n} \to \mathfrak{k}\).

When \(n \geq 2\), let \(\mathfrak{g}_\bullet\) denote the Lie subalgebra of \(\mathfrak{k}\) generated by \(\{ e_j,f_j,h_j \mid j \in I_\bullet \}\).
Clearly, there exists a Lie algebra homomorphism
\[
  \psi_{2n-1}' \colon \mathfrak{so}_{2n-1} \to \mathfrak{k}
\]
such that
\[
  \psi_{2n-1}'(x_i) = x_{i+1} \ \text{ for all } i \in \{ 1,\dots,n-1 \},\ x \in \{ e,f,h \}.
\]
Let \(\psi_{2n-1}\) denote the composition
\[
  \mathfrak{so}_{2n-1} \xrightarrow{\psi_{2n-1}'} \mathfrak{k} \xrightarrow{\psi_{2n}^{-1}} \mathfrak{so}_{2n}.
\]
In fact, it coincides with the homomorphism \(\phi_{2n-1}\) \eqref{eq_so_embd}. 

\begin{prop}[Equivalences of embeddings]
  \hfill
  \begin{enumerate}
    \item Let \(\nu \in X^+_{2n+1}\).
    The \(\mathfrak{so}_{2n}\)-modules \(V_{2n+1}(\nu)\) defined via \(\phi_{2n}\) and \(\psi_{2n}\) are isomorphic.
    \item Let \(\lambda \in X^+_{2n}\).
    The \(\mathfrak{so}_{2n-1}\)-modules \(V_{2n}(\lambda)\) defined via \(\phi_{2n-1}\) and \(\psi_{2n-1}\) are isomorphic.
  \end{enumerate}
\end{prop}
\begin{proof}
  \hfill
  \begin{enumerate}
    \item By \cite[Theorem 1.3]{Dyn52}, we only need to prove the assertion for \(\nu = \gamma_1\).
    This can be straightforwardly verified.
    \item The assertion is trivial.
  \end{enumerate}
  
\end{proof}

\begin{cor}[Branching rule for symmetric pairs of type \(B\mathrm{II}\)]\label{cor_branching_symm_pair_BII}
  \hfill
  \begin{enumerate}
    \item\label{item_B_BII_cor_branching_symm_pair_BII} Let \(\nu \in X^+_{2n+1}\).
    Then, we have
    \[
      V_{2n+1}(\nu)|_{\mathfrak{k}} \simeq \bigoplus_{\substack{\lambda \in X^+_{2n, \mathsf{s}(\nu)} \\ \mathsf{par}(\lambda) \overset{\textsf{hor}}{\subseteq} \mathsf{par}(\nu)}} V_{2n}(\lambda).
    \]
    Here, we regard the \(\mathfrak{so}_{2n}\)-module \(V_{2n}(\lambda)\) as a \(\mathfrak{k}\)-module via the isomorphism \(\psi_{2n}\).
    \item\label{item_BII_B_cor_branching_symm_pair_BII} Let \(\lambda \in X^+_{2n}\).
    Then, we have
    \[
      V_{2n}(\lambda)|_{\mathfrak{g}_\bullet} \simeq \bigoplus_{\substack{\xi \in X^+_{2n-1, \mathsf{s}(\nu)} \\ \mathsf{par}(\xi) \overset{\textsf{hor}}{\subseteq} \mathsf{par}(\lambda)}} V_{2n-1}(\xi).
    \]
    Here, we regard the \(\mathfrak{so}_{2n-1}\)-module \(V_{2n-1}(\xi)\) as a \(\mathfrak{g}_\bullet\)-module via the isomorphism \(\psi_{2n-1}\).
  \end{enumerate}
\end{cor}

\subsection{Quantum symmetric pairs}
Let \(n \in \mathbb{Z}_{\geq 1}\) and set \(\mathbf{U} := \mathbf{U}(\mathfrak{so}_{2n+1})\).

The \emph{\(\imath\)quantum group of type \(B\mathrm{II}_n\)} is the subalgebra \(\mathbf{U}^\imath\) of \(\mathbf{U}\) generated by the following elements:
\begin{itemize}
  \item \(E_j,F_j,K_j^{\pm 1}\) for \(j \in I_\bullet := \{ 2,\dots,n \}\),
  \item \(B_1 := F_1 + (-1)^{n-1} q^{2n-3} T_2 \cdots T_{n-1} T_n T_{n-1} \cdots T_2(E_1) K_1^{-1}\).
\end{itemize}
where \(T_i \in \mathsf{Aut}_{\textsf{\(\mathbb{Q}(q)\)-alg}}(\mathbf{U})\) denotes the Lusztig braid group action:
\begin{align*}
  &T_i(E_j) :=
  \begin{cases}
    -F_iK_i & \text{ if } a_{i,j}=2, \\
    E_j & \text{ if } a_{i,j} = 0, \\
    E_iE_j - q_i^{-1}E_jE_i & \text{ if } a_{i,j} = -1, \\
    E_i^{(2)} E_j - q_i^{-1} E_iE_jE_i + q_i^{-2} E_j E_i^{(2)} & \text{ if } a_{i,j} = -2,
  \end{cases} \\
  &T_i(F_j) :=
  \begin{cases}
    -K_i^{-1}E_i & \text{ if } a_{i,j} = 2,\\
    F_j & \text{ if } a_{i,j} = 0, \\
    F_jF_i - q_i F_iF_j & \text{ if } a_{i,j} = -1, \\
    F_j F_i^{(2)} - q_i F_i F_j F_i + q_i^2 F_i^{(2)} F_j & \text{ if } a_{i,j} = -2,
  \end{cases} \\
  &T_i(K_h) := K_{h-\langle h, \alpha_i \rangle h_i}.
\end{align*}

When \(n \geq 2\), we set \(\mathbf{U}_\bullet\) to be the subalgebra of \(\mathbf{U}^\imath\) generated by \(\{ E_j,F_j,K_j^{\pm 1} \mid j \in I_\bullet \}\).

\begin{prop}[Vector \(v^\imath_0\)]
  Set
  \[
    v^\imath_0 := v_1 + (-1)^n q^{-2n+1} v_{\overline{1}} \in V_\natural.
  \]
  Then, it spans a \(\mathbf{U}^\imath\)-submodule of \(V_\natural\) isomorphic to the trivial module.
\end{prop}

\begin{assum}[Simple \(\mathbf{U}^\imath\)-modules]\label{assum}
  For each \(\lambda \in X^+_{2n,0}\), there exists a unique \textup{(}up to isomorphism\textup{)} simple classical weight \(\mathbf{U}^\imath\)-module \(V^\imath(\lambda)\) which appears as a submodule of \(V_q(\nu)\) for some \(\nu\), and is isomorphic to \(V_{2n}(\lambda)\) at \(q = 1\).
\end{assum}

\begin{thm}[Branching rules for quantum symmetric pairs of type \(B\mathrm{II}\)]
  Under Assumption \textup{\ref{assum}}, the following hold:
  \begin{enumerate}
    \item Let \(\nu \in X^+_{2n+1, 0}\).
    Then, we have
    \[
      V_q(\nu)|_{\mathbf{U}^\imath} \simeq \bigoplus_{\substack{\lambda \in X^+_{2n, 0} \\ \mathsf{par}(\lambda) \overset{\textsf{hor}}{\subseteq} \mathsf{par}(\nu)}} V^\imath(\lambda).
    \]
    \item Assume further that \(n \geq 2\).
    Let \(\lambda \in X^+_{2n,0}\).
    Then, we have
    \[
      V^\imath(\lambda)|_{\mathbf{U}_\bullet} \simeq \bigoplus_{\substack{\xi \in X^+_{2n-1,0} \\ \mathsf{par}(\xi) \overset{\textsf{hor}}{\subseteq} \mathsf{par}(\lambda)}} V_{2n-1,q}(\xi).
    \]
  \end{enumerate}
\end{thm}
\begin{proof}
  \hfill
  \begin{enumerate}
    \item The assertion follows from \cite[Theorem 4.15]{Mee24} and Theorem \ref{cor_branching_symm_pair_BII} \eqref{item_B_BII_cor_branching_symm_pair_BII}.
    \item The assertion follows from the general theory of quantum groups and Theorem \ref{cor_branching_symm_pair_BII} \eqref{item_BII_B_cor_branching_symm_pair_BII}.
  \end{enumerate}
\end{proof}


\begin{thebibliography}{GeTs50b}
  \bibitem[BaWa18]{BaWa18} H. Bao and W. Wang, Canonical bases arising from quantum symmetric pairs, Invent. Math. 213 (2018), no. 3, 1099--1177.
  \bibitem[BuSc17]{BuSc17} D. W. Bump and A. Schilling, Crystal Bases. Representations and Combinatorics, World Scientific Publishing Co. Pte. Ltd., Hackensack, NJ, 2017. xii+279 pp.
  \bibitem[Dyn52]{Dyn52} E. B. Dynkin, Semisimple subalgebras of semisimple Lie algebras, (Russian) Mat. Sbornik N.S. 30/72 (1952), 349--462 (3 plates).
  \bibitem[Ful97]{Ful97} W. Fulton, Young Tableaux. With Applications to Representation Theory and Geometry, London Math. Soc. Stud. Texts, 35, Cambridge University Press, Cambridge, 1997. x+260 pp.
  \bibitem[GeTs50a]{GeTs50} I. M. Gel'fand and M. L. Tsetlin, Finite-dimensional representations of the group of unimodular matrices, (Russian) Doklady Akad. Nauk SSSR (N.S.) 71 (1950), 825--828.
  \bibitem[GeTs50b]{GeTs50b} I. M. Gel'fand and M. L. Tsetlin, Finite-dimensional representations of groups of orthogonal matrices, (Russian) Doklady Akad. Nauk SSSR (N.S.) 71 (1950), 1017--1020.
  \bibitem[Hay92]{Hay92} T. Hayashi, Quantum deformation of classical groups, Publ. Res. Inst. Math. Sci. 28 (1992), no. 1, 57--81.
  \bibitem[HoKa02]{HoKa02} J. Hong and S. Kang, Introduction to Quantum Groups and Crystal Bases, Grad. Stud. Math., 42, American Mathematical Society, Providence, RI, 2002. xviii+307 pp.
  \bibitem[KaNa94]{KaNa94} M. Kashiwara and T. Nakashima, Crystal graphs for representations of the $q$-analogue of classical Lie algebras, J. Algebra 165 (1994), no. 2, 295--345.
  \bibitem[Kol14]{Kol14} S. Kolb, Quantum symmetric Kac-Moody pairs, Adv. Math. 267 (2014), 395--469.
  \bibitem[KoSt25]{KoSt24} S. Kolb and J. Stephens, Representation theory of very non-standard quantum \(\mathfrak{so}(2N-1)\), J. Pure Appl. Algebra 229 (2025), no. 7, Paper No. 107990, 38 pp.
  \bibitem[Lec03]{Lec03} C. Lecouvey, Schensted-type correspondences and plactic monoids for types \(B_n\) and \(D_n\), J. Algebraic Combin. 18 (2003), no. 2, 99--133.
  \bibitem[Lus93]{Lus93} G. Lusztig, Introduction to Quantum Groups, Reprint of the 1994 edition, Modern Birkh\"{a}user Classics. Birkh\"{a}user/Springer, New York, 2010. xiv+346 pp.
  \bibitem[Mee26]{Mee24} S. Meereboer, Symmetries for spherical functions of type \(\chi\) for quantum symmetric pairs, Transform. Groups 31 (2026), no. 2, 1851--1894.
  \bibitem[Mol06]{Mol06b} A. I. Molev, Gelfand--Tsetlin bases for classical Lie algebras, Handbook of algebra. Vol. 4, 109--170., Handb. Algebr., 4, Elsevier/North-Holland, Amsterdam, 2006.
  \bibitem[Wan23]{Wan22} W. Wang, Quantum symmetric pairs, ICM---International Congress of Mathematicians. Vol. IV. Sections 5--8, 3080--3102, EMS Press, Berlin, [2023], \copyright2023.
  \bibitem[Wat21]{Wat21b} H. Watanabe, Classical weight modules over \(\imath\)quantum groups, J. Algebra 578 (2021), 241--302.
  \bibitem[Wat23]{Wat23b} H. Watanabe, Stability of \(\imath\)canonical bases of irreducible finite type of real rank one, Represent. Theory 27 (2023), 1--29.
  \bibitem[Wat25]{Wat23f} H. Watanabe, Symplectic tableaux and quantum symmetric pairs, J. Comb. Algebra (2025), published online first, DOI 10.4171/JCA/113.
\end{thebibliography}
\end{document}